\documentclass[onefignum,onetabnum]{siamart220329}

\usepackage{lipsum}
\usepackage{amsfonts}
\usepackage{amssymb}
\usepackage{graphicx}
\usepackage{epstopdf}
\usepackage{algpseudocode}
\usepackage{tensor}
\usepackage{verbatim}
\usepackage{mathtools}
\usepackage{cleveref}
\usepackage{framed}
\usepackage{tikz}
\usepackage{tikz-3dplot}
\usetikzlibrary{calc, shapes, angles, quotes, patterns, decorations.pathreplacing, cd}
\usepackage{subcaption}
\usepackage{bm}
\usepackage{esint}
\usepackage{dsfont}

\ifpdf
\DeclareGraphicsExtensions{.eps,.pdf,.png,.jpg}
\else
\DeclareGraphicsExtensions{.eps}
\fi

\newlength{\leftstackrelawd}
\newlength{\leftstackrelbwd}
\def\leftstackrel#1#2{\settowidth{\leftstackrelawd}%
	{${{}^{#1}}$}\settowidth{\leftstackrelbwd}{$#2$}%
	\addtolength{\leftstackrelawd}{-\leftstackrelbwd}%
	\leavevmode\ifthenelse{\lengthtest{\leftstackrelawd>0pt}}%
	{\kern-.5\leftstackrelawd}{}\mathrel{\mathop{#2}\limits^{#1}}}

\newcommand{\bdd}[1]{ \boldsymbol{#1} }
\newcommand{\unitvec}[1]{\mathbf{#1}}

\newcommand{\vertiii}[1]{{\left\vert\kern-0.25ex\left\vert\kern-0.25ex\left\vert #1 
		\right\vert\kern-0.25ex\right\vert\kern-0.25ex\right\vert}}

\newcommand{\inormsup}[4]{\tensor*[_{#1}]{ \left\|#2 \right\| }{_{#3}^{#4} } }
\newcommand{\inorm}[3]{\inormsup{#1}{#2}{#3}{}}

\newcommand{\reftri}[0]{T}
\newcommand{\reftet}[0]{K}

\newsiamremark{remark}{Remark}
\newsiamremark{hypothesis}{Hypothesis}
\crefname{hypothesis}{Hypothesis}{Hypotheses}
\newsiamthm{claim}{Claim}

\headers{Stable Liftings of Polynomial Traces on Tetrahedra}{C. Parker and E. S\"{u}li}

\title{Stable Liftings of Polynomial Traces on Tetrahedra \thanks{Submitted to the editors DATE. \funding{The first author acknowledges that this material is based upon work supported by the National Science Foundation under Award No.
			DMS-2201487. }}}

\author{Charles Parker\thanks{ Mathematical Institute, University of Oxford, Andrew Wiles Building, Woodstock Road, Oxford OX2 6GG, UK  (\email{charles.parker@maths.ox.ac.uk}, \email{suli@maths.ox.ac.uk})} \and Endre S\"{u}li\footnotemark[2] }

\usepackage{amsopn}
\DeclareMathOperator{\supp}{supp}
\DeclareMathOperator{\Tr}{Tr}
\DeclareMathOperator{\dist}{dist}
\DeclareMathOperator{\dive}{div}
\renewcommand{\d}[1]{\,\mathrm{d}{#1}}

\ifpdf
\hypersetup{
  pdftitle={Stable Liftings of Polynomial Traces on Tetrahedra},
  pdfauthor={C. Parker and E. S\"{u}li}
}
\fi

\begin{document}

\maketitle

\begin{abstract}
	On the reference tetrahedron $\reftet$, we construct, for each $k \in \mathbb{N}_0$, a right inverse for the trace operator $u \mapsto (u, \partial_{\unitvec{n}} u, \ldots,  \partial_{\unitvec{n}}^k u)|_{\partial \reftet}$. The operator is stable as a mapping from the trace space of $W^{s, p}(\reftet)$ to $W^{s, p}(\reftet)$ for all $p \in (1, \infty)$ and $s \in (k+1/p, \infty)$. Moreover, if the data is the trace of a polynomial of degree $N \in \mathbb{N}_0$, then the resulting lifting is a polynomial of degree $N$. One consequence of the analysis is a novel characterization for the range of the trace operator.
\end{abstract}



\section{Introduction}
\label{sec:intro}

The numerical analysis of high-order finite element and spectral element methods heavily rely on the existence of stable polynomial liftings -- bounded operators mapping suitable piecewise polynomials on the boundary of the element to polynomials defined over the entire element. A number of operators have been constructed on the reference triangle and square, beginning with the pioneering work of Babu\v{s}ka et al. \cite{BCMP91,BabSuri87}. Their lifting maps $H^{\frac{1}{2}}(\partial E)$ boundedly into $H^1(E)$, where $E$ is either a suitable reference triangle or square, with the additional property that if the datum is continuous and its restriction to each edge is a polynomial of degree $N \geq 0$, then the lifting is a polynomial of degree $N$. Other constructions for continuous piecewise polynomials on $\partial E$ are stable from a discrete trace space to $L^2(E)$ \cite{AinJia19}, from $L^2(\partial E)$ to $H^{\frac{1}{2}}(E)$ \cite{Ain09poly}, from $W^{1-\frac{1}{p}, p}(\partial E)$ to $W^{1, p}(E)$ for $1 < p < \infty$ \cite{Melenk05}, and from $W^{s-\frac{1}{p}, p}(\partial E)$ to $W^{s, p}(E)$ for $s \geq 1$ and $1 < p < \infty$ \cite{Parker23}. Liftings for other types of traces are also available; e.g. lifting the normal trace of $H(\mathrm{div}; E)$ \cite{Ain09poly}, lifting the trace and normal derivative simultaneously into $H^2(E)$ \cite{AinCP19Extension}, and lifting an arbitrary number of normal derivatives simultaneously into $W^{s, p}(E)$ \cite{Parker23}. 

Many of the above results have been extended to three space dimensions. Mu\~{n}oz-Sola \cite{Munoz97} generalized the construction of Babu\v{s}ka et al. \cite{BCMP91,BabSuri87} to the tetrahedron, while Belgacem \cite{Belgacem} gave a different construction for the cube using orthogonal polynomials. Commuting lifting operators for the spaces appearing in the de Rham complex on tetrahedra \cite{Dem08,Dem09,Dem12} and hexahedra \cite{Costabel08} have also been constructed. These operators, among others, have been used extensively in a priori error analysis \cite{AinCP19StokesII,BabSuri87,Ern22,GuoBab10,Melenk05,Munoz97}, a posteriori error analysis \cite{Braess09,Chaumont2022curlcurl,Chaumont23,Ern20}, the analysis of preconditioners \cite{AinJia19,AinJiang21,AinCP19Precon,AinCP19StokesIII,AinCP21LEP,BCMP91,SchMelPechZag08}, the analysis of sprectral element methods, particularly in weighted Sobolev spaces \cite{Bern95,Bern07,Bern10,Bern89}, and in the stability analysis of mixed finite element methods \cite{AinCP19StokesI,AznaranHuParker23,Dem07,Demkowicz03,Dem08book,Lederer18}. Nevertheless, two types of operators are notably missing from the currently available results in three dimensions: (i) lifting operators stable in $L^p$-based Sobelev spaces, crucial in the analysis of high-order finite element methods for nonlinear problems; and (ii) lifting operators for the simultaneous lifting of the trace and normal derivative (and higher-order normal derivatives) which appear in the analysis of fourth-order (and higher-order) problems and in the analysis of mixed finite element methods for problems in electromagnetism and incompressible flow.

We address both of the above problems; namely, for each $k \in \mathbb{N}_0$, we construct a right inverse for the trace operator $u \mapsto (u, \partial_{\unitvec{n}} u, \ldots,  \partial_{\unitvec{n}}^k u)|_{\partial \reftet}$ on the reference tetrahedron $\reftet$ that is stable from the trace of $W^{s, p}(\reftet)$ to $W^{s, p}(\reftet)$ for all $p \in (1, \infty)$ and $s \in (k+1/p, \infty)$. Additionally, if the data is the trace of a polynomial of degree $N \in \mathbb{N}_0$, then the resulting lifting is a polynomial of degree $N$. A precise statement appears at the end of \cref{sec:trace-review}, which also contains a characterization for the trace space that appears to be novel. These results generalize our construction on the reference triangle \cite{Parker23} to the reference tetrahedron and to Sobolev spaces with minimal regularity.

The remainder of the manuscript is organized as follows. In \cref{sec:construction}, we detail an explicit construction of the lifting operator in a sequence of four steps, each consisting of an intermediate single-face lifting operator. The remainder of the manuscript is devoted to the analysis of the intermediate single-face operators: \cref{sec:whole-space-cont,sec:whole-space-lp-cont} characterize the continuity of a related operator defined on all of $\mathbb{R}^3$, while \cref{sec:cont} concludes with the proofs of the continuity properties of the intermediate operators.

\section{The traces of $W^{s, q}(\reftet)$ functions and statement of main result}
\label{sec:trace-review}

	We begin by reviewing the regularity properties of the traces of a function $u$ defined on a tetrahedron. Here, we will work in the setting of Sobolev spaces defined on an open Lipschitz domain $\mathcal{O} \subseteq \mathbb{R}^d$. Let  $s = m + \sigma$ be a nonnegative real number with $m \in \mathbb{N}_0$ and $\sigma \in [0, 1)$. We denote by $W^{s, p}(\mathcal{O})$, $p \in [1, \infty)$, the standard fractional Sobolev(-Slobodeckij) space \cite{Adams03} equipped with norm defined by
	\begin{align*}
		\|v\|_{s, p, \mathcal{O}}^p := \sum_{n=1}^{m} |v|_{n, p, \mathcal{O}}^p + \begin{cases}
			\sum_{|\alpha| = m} |D^{\alpha} v|_{\sigma, p, \mathcal{O}}^p & \text{if } \sigma > 0, \\
			0 & \text{otherwise},
		\end{cases} 
	\end{align*} 
	where the integer-valued seminorms and fractional seminorms are given by
	\begin{align*}
		|v|_{n, p, \mathcal{O}}^p := \sum_{ |\alpha| = n } \int_{\mathcal{O}} |D^{\alpha} v(\bdd{x})|^p  \d{\bdd{x}} \quad \text{and} \quad |v|_{\sigma, p, \mathcal{O}}^p := \iint_{\mathcal{O} \times \mathcal{O}} \frac{ | v(\bdd{x}) -  v(\bdd{y})|^p }{ |\bdd{x}-\bdd{y}|^{\sigma p + d}  } \d{\bdd{x}} \d{\bdd{y}},
	\end{align*}
	with the usual modification for $p = \infty$. When $s = 0$, the Sobolev space $W^{0, p}(\mathcal{O})$ coincides with the standard Lebesque space $L^p(\mathcal{O})$, and we denote the norm by $\|\cdot\|_{p, \mathcal{O}}$. We also require fractional Sobolev spaces defined on domain boundaries. Given a  $C^{k, 1}$, $k \in \mathbb{N}_0$, $(d-1)$-dimensional manifold $\Gamma \subseteq \partial \mathcal{O}$, the surface gradient $D_{\Gamma}$ is well-defined a.e. on $\Gamma$, and we define $W^{s, p}(\Gamma)$, $0 \leq s \leq k+1$, analogously (see e.g. \cite[\S 2.5.2]{Necas11}) with the norm
	\begin{align*}
		\|v\|_{s, p, \Gamma}^p := \sum_{|\beta| \leq m} \int_{\Gamma} |D_{\Gamma}^{\beta} v(\bdd{x})|^p \d{\bdd{x}} + \sum_{|\beta| = m}  \iint_{\Gamma \times \Gamma} \frac{ | D_{\Gamma}^{\beta}v(\bdd{x}) -  D_{\Gamma}^{\beta}v(\bdd{y})|^p }{ |\bdd{x}-\bdd{y}|^{\sigma p + d-1}  } \d{\bdd{x}} \d{\bdd{y}},
	\end{align*}
	where the sums are over multi-indices $\beta \in \mathbb{N}_0^{d-1}$. The seminorms $|\cdot|_{s, p, \Gamma}$ are defined similarly. 
	
	\subsection{Elementary trace results}
	
	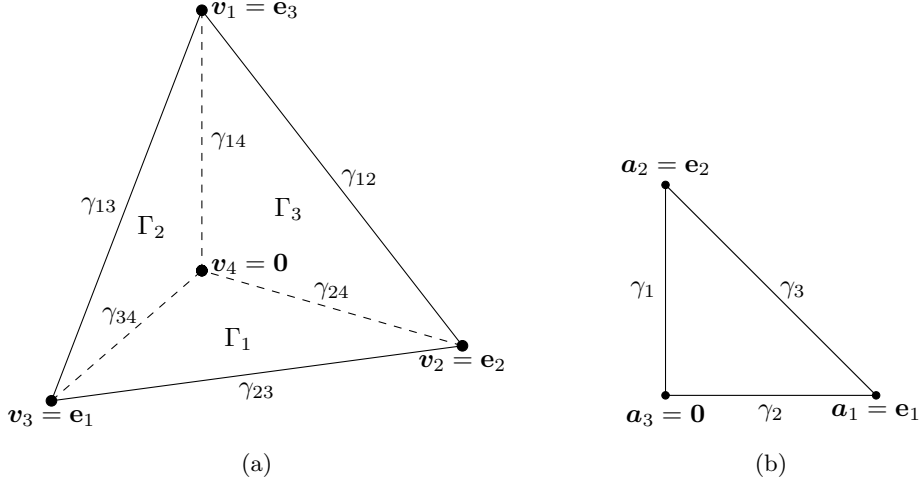
\begin{figure}[htb]
		\centering
		\begin{subfigure}[b]{0.6\linewidth}
			\centering
			\tdplotsetmaincoords{60}{120}
			\begin{tikzpicture}[tdplot_main_coords]
				\coordinate (v1) at (0, 0, 4);
				\coordinate (v2) at (0, 4, 0);
				\coordinate (v3) at (4, 0, 0);
				\coordinate (v4) at (0, 0, 0);
				\coordinate (e12) at ($(v1)!1/2!(v2)$);
				\coordinate (e13) at ($(v1)!1/2!(v3)$);
				\coordinate (e14) at ($(v1)!1/2!(v4)$);
				\coordinate (e23) at ($(v2)!1/2!(v3)$);
				\coordinate (e24) at ($(v2)!1/2!(v4)$);
				\coordinate (e34) at ($(v3)!1/2!(v4)$);
				
				\filldraw (v1) circle(2pt) node[align=center,right]{$\bdd{v}_1 = \unitvec{e}_3$};
				\filldraw (v2) circle(2pt) node[align=center,below]{$\bdd{v}_2 = \unitvec{e}_2$};
				\filldraw (v3) circle(2pt) node[align=center,below]{$\bdd{v}_3 = \unitvec{e}_1$};
				\filldraw (v4) circle(2pt)  circle(2pt) node[align=center,right]{};
				
				\draw ($(v4)+(0,0,0.1)$) node[align=center,right]{$\bdd{v}_4 = \unitvec{0}$};
				
				\draw (v1) -- (v2);
				\draw (v1) -- (v3);
				\draw[dashed] (v1) -- (v4);
				\draw (v2) -- (v3);
				\draw[dashed] (v2) -- (v4);
				\draw[dashed] (v3) -- (v4);
				
				\draw ($(e12)+(0,0,0)$) node[align=center,right]{${\gamma}_{12}$};
				\draw ($(e13)+(0,0,0)$) node[align=center,left]{${\gamma}_{13}$};
				\draw ($(e14)+(0,0,0)$) node[align=center,right]{${\gamma}_{14}$};
				\draw ($(e23)+(0,0,0)$) node[align=center,below]{${\gamma}_{23}$};
				\draw ($(e24)+(0,0,0)$) node[align=center,above]{${\gamma}_{24}$};
				\draw ($(e34)+(0,-0.1,0)$) node[align=center,above]{${\gamma}_{34}$};
				
				\draw (4/3,4/3,0) node[align=center]{$\Gamma_1$};
				\draw (4/3,0,4/3) node[align=center]{$\Gamma_2$};
				\draw (0,4/3,4/3) node[align=center]{$\Gamma_3$};
			\end{tikzpicture}
			\caption{}
			\label{fig:reference-tet}
		\end{subfigure}
		\hfill
		\begin{subfigure}[b]{0.35\linewidth}
			\centering
			\begin{tikzpicture}[scale=0.7]
				\filldraw (0,0) circle (2pt) 
				node[align=center,below]{}
				-- (4,0) circle (2pt)
				node[align=center,below]{}	
				-- (0,4) circle (2pt) 
				node[align=center,above]{}
				-- (0,0);
				
				\coordinate (a1) at (0,0);
				\coordinate (a2) at (4,0);
				\coordinate (a3) at (0,4);
				
				\draw (a1) node[align=center,below]{$\bdd{a}_3=\unitvec{0}$};
				\draw (a2) node[align=center,below]{$\bdd{a}_1 = \unitvec{e}_1$};
				\draw (a3) node[align=center,above]{$\bdd{a}_2 = \unitvec{e}_2$};
				
				\coordinate (e1) at ($(a2)!1/2!(a3)$);
				\coordinate (e2) at ($(a1)!1/2!(a3)$);
				\coordinate (e3) at ($(a1)!1/2!(a2)$);
				
				\coordinate (e113) at ($(a2)!1/3!(a3)$);
				\coordinate (e123) at ($(a2)!2/3!(a3)$);
				\coordinate (e1231) at ($(a2)!2/3+1/sqrt(32)!(a3)$);
				
				\coordinate (e213) at ($(a3)!1/3!(a1)$);
				\coordinate (e223) at ($(a3)!2/3!(a1)$);
				\coordinate (e2231) at ($(a3)!2/3+1/4!(a1)$);
				
				\coordinate (e313) at ($(a1)!1/3!(a2)$);
				\coordinate (e323) at ($(a1)!2/3!(a2)$);
				\coordinate (e3231) at ($(a1)!2/3+1/4!(a2)$);
				
				\draw ($(e1)+(0,0)$) node[align=center,right]{${\gamma}_3$};
				\draw ($(e2)+(0,0)$) node[align=center,left]{${\gamma}_1$};
				\draw ($(e3)+(0,0)$) node[align=center,below]{${\gamma}_2$};
			\end{tikzpicture}		
			\caption{}
			\label{fig:reference triangle}
		\end{subfigure}
		\caption{Reference (a) tetrahedron and (b) triangle, where $\unitvec{e}_i$ are the standard unit vectors. Note that the label for $\Gamma_4 = \{ (x, y, z) \in \bar{K} : x + y + z = 1 \}$ is omitted in (a). } 
	\end{figure}
	
	When the domain is the reference tetrahedron $\reftet := \{(x, y, z) \in \mathbb{R}^3 : 0 < x, y, z, x + y + z < 1\}$ depicted in \cref{fig:reference-tet}, the space $W^{r, p}(\partial \reftet)$, $0 \leq r < 1$, may be equipped with an equivalent norm that is more amenable to the analysis of traces. Let $\Gamma_i$ and $\Gamma_j$, $1 \leq i < j \leq 4$, be two faces of $\reftet$ and let $\gamma_{ij} = \gamma_{ji}$ denote the shared edge with vertices $\bdd{a}$ and $\bdd{b}$. Then, the vertices of $\Gamma_i$ are denoted by $\bdd{a}$, $\bdd{b}$, and $\bdd{c}_i$, while the vertices of $\Gamma_j$ are denoted by $\bdd{a}$, $\bdd{b}$, and $\bdd{c}_j$. Since $\Gamma_i$ and $\Gamma_j$ are both triangles, there exist unique affine mappings $\bdd{F}_{ij} : \reftri \to \Gamma_i$ and $\bdd{F}_{ji} : \reftri \to \Gamma_j$ from the reference triangle $\reftri := \{ (x, y) \in \mathbb{R}^2 : 0 < x, y, x+ y < 1 \}$, labeled as in \cref{fig:reference triangle}, satisfying
	\begin{subequations}
		\label{eq:faces-to-reftri-edge-mappings}
		\begin{alignat}{3}
			\bdd{F}_{ij}(0, 0) &= \bdd{a}, \qquad & \bdd{F}_{ij}(1, 0) = \bdd{b}, \quad \text{and} \quad & \bdd{F}_{ij}(0, 1) &= \bdd{c}_i, \\
			\bdd{F}_{ji}(0, 0) &= \bdd{a}, \qquad & \bdd{F}_{ji}(1, 0) = \bdd{b}, \quad \text{and} \quad & \bdd{F}_{ji}(0, 1) &= \bdd{c}_j,
		\end{alignat}
	\end{subequations}
	and we define the following norm:
	\begin{align*}
		\vertiii{f}_{r, p, \partial \reftet}^p := \sum_{i=1}^{4}  \|f_i\|_{r, p, \Gamma_i}^p + \begin{dcases}
			\sum_{1 \leq i < j \leq 4} \mathcal{I}_{ij}^p(f_{i}, f_{j}) & \text{if } rp = 1, \\
			0 & \text{otherwise},
		\end{dcases}
	\end{align*}
	where $f_i$ denotes the restriction of $f$ to $\Gamma_i$ and $\mathcal{I}_{i j}^p(f, g)$ is defined by the rule
	\begin{align}
		\label{eq:edge-integral-def}
		\mathcal{I}_{ij}^p(f, g) := \int_{\reftri}  |f \circ \bdd{F}_{ij}(\bdd{x}) - g \circ \bdd{F}_{ji}(\bdd{x})|^p  \frac{\d{\bdd{x}}}{x_2}.
	\end{align}
	Thanks to \cref{lem:equivalent-boundary-norm}, $\vertiii{\cdot}_{r, p, \partial \reftet}$ is an equivalent norm on $W^{r, p}(\partial \reftet)$; i.e. 
	\begin{align}
		\label{eq:equivalent-boundary-norm}
		\| f \|_{r, p, \partial \reftet} \approx_{r, p} \vertiii{f}_{r, p, \partial K} \qquad \forall f \in W^{r, p}(\partial \reftet),
	\end{align}
	and we shall use the two norms interchangeably with the common notation $\|\cdot\|_{r, p, \reftet}$. Here, and in what follows, we use the standard notation $a \lesssim_c b$ to indicate $a \leq C b$ where $C$ is a constant depending only on $c$, and $a \approx_c b$ if $a \lesssim_c b$ and $b \lesssim_c a$. 
		
	Now let $u \in W^{s, p}(\reftet)$, $1 < p < \infty$, $s = m + \sigma > 1/p$ with $m \in \mathbb{N}_0$ and $\sigma \in [0, 1)$, be a function defined on the reference tetrahedron. The presence of edges and corners on the boundary of $\reftet$ limit the regularity of the trace of $u$. Nevertheless, we can iteratively apply the standard $W^{s, p}(\reftet)$ trace theorem (e.g. \cite[Theorem 3.1]{Jerison95} or \cite[p. 208 Theorem 1]{Jonsson84}): $W^{s, p}(K)$ embeds continuously into $W^{s-\frac{1}{p}, p}(\partial K)$ for $1/p < s < 1 + 1/p$. In particular, for $k \in \mathbb{N}_0$, the $k$th-order derivative tensor given by
	\begin{align*}
		(D^k u)_{i_1 i_2 \ldots i_k} = \partial_{x_{i_1}} \partial_{x_{i_2}} \cdots \partial_{x_{i_k}} u
	\end{align*}
	satisfies $D^{k} u \in W^{s-k, p}(\reftet) \subset W^{1+\sigma, p}(\reftet)$, $0 \leq k \leq m-1$, and $D^m u \in W^{\sigma, p}(\reftet)$; thus, the traces satisfy
	\begin{align*}
		\begin{cases}
			D^k u|_{\partial \reftet} \in W^{1-\frac{1}{p},p}(\partial \reftet) & \text{for }  0 \leq k < s-\frac{1}{p}, \\
			D^{m-1} u|_{\partial \reftet} \in W^{1+\sigma-\frac{1}{p}, p}(\partial \reftet) & \text{if } m \geq 1 \text{ and }  \sigma p < 1, \\
			D^{m} u|_{\partial \reftet} \in W^{\sigma-\frac{1}{p}, p}(\partial \reftet) & \text{if } \sigma p > 1.
		\end{cases}
	\end{align*}
	Additionally, the trace of $W^{m+\frac{1}{2}, 2}(\mathbb{R}^3)$, $m \geq 1$, on the plane $\mathbb{R}^2 \times \{0\}$ belongs to $W^{m, 2}(\mathbb{R}^2)$ (see e.g. \cite[Chapter 7]{Adams03} or \cite[p. 20 Theorem 4]{Jonsson84}), and so standard arguments show that the trace of $W^{m+\frac{1}{2}, 2}(\reftet)$ on the face $\Gamma_i$, $1 \leq i \leq 4$, belongs to $W^{m, 2}(\Gamma_i)$. Thanks to the norm-equivalence \cref{eq:equivalent-boundary-norm}, we arrive at the following conditions:
	\begin{align}
		\label{eq:review:usual-trace-conditions}
		\left\{ 
		\begin{aligned}
			\sum_{i=1}^{4} \|D^k u \|_{p, \Gamma_i}^p &< \infty  \quad & & \text{for } 0 \leq k < s-\frac{1}{p}, \\
			\sum_{i=1}^{4}\|D^{m-1} u \|_{1 + \sigma - \frac{1}{p}, p, \Gamma_i}^p &< \infty \quad & &\text{if } m \geq 1 \text{ and either } \sigma p < 1 \text{ or } (\sigma, p) = \left(\frac{1}{2}, 2 \right), \\
			\sum_{i=1}^{4}\|D^{m} u \|_{\sigma-\frac{1}{p}, p, \Gamma_i}^p &< \infty \quad & & \text{if } \sigma p > 1, \\
			\sum_{1 \leq i < j \leq 4}\mathcal{I}_{ij}^p(D^{m} u, D^{m} u) &< \infty \quad & & \text{if } \sigma p = 2.
		\end{aligned}
		\right.
	\end{align} 
	\begin{remark}
		The case $\sigma p = 1$ for $p \neq 2$, which is not included in conditions \cref{eq:review:usual-trace-conditions}, is beyond the scope of this paper  since the trace of a $W^{m+\frac{1}{p}, p}(\mathbb{R}^3)$, $m \in \mathbb{N}$, function on the plane $\mathbb{R}^2 \times \{0\}$ belongs to a Besov space, which cannot be identified with an integer-order Sobolev space \cite[p. 20 Theorem 4]{Jonsson84}. 
	\end{remark}
	
	When $s > 2/p$, we obtain additional conditions since the trace of a $W^{s, p}(\reftet)$ function on the edge $\gamma_{ij}$, $1 \leq i < j \leq 4$ is well-defined. This can be seen from standard arguments owing to the fact that the trace of $W^{s, p}(\mathbb{R}^3)$ on the line $\mathbb{R} \times \{0\}^2$ is well-defined. In particular, the traces of the $k$-th derivative tensor, $0 \leq k < s-2/p$, on $\Gamma_i$ and $\Gamma_j$, $1 \leq i < j \leq 4$, must agree on the shared edge $\gamma_{ij}$:
	\begin{align}
		\label{eq:review:edge-continuity-conditions}
		D^k u|_{\Gamma_{i}}(\bdd{x}) &= D^k u|_{\Gamma_j}(\bdd{x}) \qquad \text{for a.e. $\bdd{x} \in \gamma_{ij}$ and all } 0 \leq k < s-\frac{2}{p},
	\end{align}
	where \cref{eq:review:edge-continuity-conditions} is to be interpreted in the trace sense.
	
	\subsection{Trace operators}
	
	We now turn to the consequences of \cref{eq:review:usual-trace-conditions,eq:review:edge-continuity-conditions} for various trace operators.
	
	\subsubsection{Zeroth-order operator}
	First consider the the $0$th-order ``boundary-derivative" operator $\mathfrak{D}_i^0$ on $\Gamma_i$, $1 \leq i \leq 4$, defined formally by the rule
	\begin{align}
		\label{eq:sigma0-def}
		\mathfrak{D}_i^0(f) := f \qquad \text{on } \Gamma_i.
	\end{align}
	Then, \cref{eq:review:usual-trace-conditions,eq:review:edge-continuity-conditions} show that for $u \in W^{s, p}(\reftet)$, $(s, q) \in \mathcal{A}_0$, where
	\begin{align}
		\label{eq:admissible-set-def}
		\mathcal{A}_k := \left\{ (s, p) \in \mathbb{R}^2 : 1 < p < \infty, \ (s-k)p > 1, \text{ and } s - \frac{1}{p} \notin \mathbb{Z} \text{ if } p \neq 2 \right\}, \quad k \in \mathbb{N}_0,
	\end{align}
	the trace $f = u|_{\partial \reftet}$ satisfies the following conditions:
	\begin{enumerate}
		\item $W^{s-\frac{1}{p}, p}$ regularity on each face:
		\begin{align}
			\label{eq:review:sigma0-face-reg}
			\mathfrak{D}_i^0(f) \in W^{s-\frac{1}{p}, p}(\Gamma_i), \qquad 1 \leq i \leq 4.
		\end{align}
		
		\item Compatible traces along edges: For $1 \leq i < j \leq 4$, there holds
		\begin{subequations}
			\label{eq:review:sigma0-cont-combo}
			\begin{alignat}{2}
				\label{eq:review:sigma0-cont-1}
				\mathfrak{D}_{i}^0(f)|_{\gamma_{ij}} - \mathfrak{D}_{j}^0(f)|_{\gamma_{ij}} &= 0 \qquad & &\text{if } sp > 2, \\
				\label{eq:review:sigma0-cont-2}
				\mathcal{I}_{ij}^{p}(\mathfrak{D}_{i}^0(f), \mathfrak{D}_{j}^0(f)) &< \infty \qquad & & \text{if } sp = 2.
			\end{alignat}
		\end{subequations}
	\end{enumerate}
	
	If $(s-n)p = 2$ for some $n \in \mathbb{N}$, then we obtain an additional condition since the $n$-th derivative tensor satisfies $\mathcal{I}_{ij}^p(D^n u, D^n u) < \infty$ for $1 \leq i < j \leq 4$. To describe this condition we define the following notation for a $d$-dimensional tensor $S$ and vector $v \in \mathbb{R}^3$:
	\begin{align*}
		v^{\otimes 0} \cdot S := S \quad \text{and} \quad v^{\otimes j} \cdot S := S_{i_1 i_2 \cdots i_d} v_{i_1} v_{i_2} \cdots v_{i_j}, \qquad 1 \leq j \leq d.
	\end{align*}
	In particular, for $1 \leq i < j \leq 4$, denoting by $\unitvec{t}_{ij}$ a unit vector tangent to $\gamma_{ij}$, we can differentiate $\mathfrak{D}_i^0(u)$ and $\mathfrak{D}_j^0(u)$ in the direction $\unitvec{t}_{ij}$ to obtain the following identity.
	\begin{align*}
  		\frac{\partial^n \mathfrak{D}_i^0(u)}{\partial \unitvec{t}_{ij}^n} = \unitvec{t}_{ij}^{\otimes n} \cdot D^n u|_{\Gamma_i} \quad \text{and} \quad \frac{\partial^n \mathfrak{D}_j^0(u)}{\partial \unitvec{t}_{ij}^n} = \unitvec{t}_{ij}^{\otimes n} \cdot D^n u|_{\Gamma_j}.
	\end{align*}
	Consequently, the trace $f = u|_{\partial \reftet}$ also satisfies the following property:
	\begin{enumerate}
		\item[3.] Compatible tangential derivatives: For $1 \leq i < j \leq 4$ and $n \in \mathbb{N}$, there holds
		\begin{align}
			\label{eq:review:sigma0-cont-dt}
			\mathcal{I}_{ij}^{p}\left( \frac{\partial^n \mathfrak{D}_i^0(u)}{\partial \unitvec{t}_{ij}^n}, \frac{\partial^n \mathfrak{D}_j^0(u)}{\partial \unitvec{t}_{ij}^n}\right) < \infty \qquad \text{if } (s-n)p = 2.
		\end{align}
	\end{enumerate}
	
	\subsubsection{First-order operator}
	For $(s, p) \in \mathcal{A}_1$, we turn to the regularity of the trace of the gradient of $u \in W^{s, p}(\reftet)$. To this end, on each face $\Gamma_i$, $1 \leq i \leq 4$, let $\{ \bdd{\tau}_{i, 1}, \bdd{\tau}_{i, 2} \}$ be orthonormal vectors tangent to $\Gamma_i$ and let $\unitvec{n}_i$ denote the outward unit normal vector on $\Gamma_i$. We define the 1st-order ``boundary-derivative" operator $\mathfrak{D}_i^1$ on $\Gamma_i$, $1 \leq i \leq 4$, by the rule
	\begin{align}
		\label{eq:sigma1-def}
		\mathfrak{D}_i^1(f, g) := \sum_{j=1}^{2} \frac{\partial f}{\partial \bdd{\tau}_{i, j}}  \bdd{\tau}_{i, j} + g \unitvec{n}_i  \qquad \text{on } \Gamma_i,
	\end{align} 
	so that $\mathfrak{D}_i^1(u, \partial_{\unitvec{n}} u) = Du|_{\Gamma_i}$. Again applying \cref{eq:review:usual-trace-conditions,eq:review:edge-continuity-conditions}, we obtain analogues of \cref{eq:review:sigma0-face-reg,eq:review:sigma0-cont-combo} stated below in \cref{eq:review:sigma1-face-reg,eq:review:sigma1-cont-combo} with $k=0$. However, if $(s-2)p > 2$, then the second derivative tensor has matching traces on edges (i.e. \cref{eq:review:edge-continuity-conditions} holds with $k=2$). In particular, for $1 \leq i < j \leq 4$, we define the vectors
	\begin{align}
		\label{eq:orthog-vectors-two-faces}
		\unitvec{b}_{ij} := \unitvec{t}_{ij} \times \unitvec{n}_{i} \quad \text{and} \quad \unitvec{b}_{ji}:= \unitvec{t}_{ij} \times \unitvec{n}_{j},
	\end{align}
	where we recall that $\unitvec{t}_{ij}$ is a unit vector tangent to $\gamma_{ij}$, so that on $\gamma_{ij}$, there holds
	\begin{align*}
		\unitvec{b}_{ji} \cdot \frac{\partial \mathfrak{D}^1_{i}(u, \partial_{\unitvec{n}} u)}{\partial \unitvec{b}_{ij}} = \unitvec{b}_{ji} \cdot \frac{\partial Du}{\partial \unitvec{b}_{ij}} = \frac{\partial^2 u}{\partial \unitvec{b}_{ij} \partial \unitvec{b}_{ji}} =  \unitvec{b}_{ij} \cdot \frac{\partial Du}{\partial \unitvec{b}_{ji}} = \unitvec{b}_{ij} \cdot \frac{\partial \mathfrak{D}^1_{j}(u, \partial_{\unitvec{n}} u)}{\partial \unitvec{b}_{ji}}
	\end{align*}
	in the sense of traces. As a consequence, the operator $\mathfrak{D}_i^1$ satisfies the additional condition \cref{eq:review:sigma1-deriv} below with $n=0$ thanks to \cref{eq:review:usual-trace-conditions,eq:review:edge-continuity-conditions}. Finally, we can differentiate in the direction tangent to each edge to obtain the analogue of \cref{eq:review:sigma0-cont-dt} stated in \cref{eq:review:sigma1-cont-int,eq:review:sigma1-deriv-cont-int} below. To summarize, the traces $f = u|_{\partial \reftet}$ and $g = \partial_{\unitvec{n}} u|_{\partial \reftet}$ satisfy the following for all $(s, p) \in \mathcal{A}_1$:
	\begin{enumerate}
		
		\item $W^{s-1-\frac{1}{p}, p}$ regularity on each face:
		\begin{align}
			\label{eq:review:sigma1-face-reg}
			\mathfrak{D}_i^1(f, g) \in W^{s-1-\frac{1}{p}, p}(\Gamma_i), \qquad 1 \leq i \leq 4.
		\end{align}
		
		\item Compatible traces along edges: For $1 \leq i < j \leq 4$ and $n \in \mathbb{N}_0$, there holds
		\begin{subequations}
			\label{eq:review:sigma1-cont-combo}
			\begin{alignat}{2}
				\label{eq:review:sigma1-cont}
					\mathfrak{D}_{i}^1(f, g)|_{\gamma_{ij}} - \mathfrak{D}_{j}^1(f, g)|_{\gamma_{ij}} &= 0 \qquad & &\text{if } (s-1)p > 2, \\
				\label{eq:review:sigma1-cont-int}
				\mathcal{I}_{ij}^{p}\left( \frac{\partial^n \mathfrak{D}_{i}^1(f, g)}{ \partial \unitvec{t}_{ij}^n }, \frac{\partial^n \mathfrak{D}_{j}^1(f, g)}{\partial \unitvec{t}_{ij}^n} \right) &< \infty \qquad & & \text{if } (s-n-1)p = 2.
			\end{alignat}
		\end{subequations}
		
		\item Compatible traces of higher derivatives along edges: For $1 \leq i < j \leq 4$ and $n \in \mathbb{N}_0$, there holds
		\begin{subequations}
			\label{eq:review:sigma1-deriv}
			\begin{alignat}{2}
				\label{eq:review:sigma1-deriv-cont}
				\left. \unitvec{b}_{ji} \cdot \frac{\partial \mathfrak{D}^1_{i}(f, g)}{\partial \unitvec{b}_{ij}} \right|_{\gamma_{ij}}  - \left. \unitvec{b}_{ij} \cdot \frac{\partial \mathfrak{D}^1_{j}(f, g)}{\partial \unitvec{b}_{ji}}\right|_{\gamma_{ij}} &= 0 \qquad & &\text{if } (s-2)p > 2, \\
				\label{eq:review:sigma1-deriv-cont-int}
				\mathcal{I}_{ij}^{p}\left(\unitvec{b}_{ji} \cdot \frac{\partial^{n+1} \mathfrak{D}^1_{i}(f, g)}{ \partial \unitvec{t}_{ij}^{n} \partial \unitvec{b}_{ij}}, \unitvec{b}_{ij} \cdot \frac{\partial^{n+1} \mathfrak{D}^1_{j}(f, g)}{ \partial \unitvec{t}_{ij}^n \partial \unitvec{b}_{ji}} \right) &< \infty \qquad & & \text{if } (s-n-2)p = 2.
			\end{alignat}
		\end{subequations}
	\end{enumerate}
	
	\begin{remark}
		\label{rem:review:sigma1-2d-relate}
		For smooth enough functions, conditions \cref{eq:review:sigma1-cont,eq:review:sigma1-deriv-cont} 
		may be interpreted as the application of the vertex compatibility conditions for traces on the triangle (see e.g. \cite[eqs. (2.11a) and (2.12a)]{Parker23}) at every point on the edge $\gamma_{ij}$. 
	\end{remark}

	\subsubsection{$m$th-order operator}
	We now turn to the general case of the trace of the $m$-th derivative tensor of a function $u \in W^{s, p}(\reftet)$, where $m \geq 2$ and $(s, p) \in \mathcal{A}_m$. Given a collection of functions $F = (f^0, f^1, \ldots, f^m)$ defined on $\partial \reftet$, we define the $m$-th order ``boundary-derivative" operator $\mathfrak{D}_i^m$ on $\Gamma_i$, $1 \leq i \leq 4$, by the rule
	\begin{align}
		\label{eq:review:sigmak-def}
		\mathfrak{D}_i^m(F) := \sum_{\substack{\alpha \in \mathbb{N}_0^3 \\ |\alpha| = m}} \frac{\partial^{\alpha_1+\alpha_2} f^{\alpha_3} }{\partial \bdd{\tau}_{i, 1}^{\alpha_1} \partial \bdd{\tau}_{i, 2}^{\alpha_2}} \sum_{ \phi \in \mathfrak{M}_i(\alpha)} \phi(1) \otimes \phi(2) \otimes \cdots \otimes \phi(m) \qquad \text{on } \Gamma_i,
	\end{align}
	where the set $\mathfrak{M}_{i}(\alpha)$ consists of the following mappings.
	\begin{align*}
		\mathfrak{M}_{i}(\alpha) := \{ \phi : \{1,2,\ldots, |\alpha|\} \to \{ \bdd{\tau}_{i, 1}, \bdd{\tau}_{i, 2}, \unitvec{n}_i \} \text{ s.t. } |\phi^{-1}(\bdd{\tau}_{i, j})| = \alpha_j, \ j=1,2 \},
	\end{align*}
	where we recall that $\{ \bdd{\tau}_{i, 1}, \bdd{\tau}_{i, 2} \}$ are orthonormal vectors tangent to $\Gamma_i$. For notational convenience, we set
	\begin{align*}
		\mathfrak{D}_i^l(F) := \mathfrak{D}_i^l(f^0, f^1, \ldots, f^{l}), \qquad 0 \leq l < m.
	\end{align*}
	Then, one may readily verify that $\mathfrak{D}^m_i(u, \partial_{\unitvec{n}} u, \ldots, \partial_{\unitvec{n}}^m u) = D^m u$ on $\Gamma_i$. Let $f^l = \partial_{\unitvec{n}}^l u$ on $\partial \reftet$. As before, we obtain $W^{s-m-\frac{1}{p}, p}$ regularity of $\mathfrak{D}_i^m(F)$ \cref{eq:review:sigmak-face-reg} below on each face thanks to \cref{eq:review:usual-trace-conditions} and the edge compatibility conditions \cref{eq:review:sigmak-cont-combo} below with $l = 0$ from \cref{eq:review:edge-continuity-conditions}. 
	
	As was the case with the first-order operator, there are additional edge compatibility conditions. In particular, if $(s-m-l)p > 0$ for some $1 \leq l \leq m$, then \cref{eq:review:edge-continuity-conditions} shows that the $(m+l)$th derivative tensor has matching traces on edges. Some components of the $(m+l)$th derivative tensor can be expressed in terms of $\mathfrak{D}_l^{m}(F)$. In particular, on the edge $\gamma_{ij}$, $1 \leq i < j \leq 4$, there holds
	\begin{multline*}
		\unitvec{b}_{ji}^{\otimes l} \cdot \frac{\partial^l \mathfrak{D}_i^{m}(F)}{\partial \unitvec{b}_{ij}^{l}} 
		=  \unitvec{b}_{ji}^{\otimes l} \cdot \frac{\partial^l D^m u}{\partial \unitvec{b}_{ij}^{l}} 
		= \unitvec{b}_{ji}^{\otimes l} \cdot \left( \unitvec{b}_{ij}^{\otimes l} \cdot D^{m+l} u \right) \\
		= \unitvec{b}_{ij}^{\otimes l} \cdot \left( \unitvec{b}_{ji}^{\otimes l} \cdot D^{m+l} u \right) 
		= \unitvec{b}_{ij}^{\otimes l} \cdot \frac{\partial^l D^m u}{\partial \unitvec{b}_{ji}^{l}} 
		= \unitvec{b}_{ij}^{\otimes l} \cdot \frac{\partial^l \mathfrak{D}_j^{m}(F)}{\partial \unitvec{b}_{ji}^{l}}
	\end{multline*}
	in the sense of traces, where we used symmetry of the derivative tensor $D^{m+l} u$. We can also differentiate in the direction tangent to each edge to obtain similar conditions. Consequently, $\mathfrak{D}_i^m(F)$ satisfies \cref{eq:review:sigmak-cont-combo} below. In summary, for $m \in \mathbb{N}_0$, the traces $F = (u, \partial_{\unitvec{n}} u, \ldots, \partial_{\unitvec{n}}^{m} u)$ satisfy the following for all $(s, p) \in \mathcal{A}_m$: 	
	\begin{enumerate}
		
		\item $W^{s-m-\frac{1}{p}, p}$ regularity on each face:
		\begin{align}
			\label{eq:review:sigmak-face-reg}
			\mathfrak{D}_i^m(F) \in W^{s-m-\frac{1}{p}, p}(\Gamma_i), \qquad 1 \leq i \leq 4,
		\end{align}
		where $\mathfrak{D}_i^0$, $\mathfrak{D}_i^1$, and $\mathfrak{D}_i^l$, $l \geq 2$, are defined in \cref{eq:sigma0-def}, \cref{eq:sigma1-def}, and \cref{eq:review:sigmak-def}.
		
		\item Compatible traces along edges: For $1 \leq i < j \leq 4$ and $0 \leq l \leq m$ and $n \in \mathbb{N}_0$, there holds
		\begin{subequations}
			\label{eq:review:sigmak-cont-combo}
			\begin{alignat}{2}
				\label{eq:review:sigmak-cont}
				\left. \unitvec{b}_{ji}^{\otimes l} \cdot \frac{\partial^l \mathfrak{D}_i^{m}(F)}{\partial \unitvec{b}_{ij}^{l}} \right|_{\gamma_{ij}} - \left. \unitvec{b}_{ij}^{\otimes l} \cdot \frac{\partial^l \mathfrak{D}_j^{m}(F)}{\partial \unitvec{b}_{ji}^{l}} \right|_{\gamma_{ij}} &= 0 \qquad & &\text{if } (s-m-l)p > 2, \\
				\label{eq:review:sigmak-cont-int}
				\mathcal{I}_{ij}^{p}\left(\unitvec{b}_{ji}^{\otimes l} \cdot \frac{\partial^{l+n} \mathfrak{D}_i^{m}(F)}{\partial \unitvec{t}_{ij}^{n} \partial \unitvec{b}_{ij}^{l}},  \unitvec{b}_{ij}^{\otimes l} \cdot \frac{\partial^{l+n} \mathfrak{D}_j^{m}(F)}{\partial \unitvec{t}_{ij}^{n} \partial \unitvec{b}_{ji}^{l}} \right) &< \infty \qquad & & \text{if } (s-m-l-n)p = 2.
			\end{alignat}
		\end{subequations}
	\end{enumerate}
	\begin{remark}
		As was the case in \cref{rem:review:sigma1-2d-relate}, condition \cref{eq:review:sigmak-cont} is simply the application of the vertex compatibility conditions for traces on the triangle \cite[eq. (7.2)]{Parker23} at every point on the edge $\gamma_{ij}$, provided that $u$ is smooth enough. 
	\end{remark}
	
	\subsection{The trace theorem on a tetrahedron}
	
	 Motivated by the conditions derived in the previous section, we define trace spaces as follows. Given a set of indices $\mathcal{S} \subseteq \{1,2,3,4\}$ with $|\mathcal{S}| \geq 1$, let $\Gamma_S := \cup_{i \in \mathcal{S}} \Gamma_i$. We define the trace space on part of the boundary $\Tr_k^{s, p}(\Gamma_{\mathcal{S}})$ for $k \in \mathbb{N}_0$ and $(s, p) \in \mathcal{A}_k$ as follows.	
	 \begin{align*}
	 	\begin{aligned}
	 		\Tr^{s,p}_k(\Gamma_{\mathcal{S}}) &:= \{ F = (f^0, f^1, \ldots, f^k) \in L^p(\Gamma_{\mathcal{S}})^{k+1} : \text{For $0 \leq m \leq k$,} \\ 
	 		&\qquad \quad \text{$F$ satisfies \cref{eq:review:sigmak-face-reg} for $i \in \mathcal{S}$ and } \\
	 		&\qquad \quad \text{\cref{eq:review:sigmak-cont-combo} for $i, j \in \mathcal{S}$ with $i < j$, $0 \leq l \leq m$ and $n \in \mathbb{N}_0$} \},
	 	\end{aligned}
	 \end{align*}
	 equipped with the norm
	 \begin{multline*}
	 	\| (f^0, f^1, \ldots, f^k) \|_{\Tr_k^{s, p}, \Gamma_{\mathcal{S}}}^p := \sum_{m=0}^{k} \sum_{i \in \mathcal{S}} \| f_i^{m} \|_{s-m-\frac{1}{p}, p, \Gamma_i}^p  \\
	 	+ \sum_{\substack{i, j \in \mathcal{S} \\ i < j \\ 0 \leq l \leq k \\ n \in \mathbb{N}_0} }  \begin{cases}
	 		\mathcal{I}_{ij}^p\left(\unitvec{b}_{ji}^{\otimes l} \cdot \frac{\partial^{l+n} \mathfrak{D}_i^{k}(F)}{ \partial \unitvec{t}_{ij}^{n} \partial \unitvec{b}_{ij}^{l}},  \unitvec{b}_{ij}^{\otimes l} \cdot \frac{\partial^{l+n} \mathfrak{D}_j^{k}(F)}{ \partial \unitvec{t}_{ij}^{n} \partial \unitvec{b}_{ji}^{l}} \right) & \text{if } (s-k-l-n)p = 2, \\
	 		0 & \text{otherwise}.
	 	\end{cases}
	 \end{multline*}
	 Note that the sum in the definition contains only finitely many nonzero terms,  and hence is well defined. When $\mathcal{S} = \{1,2,3,4\}$, we set $\Tr^{s,p}_k(\partial \reftet) := \Tr^{s,p}_k(\Gamma_{\mathcal{S}})$ and $\| \cdot \|_{\Tr_k^{s, p}, \partial \reftet} :=  \| \cdot \|_{\Tr_k^{s, p}, \Gamma_{\mathcal{S}}}$.  The following trace theorem is a consequence of the discussion in the previous section.
	\begin{theorem}
		\label{thm:trace-theorem}
		Let $\mathcal{S} \subseteq \{1,2,3,4\}$, $k \in \mathbb{N}_0$, and $(s, p) \in \mathcal{A}_k$ be given. Then, for every $u \in W^{s, p}(\reftet)$, the traces satisfy $(u, \partial_{\unitvec{n}} u, \ldots, \partial_{\unitvec{n}}^{k} u)|_{\Gamma_{\mathcal{S}}} \in  \Tr_k^{s, p}(\Gamma_{\mathcal{S}})$ and
		\begin{align}
			\label{eq:review:trace-theorem}
			\| (u, \partial_{\unitvec{n}} u, \ldots, \partial_{\unitvec{n}}^{k} u) \|_{\Tr_k^{s, p}, \Gamma_{\mathcal{S}}} \lesssim_{k, s, p} \| u \|_{s, p, \reftet}.
		\end{align}
	\end{theorem}

	\subsection{The trace of polynomials}
	
	Given $N \in \mathbb{N}_0$, let $\mathcal{P}_N(\reftet)$ denote the set of all polynomials of total degree at most $N$, while $\mathcal{P}_{-M} := \{0\}$ for $M > 0$. If $u \in \mathcal{P}_N(\reftet)$, then $u \in W^{s, p}(\reftet)$ for all  $s \geq 0$ and $p \geq 1$. Consequently, for each $k \in \mathbb{N}_0$ and $\mathcal{S} \subseteq \{1,2,3,4\}$, the traces $F = (u, \partial_{\unitvec{n}} u, \ldots, \partial_{\unitvec{n}}^k u)|_{\Gamma_{\mathcal{S}}} \in \Tr^{s, p}_k(\Gamma_{\mathcal{S}})$ for all $(s, p) \in \mathcal{A}_k$. In particular, $s$ may be taken to be arbitrarily large in \cref{eq:review:sigmak-cont}. Thus, the traces satisfy
	\begin{subequations}
		\label{eq:review:sigmak-cont-poly}
		\begin{alignat}{2}
			f^m_i &\in \mathcal{P}_{N-m}(\Gamma_i), \qquad & &0\leq m \leq k, \ i \in \mathcal{S}, \\
			\label{eq:review:sigmak-cont-poly-edge-compat}
			\mathfrak{D}_i^m(F)|_{\gamma_{ij}} &= \mathfrak{D}_j^m(F)|_{\gamma_{ij}}, \qquad & &0 \leq m \leq k, \ i, j \in \mathcal{S}, \ i < j, \\
			\label{eq:review:sigmak-cont-poly-edge-compat-deriv}
			\left. \unitvec{b}_{ji}^{\otimes l} \cdot \frac{\partial^l \mathfrak{D}_i^{k}(F)}{\partial \unitvec{b}_{ij}^{l}} \right|_{\gamma_{ij}} &= \left. \unitvec{b}_{ij}^{\otimes l} \cdot \frac{\partial^l \mathfrak{D}_j^{k}(F)}{\partial \unitvec{b}_{ji}^{l}} \right|_{\gamma_{ij}}, \qquad & &0 \leq l \leq k, \ i, j \in \mathcal{S}, \ i < j.
		\end{alignat}
	\end{subequations}

	Note that we have not included the integral condition \cref{eq:review:sigmak-cont-int} in the list \cref{eq:review:sigmak-cont-poly} above. The following lemma shows that if a tuple of functions defined on $\partial \reftet$ satisfy \cref{eq:review:sigmak-cont-poly}, then \cref{eq:review:sigmak-cont-int} is automatically satisfied.
	\begin{lemma}
		\label{lem:review:polys-trace-space}
		Let $\mathcal{S} \subseteq \{1,2,3,4\}$ and $k \in \mathbb{N}_0$. If $F : \Gamma_{\mathcal{S}} \to \mathbb{R}^{k+1}$ satisfies \cref{eq:review:sigmak-cont-poly}, then $F \in \Tr_k^{s, p}(\Gamma_{\mathcal{S}})$ for all $(s, p) \in \mathcal{A}_k$.
	\end{lemma}
	\begin{proof}
		Let $(s, p) \in \mathcal{A}_k$, $0 \leq l \leq m \leq k$, be given. Thanks to \cref{eq:review:sigmak-cont-poly-edge-compat,eq:review:sigmak-cont-poly-edge-compat-deriv}, the difference 
		\begin{align*}
			H_{ij} := \unitvec{b}_{ji}^{\otimes l} \cdot \frac{\partial^{l} \mathfrak{D}_i^{m}(F)}{\partial \unitvec{b}_{ij}^{l}} \circ \bdd{F}_{ij} -  \unitvec{b}_{ij}^{\otimes l} \cdot \frac{\partial^{l} \mathfrak{D}_j^{m}(F)}{\partial \unitvec{b}_{ji}^{l}} \circ \bdd{F}_{ji}, \qquad \text{ on $\reftri$}, i,j \in \mathcal{S},
		\end{align*}
		vanishes on the edge $\gamma_2$ of the reference triangle $\reftri$ and $H_{ij}$ has entries $\mathcal{P}_{N-m-l}(\reftri)$. Thus, $H_{ij} = x_2 G_{ij}$, where $G_{ij}$ has entries in $\mathcal{P}_{N-m-l-1}(\reftri)$. Consequently, for all $n \in \mathbb{N}_0$, there holds
		\begin{align*}
			\mathcal{I}_{ij}^{p}\left(\unitvec{b}_{ji}^{\otimes l} \cdot \frac{\partial^{l+n} \mathfrak{D}_i^{m}(F)}{\partial \unitvec{t}_{ij}^{n} \partial \unitvec{b}_{ij}^{l}},  \unitvec{b}_{ij}^{\otimes l} \cdot \frac{\partial^{l+n} \mathfrak{D}_j^{m}(F)}{\partial \unitvec{t}_{ij}^{n} \partial \unitvec{b}_{ji}^{l}} \right) &\approx_p \int_{\reftri} |\partial_{x_1}^n H_{ij}(\bdd{x})|^p \frac{\d{\bdd{x}}}{x_2} \\
			&= \int_{\reftri} |\partial_{x_1}^n G_{ij}(\bdd{x})|^p x_2^{p-1} \d{\bdd{x}},
		\end{align*}
		which is finite since $G_{ij}$ has polynomial entries. The inclusion $F \in \Tr_k^{s, p}(\Gamma_{\mathcal{S}})$ now follows from \cref{eq:review:sigmak-cont-poly}.
	\end{proof}

	\subsection{Statement of the main result}
	
	The aim of the current work is to construct a right inverse $\mathcal{L}_k$ of the operator $u \mapsto (u, \partial_{\unitvec{n}} u, \ldots, \partial_{\unitvec{n}}^{k} u)|_{\partial \reftet}$ for each $k \in \mathbb{N}_0$ that is bounded from $\Tr_k^{s, p}(\partial \reftet)$ into $W^{s, p}(\reftet)$ for all $(s, p) \in \mathcal{A}_k$ \textit{and} preserves polynomials in the following sense: if $F = (f^0, f^1, \ldots, f^k)$ is the trace of some degree $N$ polynomial, then $\mathcal{L}_k (F)$ is a polynomial of degree $N$. In particular, the main result is as follows.
	
	\begin{theorem}
		\label{thm:main-result}
		Let $k \in \mathbb{N}_0$. There exists a linear operator
		\begin{align*}
			\mathcal{L}_k : \bigcup_{(s, p) \in \mathcal{A}_k} \Tr_k^{s, p}(\partial \reftet) \to L^1(\reftet)
		\end{align*}
		satisfying the following properties: for all $(s, p) \in \mathcal{A}_{k}$ and $F = (f^0, f^1, \ldots, f^k) \in \Tr_k^{s, p}(\partial \reftet)$, $\mathcal{L}_k(F) \in W^{s, p}(\reftet)$,
		\begin{align*}
			\partial_{\unitvec{n}}^{l} \mathcal{L}_k(F)|_{\partial \reftet} = f^l, \qquad 0 \leq l \leq k, \quad \text{and} \quad \| \mathcal{L}_k(F)\|_{s, p, \reftet} \lesssim_{k, s, p} \| F \|_{\Tr_k^{s, p}, \partial \reftet}.
		\end{align*}
		Moreover, if  $F$ is a piecewise polynomial of degree $N \in \mathbb{N}_0$ satisfying \cref{eq:review:sigmak-cont-poly} with $\mathcal{S} = \{1,2,3,4\}$, then $\mathcal{L}_k(F) \in \mathcal{P}_N(\reftet)$.
	\end{theorem}
	\noindent The construction of the lifting operator $\mathcal{L}_k$ in \cref{thm:main-result} is the focus of the next section, and the proof of \cref{thm:main-result} appears in \cref{sec:proof-thm:main-result}. An immediate consequence is the following characterization of the range of the trace operator.
	
	\begin{corollary}
		For each $k \in \mathbb{N}_0$, the operator $u \mapsto (u, \partial_{\unitvec{n}} u, \ldots, \partial_{\unitvec{n}}^{k} u)|_{\partial \reftet}$ is surjective from $W^{s, p}(K)$ onto $\Tr_k^{s, p}(\partial K)$ for all $(s, p) \in \mathcal{A}_k$.
	\end{corollary}
	
	\section{Construction of the lifting operator}
	\label{sec:construction}

	The construction of the lifting operator $\mathcal{L}_k$, $k \in \mathbb{N}_0$, proceeds face-by-face using similar techniques to \cite{Munoz97,Parker23}. The main idea is to perform a sequence of liftings and corrections using a fundamental convolution operator (see e.g. \cite[eq. (4.2)]{Arn88}, \cite{BCMP91}, \cite{Bern95}, \cite[p. 56, eq. (2.1)]{Bern07}, \cite[\S 2.5.5]{Necas11}) and subsequent modifications to it. Given a nonnegative integer $k \in \mathbb{N}_0$, a smooth compactly supported function $b \in C_c^{\infty}(\reftri)$, and a function $f : \reftri \to \mathbb{R}$, we define the operator $\mathcal{E}_k^{[1]}$ formally by the rule
	\begin{align}
		\label{eq:ek1-def}
		\mathcal{E}_k^{[1]}(f)(\bdd{x}, z) := \frac{(-z)^k}{k!} \int_{\reftri} b(\bdd{y}) f(\bdd{x} + z \bdd{y}) \d{\bdd{y}} \qquad \forall (\bdd{x}, z) \in \reftet,
	\end{align}
	and we use the notation $\mathcal{E}_k^{[1]}[b]$ when we want to make the dependence on $b$ explicit. Note that for $(\bdd{x}, z) \in \reftet$ and $\bdd{y} \in \reftri$, there holds $\bdd{x} + z \bdd{y} \in \reftri$, and so \cref{eq:ek1-def} is well-defined for e.g. $f \in C^{\infty}(\bar{\reftri})$. For functions $f : \Gamma_1 \to \mathbb{R}$ we define
	\begin{align}
		\label{eq:ek1-gamma1-def}
		\mathcal{E}_k^{[1]}(f) := \mathcal{E}_k^{[1]}(f \circ \mathfrak{I}_1), \qquad \text{where} \quad \mathfrak{I}_1(\bdd{x}) := (\bdd{x}, 0) \qquad \forall \bdd{x} \in \reftri. 
	\end{align}
	
	\subsection{Lifting from one face}
	
	The first result concerns the interpolation and continuity properties of $\mathcal{E}_k^{[1]}$.	
	\begin{lemma}
		\label{lem:em-gamma1-cont-high}
		Let $b \in C_c^{\infty}(\reftri)$, $k \in \mathbb{N}_0$, and $(s, p) \in \mathcal{A}_k$. Then, for all $f \in W^{s-k-\frac{1}{p},p}(\Gamma_1)$, there holds
		\begin{align}
			\label{eq:em-gamma1-interp}
			\partial_{\unitvec{n}}^m \mathcal{E}_k^{[1]}(f)|_{\Gamma_1} = \delta_{mk} \left( \int_{\reftri} b(\bdd{x}) \d{\bdd{x}} \right) f, \qquad 0 \leq m \leq k,  
		\end{align}
		and 
		\begin{align}
			\label{eq:em-gamma1-continuity-high}
			\| \mathcal{E}_k^{[1]}(f) \|_{s, p, \reftet} \lesssim_{b, k, s, p} 
			\| f \|_{s-k-\frac{1}{p}, p, \Gamma_1}.
		\end{align}
		Moreover, if $f \in \mathcal{P}_N(\Gamma_1)$, $N \in \mathbb{N}_0$, then $\mathcal{E}_k^{[1]}(f) \in \mathcal{P}_{N+k}(\reftet)$.
	\end{lemma}
	The proof appears in \cref{sec:proof-em-gamma1-cont-high}. We now construct a lifting operator from $\Gamma_1$. 
	\begin{lemma}
		\label{lem:l1}
		Let $b \in C_c^{\infty}(\reftri)$ with $\int_{\reftri} b(\bdd{x}) \d{\bdd{x}} = 1$ and $k \in \mathbb{N}_0$. We formally define the following operators for $F = (f^0, f^1, \ldots, f^k)~\in~L^p(\Gamma_1)^{k+1}$:
		\begin{subequations}
			\label{eq:l1-def}
			\begin{alignat}{2}
				\mathcal{L}_0^{[1]}(F) &:= \mathcal{E}_0^{[1]}(f^0), \qquad & & \\
				\mathcal{L}_m^{[1]}(F) &:= \mathcal{E}_{m}^{[1]}( f^m - \partial_{\unitvec{n}}^{m} \mathcal{L}_{m-1}^{[1]}(F)|_{\Gamma_1} ), \qquad & &1 \leq m \leq k.
			\end{alignat}
		\end{subequations}
		Then, for all $(s, p) \in \mathcal{A}_k$ and $F \in \Tr_k^{s, p}(\Gamma_1)$, $\mathcal{L}_k^{[1]}(F)$ is well-defined and there holds
		\begin{align}
			\label{eq:l1-gamma1-interp-continuity}
			\partial_{\unitvec{n}}^m \mathcal{L}_k^{[1]}(F)|_{\Gamma_1} = f^m, \qquad 0 \leq m \leq k,  \quad \text{and} \quad \| \mathcal{L}_k^{[1]}(F) \|_{s, p, \reftet} \lesssim_{b, k, s, p} \|F\|_{\Tr_k^{s, p}, \Gamma_1}.
		\end{align}
		Moreover, if $f^m \in \mathcal{P}_{N-m}(\Gamma_1)$, $0 \leq m \leq k$, for some $N \in \mathbb{N}_0$, then $\mathcal{L}_k^{[1]}(F) \in \mathcal{P}_N(\reftet)$.
	\end{lemma}
	\begin{proof}
		We proceed by induction on $k$. The case $k = 0$ follows immediately from \cref{lem:em-gamma1-cont-high}. Now assume that the lemma is true for some $k \in \mathbb{N}_0$ and let $(s, p) \in \mathcal{A}_{k+1}$ and $F \in \Tr_{k+1}^{s, p}(\Gamma_1)$ be as in the statement of the lemma. Then, we may apply the lemma to $\tilde{F} = (f^0, f^1, \ldots, f^k) \in \Tr_k^{s, p}(\Gamma_1)$ to conclude that for $0 \leq m \leq k$ and
		\begin{align*}
			\partial_{\unitvec{n}}^m \mathcal{L}_k^{[1]}(\tilde{F})|_{\Gamma_1} = f^m, \ \ 0 \leq m \leq k, \ \ \text{and} \ \
			\| \mathcal{L}_k^{[1]}(\tilde{F}) \|_{s, p, \reftet} \lesssim_{b, k, s, p} 
			\|\tilde{F}\|_{\Tr_k^{s, p}, \Gamma_1}.
		\end{align*}
		Thanks to the trace theorem (\cref{thm:trace-theorem}), there holds $f^{k+1} - \partial_{\unitvec{n}}^{k+1} \mathcal{L}_{k}^{[1]}(F)|_{\Gamma_1} \in W^{s-k-1-\frac{1}{p}, p}(\Gamma_1)$ with
		\begin{align*}
			 \| f^{k+1} - \partial_{\unitvec{n}}^{k+1} \mathcal{L}_{k}^{[1]}(\tilde{F}) \|_{s-k-1-\frac{1}{p}, p, \Gamma_1} &\lesssim_{b, k, s, p} \|F\|_{\Tr_{k+1}^{s, p}, \Gamma_1},
		\end{align*}
		and so \cref{eq:l1-gamma1-interp-continuity} follows from \cref{lem:em-gamma1-cont-high}. Additionally, if $F$ satisfies $f^m \in \mathcal{P}_{N-m}(\Gamma_1)$, $0 \leq m \leq k+1$, for some $N \in \mathbb{N}_0$, then $\tilde{F}$ satisfies the same condition, where the upper bound of $m$ is restricted to $k$. Consequently, $\mathcal{L}_k^{[1]}(\tilde{F}) \in \mathcal{P}_N(\reftet)$ and so $f^{k+1} - \partial_{\unitvec{n}}^{k+1} \mathcal{L}_{k}^{[1]}(\tilde{F})|_{\Gamma_1} \in \mathcal{P}_{N-k-1}(\Gamma_1)$. Thus, $\mathcal{L}_{k+1}^{[1]}(F) \in \mathcal{P}_N(\reftet)$ thanks to \cref{lem:em-gamma1-cont-high}.  
	\end{proof}

	\subsection{Lifting from two faces}
	\label{sec:two-faces}
	
	We now seek a lifting operator from $\Gamma_2$ that has zero trace on $\Gamma_1$. The operator will be a generalization of the form introduced in \cite{Munoz97}. We first define an operator that lifts traces from $\Gamma_1$ and has zero trace on $\Gamma_2$, and then define the lifting operator from $\Gamma_2$ in terms of this operator. To this end, denote by $\omega_i$ the barycentric coordinates of $\reftri$ extended to all of $\mathbb{R}^2$ as follows.
	\begin{align}
		\label{eq:omegai-def}
		\omega_i(\bdd{x}) := \min\{ x_i, 1 \}, \quad 1 \leq i \leq 2, \quad \text{and} \quad \omega_3(\bdd{x}) := 1 - \min\{x_1 + x_2, 1\} \quad \forall \bdd{x} \in \mathbb{R}^2.
	\end{align}	
	Given nonnegative integers $k, r \in \mathbb{N}_0$, a smooth compactly supported function $b \in C_c^{\infty}(\reftri)$, and a function $f : \reftri \to \mathbb{R}$, we define the operator $\mathcal{M}_{k, r}^{[1]}$ formally by the rule
	\begin{align}
		\label{eq:mkr-def}
		\begin{aligned}
			\mathcal{M}_{k, r}^{[1]}(f)(\bdd{x}, z) &:= x_2^r \mathcal{E}_k^{[1]}(\omega_2^{-r} f)(\bdd{x}, z) 
			\\
			&= x_2^r \frac{(-z)^k}{k!} \int_{\reftri}  \frac{b(\bdd{y}) f(\bdd{x} + z \bdd{y})}{(x_2 + z y_2)^r} \d{\bdd{y}} \qquad \forall (\bdd{x}, z) \in \reftet.
		\end{aligned}
	\end{align}
	Note that when $r = 0$, we have $\mathcal{M}_{k, 0}^{[1]} = \mathcal{E}_k^{[1]}$. For functions $f : \Gamma_1 \to \mathbb{R}$, we again abuse notation and set $\mathcal{M}_{k, r}^{[1]}(f) := \mathcal{M}_{k, r}^{[1]}(f \circ \mathfrak{I}_1)$. 
	
	The presence of the weight $\omega_2^{-r}$ in the operator $\mathcal{M}_{k, r}^{[1]}$ means that derivatives of $f : \Gamma_{1} \to \mathbb{R}$ up to order $r$ have to vanish on edge $\gamma_{12}$ in an appropriate sense. To this end, let $s = m + \sigma$ with $m \in \mathbb{N}_0$ and $\sigma \in [0, 1)$ and $1 < p < \infty$. Given a face $\Gamma_j$, $1 \leq j \leq 4$, and $\mathfrak{E}$ a subset of the edges of $\Gamma_j$, we define the following subspaces of $W^{s, p}(\Gamma_j)$ with vanishing traces on the edges in $\mathfrak{E}$:
	\begin{multline}
		\label{eq:wsp-e-def}
		W^{s, p}_{\mathfrak{E}}(\Gamma_j) := \left\{ f \in W^{s, p}(\Gamma_j) : D_{\Gamma}^{\beta} f|_{\gamma} = 0 \text{ for all } 0 \leq |\beta| < s - \frac{1}{p} \text{ and } \gamma \in \mathfrak{E} \right. \\ \left. \vphantom{\frac{1}{p}} \text{ and } \inorm{\mathfrak{E}}{f}{s, p, \reftri} < \infty \right\},
	\end{multline}
	where the norm on $W^{s, p}_{\mathfrak{E}}(\Gamma_j)$ is given by
	\begin{align*}
		\inormsup{\mathfrak{E}}{f}{s, p, \Gamma_j}{p} := \|f\|_{s, p, \Gamma_j}^p + \begin{cases}
			\| \dist(\cdot, \bigcup_{\gamma \in \mathfrak{E}} \gamma )^{-\sigma} D_{\Gamma}^m f  \|_{p, \Gamma_j}^p & \text{if } \sigma p = 1 \text{ and } \mathfrak{E} \neq \emptyset, \\
			0 & \text{otherwise},
		\end{cases}
	\end{align*}
	and we recall that $D_{\Gamma}$ is the surface gradient operator. 
	When $\mathfrak{E}$ consists of only one element $\gamma$, we set $W_{\gamma}^{s, p}(\Gamma_j) := W_{\mathfrak{E}}^{s, p}(\Gamma_j)$ and $\inorm{\gamma}{f}{s, p, \Gamma_j} := \inorm{\mathfrak{E}}{f}{s, p, \Gamma_j}$. One can readily verify that the spaces $W_{\mathfrak{E}}^{s, p}(\Gamma_j)$ are Banach spaces and that the following relations hold:
	\begin{align}
		\label{eq:wsp-e-intersection-id}
		W_{\mathfrak{E}}^{s, p}(\Gamma_j) = \bigcap_{\gamma \in \mathfrak{E}} W_{\gamma}^{s, p}(\Gamma_j) \quad \text{and} \quad \inorm{\mathfrak{E}}{f}{s, p, \Gamma_j} \approx_{s,p} \sum_{\gamma \in \mathfrak{E}} \inorm{\gamma}{f}{s, p, \Gamma_j}.
	\end{align}
	Given a subset of edges $\mathfrak{E}$ of the reference triangle $\reftri$, we define the spaces $W_{\mathfrak{E}}^{s, p}(\reftri)$ analogously.

	The first result states the continuity properties of $\mathcal{M}_{k, r}^{[1]}$.
	\begin{lemma}
		\label{lem:mmr-gamma1-cont-high}
		Let $b \in C_c^{\infty}(\reftri)$, $k, r \in \mathbb{N}_0$, and $(s, p) \in \mathcal{A}_k$. Then, for all \\ $f~\in~W^{s-k-\frac{1}{p},p}(\Gamma_1) \cap W_{\gamma_{12}}^{\min\{s-k-\frac{1}{p}, r\}, p}(\Gamma_1)$, there holds
		\begin{subequations}
			\begin{alignat}{2}
				\label{eq:mmr-gamma1-interp}
				\partial_{\unitvec{n}}^m \mathcal{M}_{k, r}^{[1]}(f)|_{\Gamma_1} &= \delta_{km} \left( \int_{\reftri} b(\bdd{x}) \d{\bdd{x}} \right) f, \qquad & &0 \leq m \leq k, \\
				\label{eq:mmr-gamma2-zero}
				\partial_{\unitvec{n}}^{j} \mathcal{M}_{k, r} ^{[1]}(f)|_{\Gamma_2} &= 0, \qquad & &0 \leq j < \min\left\{ r, s - \frac{1}{p} \right\},
			\end{alignat}
		\end{subequations}
		and 
		\begin{align}
			\label{eq:mmr-gamma1-continuity-high}
			\| \mathcal{M}_{k, r}^{[1]}(f) \|_{s, p, \reftet} \lesssim_{b, k, r, s, p} \begin{cases}
				\inorm{\gamma_{12}}{f}{s-k-\frac{1}{p}, p, \Gamma_1} & \text{if } s \leq k+r+\frac{1}{p}, \\
				\| f\|_{s-k-\frac{1}{p}, p, \Gamma_1} & \text{if } s > k + r + \frac{1}{p}.
			\end{cases}
		\end{align}
		Moreover, if $f \in \mathcal{P}_N(\Gamma_1)$, $N \in \mathbb{N}_0$, satisfies $D_{\Gamma}^{l} f|_{\gamma_{12}} = 0$ for $0 \leq l \leq r-1$, then $\mathcal{M}_{k, r}^{[1]}(f) \in \mathcal{P}_{N+k}(\reftet)$.
	\end{lemma}
	The proof of \cref{lem:mmr-gamma1-cont-high} appears in \cref{sec:proof-mmr-gamma1-cont-high}. 
	By mapping the other faces of $\reftet$ to $\Gamma_1$ and mapping $\reftet$ onto itself in an appropriate fashion, we may define operators corresponding to these faces. In particular, we define the following operator corresponding to $\Gamma_2$:
	\begin{align*}
		\mathcal{M}_{k, r}^{[2]}(f)(\bdd{x}, z) := \mathcal{M}_{k, r}^{[1]}(f \circ \mathfrak{I}_2) \circ \mathfrak{R}_{12}(\bdd{x}, z) \qquad \forall (\bdd{x}, z) \in \reftet,
	\end{align*}
	where $\mathfrak{I}_2(\bdd{x}) := (x_1, 0, x_2)$ and $\mathfrak{R}_{12}(\bdd{x}, z) := (x_1, z, x_2)$ for all $(\bdd{x}, z) \in \reftet$.
	
	Thanks to the chain rule and the smoothness of the mappings $\mathfrak{I}_2$ and $\mathfrak{R}_{12}$, the continuity and interpolation properties of $\mathcal{M}_{k, r}^{[2]}$ follow immediately from \cref{lem:mmr-gamma1-cont-high}.
	\begin{corollary}
		\label{cor:mmr-gamma2-cont-high}
		Let $b \in C_c^{\infty}(\reftri)$, $k, r \in \mathbb{N}_0$, and $(s, p) \in \mathcal{A}_k$. Then, for all $f \in W^{s-k-\frac{1}{p},p}(\Gamma_2) \cap W_{\gamma_{12}}^{\min\{s-k-\frac{1}{p}, r\}, p}(\Gamma_2)$, there holds
		\begin{subequations}
			\begin{alignat}{2}
				\label{eq:mmr2-gamma1-interp}
				\partial_{\unitvec{n}}^m \mathcal{M}_{k, r}^{[2]}(f)|_{\Gamma_2} &= \delta_{km} \left( \int_{\reftri} b(\bdd{x}) \d{\bdd{x}} \right) f, \qquad & &0 \leq m \leq k, \\
				\label{eq:mmr2-gamma2-zero}
				\partial_{\unitvec{n}}^{j} \mathcal{M}_{k, r} ^{[2]}(f)|_{\Gamma_1} &= 0, \qquad & &0 \leq j < \min\left\{ r, s - \frac{1}{p} \right\},
			\end{alignat}
		\end{subequations}
		and 
		\begin{align}
			\label{eq:mmr2-gamma2-continuity-high}
			\| \mathcal{M}_{k, r}^{[2]}(f) \|_{s, p, \reftet} \lesssim_{b, k, r, s, p} \begin{cases}
				\inorm{\gamma_{12}}{f}{s-k-\frac{1}{p}, p, \Gamma_2} & \text{if } s \leq k + r + \frac{1}{p}, \\
				\| f\|_{s-k-\frac{1}{p}, p, \Gamma_2} & \text{if } s > k + r + \frac{1}{p}.
			\end{cases}
		\end{align}
		Moreover, if $f \in \mathcal{P}_N(\Gamma_2)$, $N \in \mathbb{N}_0$, satisfies $D_{\Gamma}^{l} f|_{\gamma_{12}} = 0$ for $0 \leq l \leq r-1$, then $\mathcal{M}_{k, r}^{[2]}(f) \in \mathcal{P}_{N+k}(\reftet)$.
	\end{corollary}

	\subsubsection{Regularity of partially vanishing traces}
	
	The operator $\mathcal{M}_{k, r}^{[2]}$ lifts traces from $\Gamma_2$ to $\reftet$ and has zero trace on $\Gamma_1$, which are the properties we desired to correct the traces of $\mathcal{L}_k^{[1]}$ on $\Gamma_2$. However, $\mathcal{M}_{k, r}^{[2]}$ acts on functions belonging to $W^{s-k-\frac{1}{p},p}(\Gamma_2) \cap W_{\gamma_{12}}^{\min\{s-k - \frac{1}{p}, r\}, p}(\Gamma_2)$ rather than just functions in $W^{s-k-\frac{1}{p},p}(\Gamma_2)$. The main result of this section characterizes one scenario in which traces belong to the space $W^{s-k-\frac{1}{p},p}(\Gamma_2) \cap W_{\gamma_{12}}^{\min\{s-k - \frac{1}{p}, r\}, p}(\Gamma_2)$, and fortunately, we will encounter exactly this scenario in our construction.
	
	We have the following result which characterizes the regularity of the restriction of a trace $F \in \Tr_k^{s,p}(\Gamma_i \cup \Gamma_j)$ to $\Gamma_j$ when $F$ vanishes on $\Gamma_i$ and the first $l$ components of $F$ vanish on $\Gamma_j$.	
	\begin{lemma}
		\label{lem:vanish-trace-reg}
		Let $k \in \mathbb{N}_0$, $(s, p) \in \mathcal{A}_k$, $1 \leq l \leq k$, and $1 \leq i, j \leq 4$ with $i \neq j$ be given. Suppose that $F = (f^0, f^1, \ldots, f^k) \in \Tr_k^{s,p}(\Gamma_i \cup \Gamma_j)$ satisfies
		\begin{enumerate}
			\item[(i)] $F = (0, 0, \ldots, 0)$ on $\Gamma_i$;
			\item[(ii)] $f_j^m = 0$ on $\Gamma_j$ for $0 \leq m \leq l-1$.
		\end{enumerate}
		Then, there holds $f^l_j \in W^{s - l - \frac{1}{p}, p}(\Gamma_j) \cap W^{\min\{s - l - \frac{1}{p}, k+1 \}, p}_{\gamma_{ij}}(\Gamma_j)$ and
		\begin{subequations}
			\label{eq:trace-math-reg}
			\begin{alignat}{2}
				\label{eq:trace-match-low}
				\inorm{\gamma_{ij}}{f_j^l}{s-l-\frac{1}{p}, p, \Gamma_{j}} &\lesssim_{k, s, p} \|F\|_{\Tr_k^{s, p}, \Gamma_i \cup \Gamma_j} \qquad & & \text{if } s-l \leq k + 1 + \frac{1}{p}, \\
				\label{eq:trace-match-high}
				\|f_j^l\|_{s-l-\frac{1}{p}, p, \Gamma_{j}}  &\lesssim_{k, s, p} \|F\|_{\Tr_k^{s, p}, \Gamma_i \cup \Gamma_j} \qquad & & \text{if } s - l > k + 1 + \frac{1}{p}.
			\end{alignat}
		\end{subequations}
	\end{lemma}		
	\begin{proof}
		Without loss of generality, assume that $i < j$. We first show that for $\alpha \in \mathbb{N}_0^2$ there holds 
		\begin{align}
			\label{eq:proof:trace-vanishes-edges}
			\left. \frac{\partial^{|\alpha|} f^l_j}{\partial \unitvec{t}_{ij}^{\alpha_1} \partial \unitvec{b}_{ji}^{\alpha_2}} \right|_{\gamma_{ij}} \equiv 0 \qquad 0 \leq |\alpha| < \min\left\{s-l - \frac{2}{p}, k+1 \right\}, 
		\end{align}
		where $\unitvec{b}_{ij}$ and $\unitvec{b}_{ji}$ are defined in \cref{eq:orthog-vectors-two-faces}.
		
		\noindent \textbf{Step 1: $0 \leq \alpha_2 \leq k-l$ and $|\alpha| < \min\{s-l-2/p, k+1\}$. } Manipulating definitions shows that
		\begin{align}
			\frac{ \partial^{\alpha_1+r} \mathfrak{D}_j^{l}(F) }{ \partial \unitvec{t}_{ij}^{\alpha_1} \partial  \unitvec{b}_{ij}^r } = \unitvec{b}_{ij}^{\otimes r} \cdot \frac{ \partial^{\alpha_1} \mathfrak{D}_j^{l+r}(F)} {\partial \unitvec{t}_{ij}^{\alpha_1} } \qquad 0 \leq r \leq k-l,
		\end{align}
		and so there holds
		\begin{align*}
			\frac{\partial^{|\alpha|} f^l_j}{\partial \unitvec{t}_{ij}^{\alpha_1} \partial \unitvec{b}_{ji}^{\alpha_2}} =  \unitvec{n}_j^{\otimes l} \cdot \frac{\partial^{|\alpha|} \mathfrak{D}_j^{l}(F) }{ \partial \unitvec{t}_{ij}^{\alpha_1} \partial  \unitvec{b}_{ji}^{\alpha_2} } =  \unitvec{n}_j^{\otimes l}  \cdot \unitvec{b}_{ji}^{\otimes \alpha_2} \cdot \frac{ \partial^{\alpha_1} \mathfrak{D}_j^{l+\alpha_2}(F)}{ \partial \unitvec{t}_{ij}^{\alpha_1} },
		\end{align*}
		and using that $\mathfrak{D}_i^{l+\alpha_2}(F)|_{\Gamma_i} = 0$ by (i) gives
		\begin{align}
			\label{eq:proof:trace-id-low}
			\frac{\partial^{|\alpha|} f^l_j}{\partial \unitvec{t}_{ij}^{\alpha_1} \partial \unitvec{b}_{ji}^{\alpha_2}} \circ \bdd{F}_{ji}
			= \unitvec{n}_j^{\otimes l}  \cdot \unitvec{b}_{ji}^{\otimes \alpha_2} \cdot \left(  \frac{ \partial^{\alpha_1} \mathfrak{D}_j^{l+\alpha_2}(F)}{ \partial \unitvec{t}_{ij}^{\alpha_1} }  \circ \bdd{F}_{ji} - \frac{ \partial^{\alpha_1} \mathfrak{D}_i^{l+\alpha_2}(F)}{ \partial \unitvec{t}_{ij}^{\alpha_1} }  \circ \bdd{F}_{ij} \right) 
		\end{align}
		on $\reftri$. Equality \cref{eq:proof:trace-vanishes-edges} now follows from \cref{eq:review:sigmak-cont}.
		
		\noindent \textbf{Step 2: $k-l+1 \leq \alpha_2 \leq k$ and  $|\alpha| < \min\{s-l-2/p, k+1\}$. } The same arguments as in Step 1 show that
		\begin{align*}
			\frac{\partial^{|\alpha|} f^l_j}{ \partial \unitvec{t}_{ij}^{\alpha_1} \partial \unitvec{b}_{ji}^{\alpha_2}} = \unitvec{n}_j^{\otimes l} \cdot \unitvec{b}_{ji}^{\otimes k-l} \cdot \frac{\partial^{|\alpha|-k+l} \mathfrak{D}_j^{k}(F) }{  \partial \unitvec{t}_{ij}^{\alpha_1}  \partial  \unitvec{b}_{ji}^{\alpha_2-k+l} }.
		\end{align*}
		By construction, there exist constants $a_1$ and $a_2$ such that $\unitvec{n}_j = a_1 \unitvec{b}_{ij} + a_2 \unitvec{b}_{ji}$, and so
		\begin{align*}
			\frac{\partial^{|\alpha|} f^l_j}{ \partial \unitvec{t}_{ij}^{\alpha_1} \partial \unitvec{b}_{ji}^{\alpha_2}}  = \sum_{r=0}^{l} c_r \unitvec{b}_{ij}^{\otimes r} \cdot \unitvec{b}_{ji}^{\otimes k-r} \cdot \frac{\partial^{|\alpha|-k+l} \mathfrak{D}_j^{k}(F) }{  \partial \unitvec{t}_{ij}^{\alpha_1} \partial  \unitvec{b}_{ji}^{|\alpha|-k+l} } 
			&= \sum_{r=0}^{l} c_r \unitvec{b}_{ij}^{\otimes r} \cdot  \frac{\partial^{|\alpha|+l-r} \mathfrak{D}_j^{r}(F) }{  \partial \unitvec{t}_{ij}^{\alpha_1} \partial  \unitvec{b}_{ji}^{\alpha_2+l-r} }
		\end{align*}
		for some suitable constants $\{c_r\}_{r=0}^{l}$. For $0 \leq r \leq l-1$, $\mathfrak{D}^r_j(F) = 0$ by (ii), and so
		\begin{align*}
			\frac{\partial^{|\alpha|} f^l_j}{ \partial \unitvec{t}_{ij}^{\alpha_1} \partial \unitvec{b}_{ji}^{\alpha_2}} = c_l \unitvec{b}_{ij}^{\otimes l} \cdot  \frac{\partial^{|\alpha|} \mathfrak{D}_j^{l}(F) }{ \partial \unitvec{t}_{ij}^{\alpha_1} \partial  \unitvec{b}_{ji}^{\alpha_2} } 
			&= c_l \unitvec{b}_{ij}^{\otimes l} \cdot  \unitvec{b}_{ji}^{\otimes k-l} \cdot \frac{\partial^{|\alpha|-k+l} \mathfrak{D}_j^{k}(F) }{ \partial \unitvec{t}_{ij}^{\alpha_1} \partial  \unitvec{b}_{ji}^{\alpha_2-k+l} } \\
			&= c_l  \unitvec{b}_{ij}^{\otimes k-\alpha_2} \cdot  \unitvec{b}_{ji}^{\otimes k-l} \cdot  \left(\unitvec{b}_{ij}^{\otimes \alpha_2-k+l} \cdot \frac{\partial^{|\alpha|-k+l} \mathfrak{D}_j^{k}(F) }{    \partial \unitvec{t}_{ij}^{\alpha_1} \partial \unitvec{b}_{ji}^{\alpha_2-k+l} }  \right).
		\end{align*}
		Applying (i) then gives the following identity on $\reftri$:
		\begin{multline}
			\label{eq:proof:trace-id-high}
			\frac{\partial^{|\alpha|} f^l_j}{ \partial \unitvec{t}_{ij}^{\alpha_1} \partial \unitvec{b}_{ji}^{\alpha_2}} \circ \bdd{F}_{ji}  =  c_l  \unitvec{b}_{ij}^{\otimes k-\alpha_2} \cdot  \unitvec{b}_{ji}^{\otimes k-l} \\
			\cdot  \left( \unitvec{b}_{ij}^{\otimes \alpha_2-k+l} \cdot \frac{\partial^{|\alpha|-k+l} \mathfrak{D}_j^{k}(F) }{    \partial \unitvec{t}_{ij}^{\alpha_1} \partial  \unitvec{b}_{ji}^{\alpha_2-k+l} } \circ \bdd{F}_{ji} -  \unitvec{b}_{ji}^{\otimes \alpha_2-k+l} \cdot \frac{\partial^{|\alpha|-k+l} \mathfrak{D}_i^{k}(F) }{  \partial \unitvec{t}_{ij}^{\alpha_1} \partial  \unitvec{b}_{ij}^{\alpha_2-k+l} } \circ \bdd{F}_{ij} \right).
		\end{multline}
		Equality \cref{eq:proof:trace-vanishes-edges} then follows from \cref{eq:review:sigmak-cont}.
		
		\noindent \textbf{Step 3: $f_j^l \in W^{s-l-\frac{1}{p}, p}(\Gamma_j) \cap W^{\min\{s - l - \frac{1}{p}, k+1 \} p}_{\gamma_{ij}}(\Gamma_j)$. } For $s - 2/p \notin \mathbb{Z}$, the inclusion $f_j^l \in W^{s-l-\frac{1}{p}, p}(\Gamma_j) \cap W^{\min\{s - l - \frac{1}{p}, k+1 \} , p}_{\gamma_{ij}}(\Gamma_j)$ follows from \cref{eq:proof:trace-vanishes-edges}, and  \cref{eq:trace-match-low,eq:trace-match-high} are an immediate consequence of the definition of the $\|\cdot\|_{\Tr_k^{s, p}, \Gamma_i \cup \Gamma_j}$ norm. For $s - 2/p \in \mathbb{Z}$,  and $|\alpha| = \min\{s-l- 2/p, k+1\}$, there holds 		
		\begin{align}
			\label{eq:proof:trace-match-dist-integral}
			\int_{\Gamma_j} \left| \frac{\partial^{|\alpha|} f^l_j}{\partial \unitvec{t}_{ij}^{\alpha_1} \partial \unitvec{b}_{ji}^{\alpha_2}}(\bdd{x}) \right|^p \frac{\d{\bdd{x}}}{\dist(\bdd{x}, \gamma_{ij})} &\approx_p \int_{\reftri}  \left| \frac{\partial^{|\alpha|} f^l_j}{\partial \unitvec{t}_{ij}^{\alpha_1} \partial \unitvec{b}_{ji}^{\alpha_2}} \right|^p \circ \bdd{F}_{ji}(\bdd{x}) \frac{\d{\bdd{x}}}{x_2},
		\end{align}
		and so the inclusion $f_j^l \in W^{\min\{s - l - \frac{1}{p}, k + 1 \}, p}_{\gamma_{ij}}(\Gamma_j)$ follows from \cref{eq:proof:trace-id-low,eq:proof:trace-id-high,eq:review:sigmak-cont-int}, while \cref{eq:trace-match-low} follows from the definition of the norm.
	\end{proof}

	\subsubsection{Construction of lifting}
	
	In the following lemma, we construct the lifting operator $\mathcal{L}_{k}^{[2]}$ in the same fashion as $\mathcal{L}_k^{[1]}$ \cref{eq:l1-def}, replacing the use of $\mathcal{E}_{m}^{[1]}$ with $\mathcal{M}_{m, k+1}^{[2]}$.
	\begin{lemma}
		\label{lem:l2}
		Let $b \in C_c^{\infty}(\reftri)$ with $\int_{\reftri} b(\bdd{x}) \d{\bdd{x}} = 1$, $k \in \mathbb{N}_0$, and $\mathcal{S} = \{1,2\}$. For  $F = (f^0, f^1, \ldots, f^k) \in L^p(\Gamma_1 \cup \Gamma_2)^{k+1}$, we formally define the following operators:
		\begin{subequations}
			\label{eq:l2-def}
			\begin{alignat}{2}
				\mathcal{L}_{k, 0}^{[2]}(F) &:= \mathcal{L}_k^{[1]}(F) + \mathcal{M}_{0, k+1}^{[2]}(f_2^0 - \mathcal{L}_k^{[1]}(F)|_{\Gamma_2} ), \qquad & & \\
				\mathcal{L}_{k, m}^{[2]}(F) &:= \mathcal{L}_{k, m-1}^{[2]}(F) + \mathcal{M}_{m, k+1}^{[2]}( f_2^m - \partial_{\unitvec{n}}^{m} \mathcal{L}_{k, m-1}^{[2]}(F)|_{\Gamma_2} ), \qquad & &1 \leq m \leq k, \\
				\mathcal{L}_k^{[2]}(F) &:= \mathcal{L}_{k, k}^{[2]}(F). \qquad & &
			\end{alignat}
		\end{subequations}
		Then, for all $(s, p) \in \mathcal{A}_k$ and $F \in \Tr_k^{s, p}(\Gamma_{\mathcal{S}})$, $\mathcal{L}_{k}^{[2]}(F)$ is well-defined and there holds
		\begin{align}
			\label{eq:l2-interp-continuity-high}
			\partial_{\unitvec{n}}^m \mathcal{L}_{k}^{[2]}(F)|_{\Gamma_j} = f_j^m, \ \ 0 \leq m \leq k, \ j \in \mathcal{S}, \ \ \text{and} \ \ 	\| \mathcal{L}_{k}^{[2]}(F) \|_{s, p, \reftet} \lesssim_{b, k, s, p} 
			\| F \|_{\Tr_k^{s, p}, \Gamma_{\mathcal{S}}}.
		\end{align}
		Moreover, if for some $N \in \mathbb{N}_0$, $F$ satisfies \cref{eq:review:sigmak-cont-poly}, then $\mathcal{L}_{k}^{[2]}(F) \in \mathcal{P}_N(\reftet)$.
	\end{lemma}
	\begin{proof}
		 Let $k \in \mathbb{N}_0$, $(s, p) \in \mathcal{A}_k$, and $f \in \Tr_k^{s, p}(\mathcal{S})$ be given. 
		 
		 \noindent \textbf{Step 1: $m = 0$.} Thanks to \cref{lem:l1}, the traces $G = (g^0, g^1, \ldots, g^k)$ given by 
		\begin{align*}
			g_i^{l} := f_i^l - \partial_{\unitvec{n}} \mathcal{L}_{k}^{[1]}(F)|_{\Gamma_i}, \qquad 0 \leq l \leq k, \ 1 \leq i \leq 2,
		\end{align*}
		satisfy the hypotheses of \cref{lem:vanish-trace-reg} with $(i, j) = (1, 2)$ and $l=1$. Thanks to \cref{lem:vanish-trace-reg} and \cref{cor:mmr-gamma2-cont-high}, $\mathcal{M}_{0, k+1}^{[2]}(g_2^0)$, and hence $\mathcal{L}_{k, 0}^{[2]}(F)$, is well-defined with
		\begin{align*}
			\| \mathcal{L}_{k, 0}^{[2]}(F) \|_{s, p, \reftet} \lesssim_{b,k,s,p} \| \mathcal{L}_k^{[1]}(F)\|_{s, p, \reftet} +  \begin{cases} \hspace{-0.15cm}
				\inorm{\gamma_{12}}{f_2^0 - \mathcal{L}_k^{[1]}(F)}{s-\frac{1}{p}, p, \Gamma_2} & \hspace{-0.2cm} \text{if } s \leq k + 1 + \frac{1}{p}, \\
				\| f_2^0 - \mathcal{L}_k^{[1]}(F) \|_{s-\frac{1}{p}, p, \Gamma_2} & \hspace{-0.2cm} \text{if } s > k + 1 + \frac{1}{p}.
			\end{cases} 
		\end{align*} 
		Applying \cref{eq:l1-gamma1-interp-continuity} and \cref{eq:mmr2-gamma1-interp} gives
		\begin{align*}
			 \partial_{\unitvec{n}}^{l} \mathcal{L}_{k, 0}^{[2]}(F)|_{\Gamma_1} = f_1^l, \qquad 0 \leq l \leq k, \quad \text{and} \quad \mathcal{L}_{k, 0}^{[2]}(F)|_{\Gamma_2} = f_2^0,
		\end{align*}
		and applying \cref{eq:l1-gamma1-interp-continuity} and \cref{eq:trace-math-reg} gives
		\begin{align*}
			\| \mathcal{L}_{k, 0}^{[2]}(F) \|_{s, p, \reftet} \lesssim_{b,k,s,p} \|F\|_{\Tr_k^{s, p}, \Gamma_{\mathcal{S}}}.
		\end{align*}
		Moreover, if $F$ satisfies \cref{eq:review:sigmak-cont-poly} and for some $N \in \mathbb{N}_0$, then $F \in \Tr^{s, p}_k(\Gamma_{\mathcal{S}})$ by \cref{lem:review:polys-trace-space} and $\mathcal{L}_k^{[1]}(F) \in \mathcal{P}_N(\reftet)$ by \cref{lem:l1}. Thus, the trace $G$ satisfies \cref{eq:review:sigmak-cont-poly} for $\{i, j\} \subseteq \{1,2\}$ and $G \in \Tr^{s, p}_k(\Gamma_{\mathcal{S}})$ for all $(s, p) \in \mathcal{A}_k$. By \cref{lem:vanish-trace-reg}, $g_2^0 \in W_{\gamma_{12}}^{k + 1, p}(\Gamma_2)$ for all $p \in (1,\infty)$, and so $D_{\Gamma}^l g^0_2|_{\gamma_{12}} = 0$ for $0 \leq l \leq k$. Thanks to \cref{cor:mmr-gamma2-cont-high}, $\mathcal{L}_{k, 0}^{[2]}(F) \in \mathcal{P}_{N}(\reftet)$.
		
		\noindent \textbf{Step 2: Induction on $m$. } Assume that for some $m$ such that $0 \leq m \leq k-1$, $\mathcal{L}^{[2]}_{k, m}(F)$ is well-defined and satisfies
		\begin{align}
			\label{eq:proof:l2km-interp}
			\partial_{\unitvec{n}}^{l} \mathcal{L}_{k, m}^{[2]}(F)|_{\Gamma_1} = f_1^l, \qquad 0 \leq l \leq k, \qquad \partial_{\unitvec{n}}^{l} \mathcal{L}_{k, m}^{[2]}(F)|_{\Gamma_2} = f_2^l, \qquad 0 \leq l \leq m,
		\end{align}
		and
		\begin{align}
			\label{eq:proof:l2km-cont}
			\| \mathcal{L}_{k, m}^{[2]}(F) \|_{s, p, \reftet} \lesssim_{b,k,s,p} \|F\|_{\Tr_k^{s, p}, \Gamma_{\mathcal{S}}}.
		\end{align}
		Additionally assume that if $F$ satisfies \cref{eq:review:sigmak-cont-poly} for $\{i, j\} \subseteq \{1,2\}$ and for some $N \in \mathbb{N}_0$, then $\mathcal{L}_{k, m}^{[2]}(F) \in \mathcal{P}_N(\reftet)$.
		
		Thanks to \cref{eq:proof:l2km-interp}, the traces $G = (g^0, g^1, \ldots, g^k)$ given by 
		\begin{align*}
			g_i^{l} := f_i^l - \partial_{\unitvec{n}}^{l} \mathcal{L}_{k, m}^{[2]}(F)|_{\Gamma_i}, \qquad 0 \leq l \leq k, \ 1 \leq i \leq 2,
		\end{align*}
		satisfy the hypotheses of \cref{lem:vanish-trace-reg} with $(i, j) = (1, 2)$ and $l=m+1$. Thanks to \cref{lem:vanish-trace-reg} and \cref{cor:mmr-gamma2-cont-high}, $\mathcal{M}_{m+1, k+1}^{[2]}(g_2^{m+1})$, and hence $\mathcal{L}_{k, m+1}^{[2]}(F)$ is well-defined with
		\begin{multline*}
			\| \mathcal{L}_{k, m+1}^{[2]}(F) \|_{s, p, \reftet} \lesssim_{b,k,s,p} \| \mathcal{L}_{k,m}^{[2]}(F)\|_{s, p, \reftet} \\
			 + \begin{cases}
				\inorm{\gamma_{12}}{f_2^{m+1} - \partial_{\unitvec{n}}^{m+1} \mathcal{L}_{k,m}^{[2]}(F)}{s-m-1-\frac{1}{p}, p, \Gamma_2} & \text{if } s - m - 1 \leq k + 1 + \frac{1}{p}, \\
				\| f_2^{m+1} - \partial_{\unitvec{n}}^{m+1} \mathcal{L}_{k,m}^{[2]}(F) \|_{s-m-1-\frac{1}{p}, p, \Gamma_2} & \text{if } s - m - 1 > k + 1 + \frac{1}{p}.
			\end{cases} 
		\end{multline*} 
		Applying \cref{eq:proof:l2km-interp} and \cref{eq:mmr2-gamma1-interp} gives \cref{eq:proof:l2km-interp} for $m+1$, while applying \cref{eq:proof:l2km-cont} and \cref{eq:trace-math-reg} gives \cref{eq:proof:l2km-cont} for $m+1$.
		
		Moreover, if $F$ satisfies \cref{eq:review:sigmak-cont-poly} for some $N \in \mathbb{N}_0$, then $\mathcal{L}_{k, m}^{[2]}(F) \in \mathcal{P}_N(\reftet)$ by assumption and so the trace $G$ satisfies \cref{eq:review:sigmak-cont-poly} and $G \in \Tr^{s, p}_k(\Gamma_{\mathcal{S}})$ for all $(s, p) \in \mathcal{A}_k$. By \cref{lem:vanish-trace-reg}, $g_2^{m+1} \in W_{\gamma_{12}}^{k + 1, p}(\Gamma_2)$ for all $p \in (1,\infty)$, and so $D_{\Gamma}^l g^{m+1}_2|_{\gamma_{12}} = 0$ for $0 \leq l \leq k$. Thanks to \cref{cor:mmr-gamma2-cont-high}, $\mathcal{L}_{k, m+1}^{[2]}(F) \in \mathcal{P}_{N}(\reftet)$.
	\end{proof}
	
	\subsection{Lifting from three faces}
	
	We continue in the spirit of the previous two sections and define another lifting operator from $\Gamma_1$ with vanishing traces on $\Gamma_2$ and $\Gamma_3$. 	Given nonnegative integers $k, r \in \mathbb{N}_0$, a smooth compactly supported function $b \in C_c^{\infty}(\reftri)$, and a function $f : \reftri \to \mathbb{R}$, we define the operator $\mathcal{S}_{k, r}^{[1]}$ formally by the rule
	\begin{align}
		\label{eq:skr-def}	
		\begin{aligned}
			\mathcal{S}_{k, r}^{[1]}(f)(\bdd{x}, z) &:= (x_1 x_2)^r \mathcal{E}_k^{[1]}((\omega_1 \omega_2)^{-r} f)(\bdd{x}, z) \\
			&= (x_1 x_2)^r \frac{(-z)^k}{k!} \int_{\reftri} \frac{ b(\bdd{y}) f(\bdd{x} + z \bdd{y})}{((x_1 + z y_1)(x_2 + z y_2))^r} \d{\bdd{y}} \qquad \forall (\bdd{x}, z) \in \reftet.
		\end{aligned}
	\end{align}
	Note that when $r = 0$, we have $\mathcal{S}_{k, 0}^{[1]} = \mathcal{E}_k^{[1]}$. For functions $f : \Gamma_1 \to \mathbb{R}$, we again abuse the notation and set $\mathcal{S}_{k, r}^{[1]}(f) := \mathcal{S}_{k, r}^{[1]}(f \circ \mathfrak{I}_1)$.
	
	We require one additional family of spaces with vanishing traces. Let $s = m + \sigma$ with $m \in \mathbb{N}_0$ and $\sigma \in [0, 1)$ and $1 < p < \infty$. Given $r \in \mathbb{N}$, a face $\Gamma_j$, $1 \leq j \leq 4$, and $\mathfrak{E}$ a subset of the edges of $\Gamma_j$, we define the following subspaces of $W^{s, p}_{\mathfrak{E}}(\Gamma_j)$:
	\begin{align}
		\label{eq:wsp-er-def}
		W^{s, p}_{\mathfrak{E}, r}(\Gamma_j) := \left\{ f \in W^{s, p}(\Gamma_j) \cap W^{\min\{s, r\}, p}_{\mathfrak{E}}(\Gamma_j) : \inorm{\mathfrak{E}, r}{f}{s, p, \reftri} < \infty \right\},
	\end{align}
	where the norm on $W^{s, p}_{\mathfrak{E}, r}(\Gamma_j)$ is given by
	\begin{align*}
		\begin{aligned}
			\inormsup{\mathfrak{E}, r}{f}{s, p, \Gamma_j}{p} := &\begin{cases}
				\inormsup{\mathfrak{E}}{f}{s, p, \reftri}{p}  & \text{if } s \leq r, \\
				\|f\|_{s, p, \reftri}^p & \text{if } s > r,
			\end{cases} \\
			&+ \begin{cases}
				\left\| \dist(\cdot, \bigcup_{\gamma \in \mathfrak{E}} \gamma )^{-\sigma} \frac{\partial^{m-r+1} D_{\Gamma}^{r-1} f}{ \partial \unitvec{t}_{\gamma}^{m-r+1} }  \right\|_{p, \Gamma_j}^p & \text{if } s > r, \ \sigma p = 1, \mathfrak{E} \neq \emptyset, \\
				0 & \text{otherwise},
			\end{cases}
		\end{aligned}
	\end{align*}
	where $\unitvec{t}_{\gamma}$ is a unit-tangent vector on the edge $\gamma \in \mathfrak{E}$. For $r = 0$, we set $W_{\mathfrak{E}, 0}^{s, p}(\Gamma_j) := W^{s, p}(\Gamma_j)$. When $\mathfrak{E}$ consists of only one element $\gamma$, we set $W_{\gamma, r}^{s, p}(\Gamma_j) := W_{\mathfrak{E}, r}^{s, p}(\Gamma_j)$ and $\inorm{\gamma, r}{f}{s, p, \Gamma_j} := \inorm{\mathfrak{E}, r}{f}{s, p, \Gamma_j}$. One can again verify that $W_{\mathfrak{E}, r}^{s, p}(\Gamma_j)$ are Banach spaces and that the following analogue of \cref{eq:wsp-e-intersection-id} holds:
	\begin{align}
		\label{eq:wsp-e-intersection-id-2}
		W_{\mathfrak{E}, r}^{s, p}(\Gamma_j) = \bigcap_{\gamma \in \mathfrak{E}} W_{\gamma, r}^{s, p}(\Gamma_j) \quad \text{and} \quad \inorm{\mathfrak{E}, r}{f}{s, p, \Gamma_j} \approx_{s,p} \sum_{\gamma \in \mathfrak{E}} \inorm{\gamma, r}{f}{s, p, \Gamma_j}.
	\end{align}
	
	 The following result shows that the continuity of $\mathcal{S}_{k, r}^{[1]}$ can be characterized with these spaces.	
	\begin{lemma}
		\label{lem:smr-gamma1-cont-high}
		Let $b \in C_c^{\infty}(\reftri)$, $k, r \in \mathbb{N}_0$, $(s, p) \in \mathcal{A}_k$, and $\mathfrak{E} = \{ \gamma_{12}, \gamma_{13} \}$. Then, for all $f \in W^{s-k-\frac{1}{p}, p}_{\mathfrak{E}, r}(\Gamma_1)$, there holds
		\begin{subequations}
			\begin{alignat}{2}
				\label{eq:smr-gamma1-interp}
				\partial_{\unitvec{n}}^m \mathcal{S}_{k, r}^{[1]}(f)|_{\Gamma_1} &= \delta_{km} \left( \int_{\reftri} b(\bdd{x}) \d{\bdd{x}} \right) f, \qquad & &0 \leq m \leq k, \\
				\label{eq:smr-gamma2-zero}
				\partial_{\unitvec{n}}^{j} \mathcal{S}_{k, r} ^{[1]}(f)|_{\Gamma_i} &= 0, \qquad & &0 \leq j < \min\left\{ r, s - \frac{1}{p} \right\}, \ 2 \leq i \leq 3,
			\end{alignat}
		\end{subequations}
		and 
		\begin{align}
			\label{eq:smr-gamma1-continuity-high}
			\| \mathcal{S}_{k, r}^{[1]}(f) \|_{s, p, \reftet} \lesssim_{b, k, r, s, p} \inorm{\mathfrak{E}, r}{f}{s-k-\frac{1}{p}, p, \Gamma_1}.
		\end{align}
		Moreover, if $f \in \mathcal{P}_N(\Gamma_1)$, $N \in \mathbb{N}_0$, satisfies $D_{\Gamma}^{l} f|_{\gamma_{12}} = D_{\Gamma}^{l} f|_{\gamma_{13}} = 0$ for $0 \leq l \leq r-1$, then $\mathcal{S}_{k, r}^{[1]}(f) \in \mathcal{P}_{N+k}(\reftet)$.
	\end{lemma}
	The proof of \cref{lem:smr-gamma1-cont-high} appears in \cref{sec:proof-lem:smr-gamma1-cont-high}. We define the analogous operator associated to $\Gamma_3$ as follows.
	\begin{align*}
		\mathcal{S}_{k, r}^{[3]}(f)(\bdd{x}, z) := \mathcal{S}_{k, r}^{[1]}(f \circ \mathfrak{I}_3) \circ \mathfrak{R}_{13}(\bdd{x}, z) \qquad \forall (\bdd{x}, z) \in \reftet,
	\end{align*}
	where $\mathfrak{I}_3(\bdd{x}) := (0, x_2, x_1)$ and $\mathfrak{R}_{13}(\bdd{x}, z) := (z, x_2, x_1)$ for all $(\bdd{x}, z) \in \reftet$. Thanks to the chain rule and the smoothness of the mappings $\mathfrak{I}_3$ and $\mathfrak{R}_{13}$, the continuity and interpolation properties of $\mathcal{S}_{k, r}^{[3]}$ follow immediately from \cref{lem:smr-gamma1-cont-high}.
	\begin{corollary}
		\label{cor:smr-gamma3-cont-high}
		Let $b \in C_c^{\infty}(\reftri)$, $k, r \in \mathbb{N}_0$, $(s, p) \in \mathcal{A}_k$, and $\mathfrak{E} = \{ \gamma_{13}, \gamma_{23} \}$. Then, for all $f \in W^{s-k-\frac{1}{p}, p}_{\mathfrak{E}, r}(\Gamma_3)$, there holds
		\begin{subequations}
			\begin{alignat}{2}
				\label{eq:smr3-gamma3-interp}
				\partial_{\unitvec{n}}^m \mathcal{S}_{k, r}^{[3]}(f)|_{\Gamma_3} &= \delta_{km} \left( \int_{\reftri} b(\bdd{x}) \d{\bdd{x}} \right) f, \qquad & &0 \leq m \leq k, \\
				\label{eq:smr3-gamma12-zero}
				\partial_{\unitvec{n}}^{j} \mathcal{S}_{k, r} ^{[3]}(f)|_{\Gamma_i} &= 0, \qquad & &0 \leq j < \min\left\{ r, s - \frac{1}{p} \right\}, \ 1 \leq i \leq 2
			\end{alignat}
		\end{subequations}
		and 
		\begin{align}
			\label{eq:smr3-gamma3-continuity-high}
			\| \mathcal{S}_{k, r}^{[3]}(f) \|_{s, p, \reftet} \lesssim_{b, k, r, s, p} \inorm{\mathfrak{E}, r}{f}{s-k-\frac{1}{p}, p, \Gamma_3}.
		\end{align}
		Moreover, if $f \in \mathcal{P}_N(\Gamma_3)$, $N \in \mathbb{N}_0$, satisfies $D_{\Gamma}^{l} f|_{\gamma_{13}} = D_{\Gamma}^{l} f|_{\gamma_{23}} = 0$ for $0 \leq l \leq r-1$, then $\mathcal{S}_{k, r}^{[3]}(f) \in \mathcal{P}_{N+k}(\reftet)$.
	\end{corollary}

	We also have the following analogue of \cref{lem:vanish-trace-reg}.
	\begin{lemma}
		\label{lem:vanish-trace-reg-2}
		Let $k \in \mathbb{N}_0$, $(s, p) \in \mathcal{A}_k$, $1 \leq l \leq k$, and $1 \leq i, j \leq 4$ with $i \neq j$ be given. Suppose that $F = (f^0, f^1, \ldots, f^k) \in \Tr_k^{s,p}(\Gamma_i \cup \Gamma_j)$ satisfies
		\begin{enumerate}
			\item[(i)] $F = (0, 0, \ldots, 0)$ on $\Gamma_i$;
			\item[(ii)] $f_j^m = 0$ on $\Gamma_j$ for $0 \leq m \leq l-1$.
		\end{enumerate}
		Then, there holds $f^l_j \in W_{\gamma_{ij}, k+1}^{s - l - \frac{1}{p}, p}(\Gamma_j)$ and
		\begin{align}
			\label{eq:trace-math-reg-2}
			\inorm{\gamma_{ij}, k+1}{f_j^l}{s-l-\frac{1}{p}, p, \Gamma_{j}} &\lesssim_{k, s, p} \|F\|_{\Tr_k^{s, p}, \Gamma_i \cup \Gamma_j}. 
		\end{align}
	\end{lemma}	
	\begin{proof}
		The result follows from applying inequality \cref{eq:proof:trace-match-dist-integral} and  identity \cref{eq:proof:trace-id-high}.
	\end{proof}

	We now construct the lifting operator $\mathcal{L}_{k}^{[3]}$ in the same fashion as $\mathcal{L}_k^{[2]}$ \cref{eq:l2-def}, replacing the use of $\mathcal{M}_{m, k}^{[2]}$ with $\mathcal{S}_{m, k}^{[3]}$.	
	\begin{lemma}
		\label{lem:l3}
		Let $b \in C_c^{\infty}(\reftri)$ with $\int_{\reftri} b(\bdd{x}) \d{\bdd{x}} = 1$, $k \in \mathbb{N}_0$, and $\mathcal{S} = \{1,2,3\}$. For $F = (f^0, f^1, \ldots, f^k) \in L^p(\Gamma_{\mathcal{S}})^{k+1}$, we formally define the following operators:
		\begin{subequations}
			\label{eq:l3-def}
			\begin{alignat}{2}
				\mathcal{L}_{k, 0}^{[3]}(F) &:= \mathcal{L}_k^{[2]}(F) + \mathcal{S}_{0, k+1}^{[3]}(f_3^0 - \mathcal{L}_k^{[2]}(F)|_{\Gamma_3} ), \qquad & & \\
				\mathcal{L}_{k, m}^{[3]}(F) &:= \mathcal{L}_{k, m-1}^{[3]}(F) + \mathcal{S}_{m, k+1}^{[3]}( f_3^m - \partial_{\unitvec{n}}^{m} \mathcal{L}_{k, m-1}^{[3]}(F)|_{\Gamma_3} ), \qquad & &1 \leq m \leq k, \\
				\mathcal{L}_{k}^{[3]}(F) &:= \mathcal{L}_{k, k}^{[3]}(F).
			\end{alignat}
		\end{subequations}
		Then, for all $(s, p) \in \mathcal{A}_k$ and $F \in \Tr_k^{s, p}(\Gamma_{\mathcal{S}})$, there holds
		\begin{align}
			\label{eq:l3-interp-continuity-high}
			\partial_{\unitvec{n}}^m \mathcal{L}_k^{[3]}(F)|_{\Gamma_j} = f_j^m, \quad 0 \leq m \leq k, \ j \in \mathcal{S}, \quad \text{and} \quad  \| \mathcal{L}_k^{[3]}(F) \|_{s, p, \reftet} \lesssim_{b, k, s, p} 
			\| F \|_{\Tr_k^{s, p}, \Gamma_{\mathcal{S}}}.
		\end{align}		
		Moreover, if $F$ satisfies \cref{eq:review:sigmak-cont-poly} and for some $N \in \mathbb{N}_0$, then $\mathcal{L}_{k}^{[3]}(F) \in \mathcal{P}_N(\reftet)$.
	\end{lemma}
	\begin{proof}
		Let $k \in \mathbb{N}_0$, $(s, p) \in \mathcal{A}_k$, and $f \in \Tr_k^{s, p}(\Gamma_{\mathcal{S}})$ be given. Let $\mathfrak{E} = \{ \gamma_{13}, \gamma_{23} \}$. 
		
		\noindent \textbf{Step 1: $m = 0$.} Thanks to \cref{lem:l2}, the traces $G = (g^0, g^1, \ldots, g^k)$ given by 
		\begin{align*}
			g_i^{l} := f_i^l - \partial_{\unitvec{n}} \mathcal{L}_{k}^{[2]}(F)|_{\Gamma_i}, \qquad 0 \leq l \leq k, \ 1 \leq i \leq 3,
		\end{align*}
		satisfy the hypotheses of \cref{lem:vanish-trace-reg-2} with $(i, j) \in \{  (1, 3), (2,3)\}$ and $l=1$. Thanks to \cref{eq:wsp-e-intersection-id-2} and \cref{lem:vanish-trace-reg-2}, $g^0_3 \in W_{\mathfrak{E}, k+1}^{s - \frac{1}{p}, p}(\Gamma_3)$. Consequently, $\mathcal{S}_{0, k+1}^{[3]}(g_3^0)$ is well-defined by \cref{cor:smr-gamma3-cont-high}, and hence $\mathcal{L}_{k, 0}^{[3]}(F)$ is well-defined with
		\begin{align*}
			\| \mathcal{L}_{k, 0}^{[3]}(F) \|_{s, p, \reftet} \lesssim_{b,k,s,p} \| \mathcal{L}_k^{[2]}(F)\|_{s, p, \reftet} + \inorm{\mathfrak{E}, k+1}{f_3^0 - \mathcal{L}_k^{[2]}(F)}{s-\frac{1}{p}, p, \Gamma_3}.
		\end{align*} 
		Applying \cref{eq:l2-interp-continuity-high} and \cref{eq:smr3-gamma3-interp} gives
		\begin{align*}
			\partial_{\unitvec{n}}^{l} \mathcal{L}_{k, 0}^{[3]}(F)|_{\Gamma_i} = f_i^l, \qquad 0 \leq l \leq k, \ 1 \leq i \leq 2, \qquad \mathcal{L}_{k, 0}^{[3]}(F)|_{\Gamma_3} = f_3^0,
		\end{align*}
		and applying \cref{eq:wsp-e-intersection-id}, \cref{eq:l2-interp-continuity-high}, and \cref{eq:trace-math-reg} gives
		\begin{align*}
			\| \mathcal{L}_{k, 0}^{[3]}(F) \|_{s, p, \reftet} \lesssim_{b,k,s,p} \|F\|_{\Tr_k^{s, p}, \Gamma_{\mathcal{S}}}.
		\end{align*}
		Moreover, if $F$ satisfies \cref{eq:review:sigmak-cont-poly} for some $N \in \mathbb{N}_0$, then $\mathcal{L}_k^{[2]}(F) \in \mathcal{P}_N(\reftet)$ by \cref{lem:l2}, and so the trace $G$ satisfies \cref{eq:review:sigmak-cont-poly} and $G \in \Tr^{s, p}_k(\Gamma_{\mathcal{S}})$ for all $(s, p) \in \mathcal{A}_k$ thanks to \cref{lem:review:polys-trace-space}. By \cref{eq:wsp-e-intersection-id-2} and \cref{lem:vanish-trace-reg-2}, $g_3^0 \in W_{\mathfrak{E}}^{k + 1, p}(\Gamma_3)$ for all $p \in (1, \infty)$, and so $D_{\Gamma}^l g^0_3|_{\gamma_{13}} = D_{\Gamma}^l g^0_3|_{\gamma_{23}} = 0$ for $0 \leq l \leq k$. Thanks to \cref{cor:smr-gamma3-cont-high}, $\mathcal{L}_{k, 0}^{[3]}(F) \in \mathcal{P}_{N}(\reftet)$.
		
		\noindent \textbf{Step 2: Induction on $m$. } Assume that for some $m$ such that $0 \leq m \leq k-1$, $\mathcal{L}^{[3]}_{k, m}(F)$ is well-defined and satisfies
		\begin{subequations}
			\label{eq:proof:l3km-interp}
			\begin{alignat}{2}
				\partial_{\unitvec{n}}^{l} \mathcal{L}_{k, m}^{[3]}(F)|_{\Gamma_i} &= f_i^l, \qquad & &0 \leq l \leq k, \ 1 \leq i \leq 2, \\ \partial_{\unitvec{n}}^{l} \mathcal{L}_{k, m}^{[3]}(F)|_{\Gamma_3} &= f_3^l, \qquad & &0 \leq l \leq m,
			\end{alignat}
		\end{subequations}
		and
		\begin{align}
			\label{eq:proof:l3km-cont}
			\| \mathcal{L}_{k, m}^{[3]}(F) \|_{s, p, \reftet} \lesssim_{b,k,s,p} \|F\|_{\Tr_k^{s, p}, \Gamma_{\mathcal{S}}}.
		\end{align}
		Additionally assume that if $F$ satisfies \cref{eq:review:sigmak-cont-poly}  for some $N \in \mathbb{N}_0$, then $\mathcal{L}_{k, m}^{[3]}(F) \in \mathcal{P}_N(\reftet)$.
		
		Thanks to \cref{eq:proof:l3km-interp}, the traces $G = (g^0, g^1, \ldots, g^k)$ given by 
		\begin{align*}
			g_i^{l} := f_i^l - \partial_{\unitvec{n}}^{l} \mathcal{L}_{k, m}^{[3]}(F)|_{\Gamma_i}, \qquad 0 \leq l \leq k, \ 1 \leq i \leq 3,
		\end{align*}
		satisfy the hypotheses of \cref{lem:vanish-trace-reg-2} with $(i, j) \in \{  (1, 3), (2,3)\}$ and $l=m+1$. Thanks to \cref{eq:wsp-e-intersection-id-2} and \cref{lem:vanish-trace-reg-2}, there holds $g^{m+1}_3 \in W_{\mathfrak{E},k+1}^{s-m-1-\frac{1}{p},  p}(\Gamma_3)$. Consequently, $\mathcal{S}_{m+1, k+1}^{[3]}(g_3^{m+1})$ is well-defined by \cref{cor:smr-gamma3-cont-high}, and hence $\mathcal{L}_{k, m+1}^{[3]}(F)$ is well-defined with
		\begin{multline*}
			\| \mathcal{L}_{k, m+1}^{[3]}(F) \|_{s, p, \reftet} \lesssim_{b,k,s,p} \| \mathcal{L}_{k,m}^{[3]}(F)\|_{s, p, \reftet} \\
			+ \inorm{\mathfrak{E},k+1}{f_3^{m+1} - \partial_{\unitvec{n}}^{m+1}\mathcal{L}_{k,m}^{[3]}(F)}{s-m-1-\frac{1}{p}, p, \Gamma_3}.
		\end{multline*} 
		Applying \cref{eq:proof:l3km-interp} and \cref{eq:smr3-gamma3-interp} gives \cref{eq:proof:l3km-interp} for $m+1$, while applying \cref{eq:wsp-e-intersection-id-2}, \cref{eq:proof:l3km-cont}, and \cref{eq:trace-math-reg-2} gives \cref{eq:proof:l3km-cont} for $m+1$.
		
		Moreover, if $F$ satisfies \cref{eq:review:sigmak-cont-poly} for some $N \in \mathbb{N}_0$, then $\mathcal{L}_{k, m}^{[3]}(F) \in \mathcal{P}_N(\reftet)$ by assumption and so the trace $G$ satisfies \cref{eq:review:sigmak-cont-poly} and $G \in \Tr^{s, p}_k(\Gamma_{\mathcal{S}})$ for all $(s, p) \in \mathcal{A}_k$ thanks to \cref{lem:review:polys-trace-space}. By \cref{eq:wsp-e-intersection-id} and \cref{lem:vanish-trace-reg}, $g_3^{m+1} \in W_{\mathfrak{E}}^{k + 1, p}(\Gamma_3)$ for all $p \in (1,\infty)$, and so $D_{\Gamma}^l g^{m+1}_3|_{\gamma_{13}} = D_{\Gamma}^l g_3^{m+1}|_{\gamma_{23}} = 0$ for $0 \leq l \leq k$. Thanks to \cref{cor:smr-gamma3-cont-high}, $\mathcal{L}_{k, m+1}^{[3]}(F) \in \mathcal{P}_{N}(\reftet)$.
	\end{proof}
	
	\subsection{Lifting from four faces}	
	\label{sec:four-faces}
	
	To complete the construction of the lifting operator from the entire boundary, we
	define one final single face lifting operator from $\Gamma_1$ that vanishes on the remaining faces. Given nonnegative integers $k, r \in \mathbb{N}_0$, a smooth compactly supported function $b \in C_c^{\infty}(\reftri)$, and a function $f : \reftri \to \mathbb{R}$, we define the operator $\mathcal{R}_{k, r}^{[1]}$ formally by the rule
	\begin{align}
		\label{eq:rkr-def}
		\begin{aligned}
				&\mathcal{R}_{k, r}^{[1]}(f)(\bdd{x}, z)  \\
			&\quad := (x_1 x_2 (1-x_1-x_2-z))^r \mathcal{E}_k^{[1]}((\omega_1 \omega_2 \omega_3)^{-r} f)(\bdd{x}, z)  \\
			&\quad = (x_1 x_2 (1-x_1-x_2-z))^r \frac{(-z)^k}{k!} \int_{\reftri} \left. \frac{ b(\bdd{y}) f(\bdd{w}) \d{\bdd{y}}}{(\omega_1 \omega_2 \omega_3)^r(\bdd{w})} \right|_{\bdd{w} = \bdd{x} + z\bdd{y}}  \ \forall (\bdd{x}, z) \in \reftet.
		\end{aligned}
	\end{align}
	Note that when $r = 0$, we have $\mathcal{R}_{k, r}^{[1]} = \mathcal{E}_k^{[1]}$. For functions $f : \Gamma_1 \to \mathbb{R}$, we again abuse notation and set $\mathcal{R}_{k, r}^{[1]}(f) := \mathcal{R}_{k, r}^{[1]}(f \circ \mathfrak{I}_1)$. The weighted spaces $W_{\mathfrak{E}, r}^{s, p}(\Gamma_1)$ again play a role in the continuity of $\mathcal{R}_{k, r}^{[1]}$ as the following result shows.
	\begin{lemma}
		\label{lem:rmr-gamma1-cont-high}
		Let $b \in C_c^{\infty}(\reftri)$, $k, r \in \mathbb{N}_0$, $(s, p) \in \mathcal{A}_k$, and $\mathfrak{E} = \{ \gamma_{12}, \gamma_{13}, \gamma_{14} \}$. Then, for all $f \in W_{\mathfrak{E}, r}^{s-k - \frac{1}{p}, p}(\Gamma_1)$, there holds
		\begin{subequations}
			\begin{alignat}{2}
				\label{eq:rmr-gamma1-interp}
				\partial_{\unitvec{n}}^m \mathcal{R}_{k, r}^{[1]}(f)|_{\Gamma_1} &= \delta_{km} \left( \int_{\reftri} b(\bdd{x}) \d{\bdd{x}} \right) f, \qquad & &0 \leq m \leq k, \\
				\label{eq:rmr-gamma2-zero}
				\partial_{\unitvec{n}}^{j} \mathcal{R}_{k, r} ^{[1]}(f)|_{\Gamma_i} &= 0, \qquad & &0 \leq j < \min\left\{ r, s - \frac{1}{p} \right\}, \ 2 \leq i \leq 4,
			\end{alignat}
		\end{subequations}
		and 
		\begin{align}
			\label{eq:rmr-gamma1-continuity-high}
			\| \mathcal{R}_{k, r}^{[1]}(f) \|_{s, p, \reftet} \lesssim_{b, k, r, s, p} 	\inorm{\mathfrak{E},r}{f}{s-k-\frac{1}{p}, p, \Gamma_1}.
		\end{align}
		Moreover, if $f \in \mathcal{P}_N(\Gamma_1)$, $N \in \mathbb{N}_0$, satisfies $D_{\Gamma}^{l} f|_{\partial \Gamma_1} = 0$ for $0 \leq l \leq r-1$, then $\mathcal{R}_{k, r}^{[1]}(f) \in \mathcal{P}_{N+k}(\reftet)$.
	\end{lemma}
	The proof of \cref{lem:rmr-gamma1-cont-high} appears in \cref{sec:proof-lem:rmr-gamma1-cont-high}. The analogous operator associated to $\Gamma_4$ is given by
	\begin{align*}
		\mathcal{R}_{k, r}^{[4]}(f)(\bdd{x}, z) := 3^{-\frac{k}{2}} \mathcal{R}_{k, r}^{[1]}(f \circ \mathfrak{I}_4) \circ \mathfrak{R}_{14}(\bdd{x}, z) \qquad \forall (\bdd{x}, z) \in \reftet,
	\end{align*}
	where $\mathfrak{I}_4(\bdd{x}) :=  (x_1, x_2, 1 - x_1 - x_2)$ and $\mathfrak{R}_{14}(\bdd{x}, z) := (x_1, x_2, 1-x_1-x_2-z)$ for all $(\bdd{x}, z) \in \reftet$. Thanks to the chain rule and the smoothness of the mappings $\mathfrak{I}_4$ and $\mathfrak{R}_{14}$, the continuity and interpolation properties of $\mathcal{R}_{k, r}^{[4]}$ follow immediately from \cref{lem:rmr-gamma1-cont-high}.
	\begin{corollary}
		\label{cor:rmr-gamma4-cont-high}
		Let $b \in C_c^{\infty}(\reftri)$, $k, r \in \mathbb{N}_0$, $(s, p) \in \mathcal{A}_k$, and $\mathfrak{E} = \{ \gamma_{14}, \gamma_{24}, \gamma_{34} \}$. Then, for all $f \in W_{\mathfrak{E}, r}^{s-k- \frac{1}{p},  p}(\Gamma_4)$, there holds
		\begin{subequations}
			\begin{alignat}{2}
				\label{eq:rmr4-gamma4-interp}
				\partial_{\unitvec{n}}^m \mathcal{R}_{k, r}^{[4]}(f)|_{\Gamma_4} &= \delta_{km} \left( \int_{\reftri} b(\bdd{x}) \d{\bdd{x}} \right) f, \qquad & &0 \leq m \leq k, \\
				\label{eq:rmr4-gamma14-zero}
				\partial_{\unitvec{n}}^{j} \mathcal{R}_{k, r} ^{[4]}(f)|_{\Gamma_i} &= 0, \qquad & &0 \leq j < \min\left\{ r, s - \frac{1}{p} \right\}, \ 1 \leq i \leq 3
			\end{alignat}
		\end{subequations}
		and 
		\begin{align}
			\label{eq:rmr4-gamma4-continuity-high}
			\| \mathcal{R}_{k, r}^{[4]}(f) \|_{s, p, \reftet} \lesssim_{b, k, r, s, p} \inorm{\mathfrak{E}, r}{f}{s-k-\frac{1}{p}, p, \Gamma_4}.
		\end{align}
		Moreover, if $f \in \mathcal{P}_N(\Gamma_4)$, $N \in \mathbb{N}_0$, satisfies $D_{\Gamma}^{l} f|_{\partial \Gamma_4} = 0$ for $0 \leq l \leq r-1$, then $\mathcal{R}_{k, r}^{[4]}(f) \in \mathcal{P}_{N+k}(\reftet)$.
	\end{corollary}
	
	Finally, we construct the lifting operator $\mathcal{L}_{k}^{[4]}$ in the same fashion as $\mathcal{L}_k^{[3]}$ \cref{eq:l3-def}, replacing the use of $\mathcal{S}_{m, k+1}^{[3]}$ with $\mathcal{R}_{m, k+1}^{[4]}$.	
	\begin{lemma}
		\label{lem:l4}
		Let $b \in C_c^{\infty}(\reftri)$ with $\int_{\reftri} b(\bdd{x}) \d{\bdd{x}} = 1$ and $k \in \mathbb{N}_0$. For $F = (f^0, f^1, \ldots, f^k) \in L^p(\partial K)^{k+1}$, we formally define the following operators:
		\begin{subequations}
			\label{eq:l4-def}
			\begin{alignat}{2}
				\mathcal{L}_{k, 0}^{[4]}(F) &:= \mathcal{L}_k^{[3]}(F) + \mathcal{R}_{0, k+1}^{[4]}(f_4^0 - \mathcal{L}_k^{[3]}(F)|_{\Gamma_4} ), \qquad & & \\
				\mathcal{L}_{k, m}^{[4]}(F) &:= \mathcal{L}_{k, m-1}^{[4]}(F) + \mathcal{R}_{m, k+1}^{[4]}( f_4^m - \partial_{\unitvec{n}}^{m} \mathcal{L}_{k,m-1}^{[4]}(F)|_{\Gamma_4} ), \qquad & &1 \leq m \leq k, \\
				\mathcal{L}_k^{[4]}(F) &:= \mathcal{L}_{k, k}^{[4]}(F). \qquad & &
			\end{alignat}
		\end{subequations}
		Then, for all $(s, p) \in \mathcal{A}_k$ and $F \in \Tr_k^{s, p}(\partial \reftet)$, there holds
		\begin{align}
			\label{eq:l4-interp-continuity-high}
			\partial_{\unitvec{n}}^m \mathcal{L}_k^{[4]}(F)|_{\partial \reftet} = f^m, \quad 0 \leq m \leq k,  \quad \text{and} \quad \| \mathcal{L}_k^{[4]}(F) \|_{s, p, \reftet} \lesssim_{b, k, s, p} 
			\| F \|_{\Tr_k^{s, p}, \partial \reftet}.
		\end{align}
		Moreover, if $F$ satisfies \cref{eq:review:sigmak-cont-poly} for some $N \in \mathbb{N}_0$, then $\mathcal{L}_k^{[4]}(F) \in \mathcal{P}_N(\reftet)$.
	\end{lemma}
	\begin{proof}
		Let $k \in \mathbb{N}_0$, $(s, p) \in \mathcal{A}_k$, and $f \in \Tr_k^{s, p}(\partial \reftet)$ be given and set $\mathfrak{E} := \{ \gamma_{14}, \gamma_{24}, \gamma_{34} \}$. 
		
		\noindent \textbf{Step 1: $m = 0$.} Thanks to \cref{lem:l3}, the traces $G = (g^0, g^1, \ldots, g^k)$ given by 
		\begin{align*}
			g_i^{l} := f_i^l - \partial_{\unitvec{n}} \mathcal{L}_{k}^{[3]}(F)|_{\Gamma_i}, \qquad 0 \leq l \leq k, \ 1 \leq i \leq 4,
		\end{align*}
		satisfies the hypotheses of \cref{lem:vanish-trace-reg-2} with $(i, j) \in \{  (1, 4), (2,4), (3,4)\}$ and $l=1$. Thanks to \cref{eq:wsp-e-intersection-id-2} and \cref{lem:vanish-trace-reg-2}, $g^0_4 \in W_{\mathfrak{E},k+1}^{ s-\frac{1}{p}, p}(\Gamma_4)$. Consequently, $\mathcal{R}_{0, k+1}^{[4]}(g_4^0)$ is well-defined by \cref{cor:rmr-gamma4-cont-high}, and hence $\mathcal{L}_{k, 0}^{[4]}(F)$ is well-defined with
		\begin{align*}
			\| \mathcal{L}_{k, 0}^{[4]}(F) \|_{s, p, \reftet} \lesssim_{b,k,s,p} \| \mathcal{L}_k^{[3]}(F)\|_{s, p, \reftet} + \inorm{\mathfrak{E},k+1}{f_4^0 - \mathcal{L}_k^{[3]}(F)}{s-\frac{1}{p}, p, \Gamma_4}.
		\end{align*} 
		Applying \cref{eq:l3-interp-continuity-high} and \cref{eq:rmr4-gamma4-interp} gives
		\begin{align*}
			\partial_{\unitvec{n}}^{l} \mathcal{L}_{k, 0}^{[4]}(F)|_{\Gamma_i} = f_i^l, \qquad 0 \leq l \leq k, \ 1 \leq i \leq 3, \qquad \mathcal{L}_{k, 0}^{[4]}(F)|_{\Gamma_3} = f_4^0,
		\end{align*}
		and applying \cref{eq:wsp-e-intersection-id-2}, \cref{eq:l3-interp-continuity-high}, and \cref{eq:trace-math-reg-2} gives
		\begin{align*}
			\| \mathcal{L}_{k, 0}^{[4]}(F) \|_{s, p, \reftet} \lesssim_{b,k,s,p} \|F\|_{\Tr_k^{s, p}, \partial \reftet}.
		\end{align*}
		Moreover, if $F$ satisfies \cref{eq:review:sigmak-cont-poly} for some $N \in \mathbb{N}_0$, then $\mathcal{L}_k^{[3]}(F) \in \mathcal{P}_N(\reftet)$ by \cref{lem:l2}, and so the trace $G$ satisfies \cref{eq:review:sigmak-cont-poly} and $G \in \Tr^{s, p}_k(\partial \reftet)$ for all $(s, p) \in \mathcal{A}_k$ thanks to \cref{lem:review:polys-trace-space}. By \cref{eq:wsp-e-intersection-id-2} and \cref{lem:vanish-trace-reg-2}, $g_4^0 \in W_{\mathfrak{E}}^{k + 1, p}(\Gamma_4)$ for all $p \in (1,\infty)$, and so $D_{\Gamma}^l g^0_4|_{\partial \Gamma_4} = 0$ for $0 \leq l \leq k$. Thanks to \cref{cor:rmr-gamma4-cont-high}, $\mathcal{L}_{k, 0}^{[4]}(F) \in \mathcal{P}_{N}(\reftet)$.
		
		\noindent \textbf{Step 2: Induction on $m$. } Assume that for some $0 \leq m \leq k-1$, $\mathcal{L}^{[4]}_{k, m}(F)$ is well-defined and satisfies
		\begin{subequations}
			\label{eq:proof:l4km-interp}
			\begin{alignat}{2}
				\partial_{\unitvec{n}}^{l} \mathcal{L}_{k, m}^{[4]}(F)|_{\Gamma_i} &= f_i^l, \qquad & & 0 \leq l \leq k, \ 1 \leq i \leq 3, \\
				\partial_{\unitvec{n}}^{l} \mathcal{L}_{k, m}^{[4]}(F)|_{\Gamma_4} &= f_4^l, \qquad & &0 \leq l \leq m,
			\end{alignat}
		\end{subequations}
		and
		\begin{align}
			\label{eq:proof:l4km-cont}
			\| \mathcal{L}_{k, m}^{[4]}(F) \|_{s, p, \reftet} \lesssim_{b,k,s,p} \|F\|_{\Tr_k^{s, p}, \partial \reftet}.
		\end{align}
		Additionally assume that if $F$ satisfies \cref{eq:review:sigmak-cont-poly} for some $N \in \mathbb{N}_0$, then $\mathcal{L}_{k, m}^{[4]}(F) \in \mathcal{P}_N(\reftet)$.
		
		Thanks to \cref{eq:proof:l4km-interp}, the traces $G = (g^0, g^1, \ldots, g^k)$ given by 
		\begin{align*}
			g_i^{l} := f_i^l - \partial_{\unitvec{n}}^{l} \mathcal{L}_{k, m}^{[4]}(F)|_{\Gamma_i}, \qquad 0 \leq l \leq k, \ 1 \leq i \leq 4,
		\end{align*}
		satisfies the hypotheses of \cref{lem:vanish-trace-reg-2} with $(i, j) \in \{  (1, 4), (2,4), (3,4) \}$ and $l=m+1$. Thanks to \cref{eq:wsp-e-intersection-id-2} and \cref{lem:vanish-trace-reg-2}, $g^{m+1}_4 \in W_{\mathfrak{E},k+1}^{s-m-1-\frac{1}{p}, p}(\Gamma_4)$. Consequently, $\mathcal{R}_{m+1, k+1}^{[4]}(g_4^{m+1})$ is well-defined by \cref{cor:rmr-gamma4-cont-high}, and hence $\mathcal{L}_{k, m+1}^{[4]}(F)$ is well-defined with
		\begin{multline*}
			\| \mathcal{L}_{k, m+1}^{[4]}(F) \|_{s, p, \reftet} \lesssim_{b,k,s,p} \| \mathcal{L}_{k,m}^{[4]}(F)\|_{s, p, \reftet} \\
			+ \inorm{\mathfrak{E},k+1}{f_4^{m+1} - \partial_{\unitvec{n}}^{m+1}\mathcal{L}_{k,m}^{[4]}(F)}{s-m-1-\frac{1}{p}, p, \Gamma_4}.
		\end{multline*} 
		Applying \cref{eq:proof:l4km-interp} and \cref{eq:rmr4-gamma4-interp} gives \cref{eq:proof:l4km-interp} for $m+1$, while applying \cref{eq:wsp-e-intersection-id-2}, \cref{eq:proof:l4km-cont}, and \cref{eq:trace-math-reg-2} gives \cref{eq:proof:l4km-cont} for $m+1$.
		
		Moreover, if $F$ satisfies \cref{eq:review:sigmak-cont-poly} for some $N \in \mathbb{N}_0$, then $\mathcal{L}_{k, m}^{[4]}(F) \in \mathcal{P}_N(\reftet)$ by assumption and so the trace $G$ satisfies \cref{eq:review:sigmak-cont-poly} and $G \in \Tr^{s, p}_k(\partial \reftet)$ for all $(s, p) \in \mathcal{A}_k$ thanks to \cref{lem:review:polys-trace-space}. By \cref{eq:wsp-e-intersection-id-2} and \cref{lem:vanish-trace-reg-2}, $g_4^{m+1} \in W_{\mathfrak{E}}^{k + 1, p}(\Gamma_4)$ for all $p \in (1,\infty)$, and so $D_{\Gamma}^l g^{m+1}_4|_{ \partial \Gamma_4 } = 0$ for $0 \leq l \leq k$. Thanks to \cref{cor:rmr-gamma4-cont-high}, $\mathcal{L}_{k, m+1}^{[4]}(F) \in \mathcal{P}_{N}(\reftet)$.
	\end{proof}

	\subsection{Proof of \cref{thm:main-result}}
	\label{sec:proof-thm:main-result}
	
	Let $b \in C_c^{\infty}(\reftri)$ be any smooth function satisfying  $\int_{\reftri} b(\bdd{x}) \d{\bdd{x}} = 1$. Then, $\mathcal{L}_k := \mathcal{L}_k^{[4]}$, where $\mathcal{L}_{k}^{[4]}$ is defined in \cref{eq:l4-def} satisfies the desired properties thanks to \cref{lem:l4}. \hfill \proofbox
	
	\section{Whole space operators}
	\label{sec:whole-space-cont}
	
	In this section, we examine the continuity properties of the following operators, which are the whole space extensions of the lifting operators $\mathcal{E}_k^{[1]}$ \cref{eq:ek1-def}: Given $k \in \mathbb{N}_0$, $\chi \in C^{\infty}_c(\mathbb{R})$, and $b \in C^{\infty}_c(\mathbb{R}^2)$ and a function $f : \mathbb{R}^2 \to \mathbb{R}$, we define the lifting operator $\tilde{\mathcal{E}}_k$ by the rule
	\begin{align}
		\label{eq:tilde-ek-def}
		\tilde{\mathcal{E}}_k(f)(\bdd{x}, z) := \chi(z) z^k \int_{\mathbb{R}^2} b(\bdd{y}) f(\bdd{x} + z\bdd{y}) \d{\bdd{y}} \qquad \forall (\bdd{x}, z) \in \mathbb{R}^2 \times \mathbb{R}.
	\end{align}
	We use the notation $\tilde{\mathcal{E}}_k[\chi, b]$ when we want to make the dependence on $\chi$ and $b$ explicit. The advantage of working with the operator $\tilde{\mathcal{E}}_k$ is that we shall capitalize on the abundance of equivalent $W^{s, p}(U)$-norms when $U$ is all of $\mathbb{R}^d$ or the half-space $\mathbb{R}^{d}_+ = \mathbb{R}^{d-1} \times (0, \infty)$, $d > 1$. In particular, we recall the following norm-equivalence on $W^{s, p}(U)$, $0 < s < 1$, $1 < p < \infty$, with $U = \mathbb{R}^d$ or $\mathbb{R}^d_+$ (see e.g. \cite[Theorems 6.38 \& 6.61]{Leoni23}):	
	\begin{align}
		\label{eq:wsp-tensorized-norm}
		|f|_{s, p, U}^p \approx_{s, p, d}  \sum_{i=1}^{d} \int_{0}^{\infty} \int_{U} \frac{ |f(\bdd{x} + t \unitvec{e}_i) - f(\bdd{x})|^{p} }{ t^{1+sp} } \d{\bdd{x}} \d{t} \qquad \forall f \in W^{s, p}(U). 
	\end{align}
	The main result of this section is the following analogue of \cref{lem:em-gamma1-cont-high}.
	\begin{theorem}
		\label{thm:tilde-em-cont-high}
		Let $\chi \in C^{\infty}_c(\mathbb{R})$ with $\supp \chi \in (-2, 2)$, $b \in C^{\infty}_c(\mathbb{R}^2)$ with $\supp b \subset \reftri$, and $k \in \mathbb{N}_0$ be given. Then, for $(s, p) \in \mathcal{A}_k \cup (k+\frac{1}{2}, 2)$, there holds
		\begin{align}
			\label{eq:tilde-em-cont-high}
			\| \tilde{\mathcal{E}}_k(f) \|_{s, p, \mathbb{R}^3_+} \lesssim_{\chi, b, k, s, p} \|f\|_{s-k-\frac{1}{p}, p, \mathbb{R}^2} \qquad \forall f \in W^{s-k-\frac{1}{p}, p}(\mathbb{R}^2). 
		\end{align}  
	\end{theorem}
	The proof of \cref{thm:tilde-em-cont-high} appears in \cref{sec:proof-tilde-em-cont-high}.

    \subsection{Continuity of $\tilde{\mathcal{E}}_0$}
    
    We begin by recording the particular case of \cref{thm:tilde-em-cont-high} with $k=0$, which follows from the same arguments as in the proof of \cite[Theorem 9.21]{Leoni23}. 
    \begin{lemma}
    	Let $\chi \in C^{\infty}_c(\mathbb{R})$ with $\supp \chi \in (-2, 2)$ and $b \in C^{\infty}_c(\mathbb{R}^2)$ with $\supp b \subset \reftri$. Then, for $1 < p < \infty$ and $1/p < s < 1$, there holds
    	\begin{align}
    		\label{eq:tilde-e0-cont-lowreg}
    		\| \tilde{\mathcal{E}}_0(f) \|_{s, p, \mathbb{R}^3_+} \lesssim_{\chi, b, s, p} \|f\|_{s-\frac{1}{p}, p, \mathbb{R}^2} \qquad \forall f \in W^{s-\frac{1}{p}, p}(\mathbb{R}^2). 
    	\end{align}
    \end{lemma}
   	When $p=2$, the above result is also true for $s = 1/2$ as the following lemma shows.
   	\begin{lemma}
   		\label{lem:tilde-e0-cont-h12l2}
   		Let $\chi \in C^{\infty}_c(\mathbb{R})$ with $\supp \chi \in (-2, 2)$ and $b \in C^{\infty}_c(\mathbb{R}^2)$ with $\supp b \subset \reftri$. Then, there holds
   		\begin{align}
   			\label{eq:tilde-e0-cont-h12l2}
   			\| \tilde{\mathcal{E}}_0(f) \|_{\frac{1}{2}, 2, \mathbb{R}^3_+} \lesssim_{\chi, b} \|f\|_{2, \mathbb{R}^2} \qquad \forall f \in L^{2}(\mathbb{R}^2). 
   		\end{align}
	\end{lemma}
	\begin{proof}
		By density, it suffices to consider $f \in C^{\infty}_{c}(\mathbb{R}^2)$. For $k \in \mathbb{N}_0$ define
		\begin{align}
			\label{eq:g-def}
			g_k(\bdd{x}, z) := z^k \int_{\mathbb{R}^2} b(\bdd{y}) f(\bdd{x} + z\bdd{y}) \d{\bdd{y}}, \qquad (\bdd{x}, z) \in \mathbb{R}^3.
		\end{align}	
		
		\noindent \textbf{Step 1: $H^{1/2}(\mathbb{R}^3_+)$ bound for $g_0$. } Thanks to \cref{eq:wsp-tensorized-norm}, there holds
		\begin{align*}
			| g_0 |_{\frac{1}{2}, 2, \mathbb{R}_+^3}^2 \approx \int_{0}^{\infty} |g_0(\cdot, z)|_{\frac{1}{2},2,\mathbb{R}^2}^2 \d{z} + \int_{\mathbb{R}^2} |g_0(\bdd{x}, \cdot)|_{\frac{1}{2}, 2, \mathbb{R}_+}^2 \d{\bdd{x}} =: I_1 + I_2. 
		\end{align*}
		We now follow the steps in the proof of \cite[Theorem 2.2]{Bern95}. Let $\hat{\cdot}$ denote the Fourier transform with respect to the $\bdd{x}$-variable. Then,
		\begin{align*}
			I_1 \approx \int_{0}^{\infty} \int_{\mathbb{R}^2} |\bdd{\xi}| \cdot |\hat{g}_0(\bdd{\xi}, z)|^2 \d{\bdd{\xi}} \d{z} &= \int_{0}^{\infty} \int_{\mathbb{R}^2} |\bdd{\xi}| \cdot |\hat{b}(\bdd{\xi} z) \hat{f}(\bdd{\xi})|^2 \d{\bdd{\xi}} \d{z} \\
			&= \int_{\mathbb{R}^2} \left( |\bdd{\xi}| \cdot \| \hat{b}(\bdd{\xi} \cdot ) \|_{2, \mathbb{R}_+}^2 \right) |\hat{f}(\bdd{\xi})|^2 \d{\bdd{\xi}},
		\end{align*}
		where we used the following convolution identity for $z > 0$:
		\begin{align}
			\label{eq:proof:tilde-e0-rewrite}
			g_0(\bdd{x}, z) = \int_{\mathbb{R}^2} z^{-2} b\left(\frac{\bdd{y} - \bdd{x}}{z}\right) f(\bdd{y}) \d{\bdd{y}} \implies \hat{g}_0(\bdd{\xi}, z) = \hat{b}(\bdd{\xi} z) \hat{f}(\bdd{\xi}).
		\end{align}
		Similarly, there holds
		\begin{align*}
			I_2 \approx \int_{\mathbb{R}^2} |\hat{g}(\bdd{\xi}, \cdot)|^2_{\frac{1}{2}, 2, \mathbb{R}_+} \d{\bdd{\xi}} = \int_{\mathbb{R}^2} | \hat{b}(\bdd{\xi} \cdot ) |_{\frac{1}{2}, 2, \mathbb{R}_+}^2 |\hat{f}(\bdd{\xi})|^2 \d{\bdd{\xi}}.
		\end{align*}		
		Thanks to a change of variables, we obtain
		\begin{align*}
			|\bdd{\xi}| \cdot \| \hat{b}(\bdd{\xi} \cdot ) \|_{2, \mathbb{R}_+}^2 +  | \hat{b}(\bdd{\xi} \cdot ) |_{\frac{1}{2}, 2, \mathbb{R}_+}^2 
			&\leq \sup_{\bdd{\omega} \in \mathbb{S}^2} \left(  |\bdd{\xi}| \cdot \| \hat{b}(|\bdd{\xi}| \bdd{\omega} \cdot ) \|_{2, \mathbb{R}_+}^2 + | \hat{b}(|\bdd{\xi}| \bdd{\omega} \cdot ) |_{\frac{1}{2}, 2, \mathbb{R}_+}^2 \right) \\
			&= \sup_{\bdd{\omega} \in \mathbb{S}^2} \| \hat{b}(\bdd{\omega} \cdot ) \|_{\frac{1}{2}, 2, \mathbb{R}_+}^2,
		\end{align*}
		which is finite since $\hat{b}$ is a Schwartz function, and so $| g_0 |_{\frac{1}{2}, 2, \mathbb{R}_+^3} \lesssim_{b}   \|f\|_{2, \mathbb{R}^2}.$
		
		\textbf{Step 2: $H^{1/2}(\mathbb{R}^3_+)$ bound on $\tilde{\mathcal{E}}_0(f)$. } For $i=1,2$, there holds
		\begin{multline*}
			\int_{0}^{\infty} \int_{\mathbb{R}^3_+} \frac{ |\tilde{\mathcal{E}}_0(f)(\bdd{x} + t\unitvec{e}_i, z) - \tilde{\mathcal{E}}_0(f)(\bdd{x}, z)|^2  }{ t^{2} } \d{\bdd{x}} \d{z} \d{t} \\
			= \int_{0}^{\infty} \int_{\mathbb{R}^3_+} |\chi(z)|^2 \frac{ |g_0(\bdd{x} + t\unitvec{e}_i, z) - g_0(\bdd{x}, z)|^2  }{ t^{2} } \d{\bdd{x}} \d{z} \d{t} \lesssim_{\chi, b} \|f\|_{2, \mathbb{R}^2}.
		\end{multline*}
		where we used \cref{eq:wsp-tensorized-norm} and step 1.	Thanks to the relation
		\begin{multline*}
			|\tilde{\mathcal{E}}_0(f)(\bdd{x}, z + t) - \tilde{\mathcal{E}}_0(f)(\bdd{x}, z)|^2 \lesssim |\chi(z + t)|^2 |g_0(\bdd{x}, z + t) - g_0(\bdd{x}, t)|^2 \\ + |\chi(z+t) - \chi(z)|^2 | g_0(\bdd{x}, z)|^2,
		\end{multline*}
		we obtain
		\begin{multline*}
			\int_{0}^{\infty} \int_{\mathbb{R}^3_+}  \frac{|\tilde{\mathcal{E}}_0(f)(\bdd{x}, z + t) - \tilde{\mathcal{E}}_0(f)(\bdd{x}, z)|^2}{t^2} \d{\bdd{x}} \d{z} \d{t} \\
			 \lesssim \|\chi\|_{\infty, \mathbb{R}_+}^2 |g_0|_{\frac{1}{2}, 2, \mathbb{R}_+^3}^2 + \int_{0}^{\infty} \int_{\mathbb{R}^3_+}  \frac{|\chi(z+t) - \chi(z)|^2}{t^2} |g_0(\bdd{x}, z)|^2 \d{\bdd{x}} \d{z} \d{t}.
		\end{multline*}
		Now, applying Hardy's inequality \cite[Theorem 327]{Hardy52} gives
		\begin{align*}
			\int_{0}^{\infty} \frac{|\chi(z+t) - \chi(z)|^2}{t^2} \d{t} 
			&= \int_{0}^{\infty} \left( \frac{1}{t} \int_{z}^{z+t} \chi'(r) \d{r} \right)^2 \d{t} \\
			&= \int_{0}^{\infty} \left( \frac{1}{t} \int_{0}^{t} \chi'(r+z) \d{r} \right)^2 \d{t} 
			\lesssim \|\chi'(\cdot + z)\|_{2, \mathbb{R}_+}^2,
		\end{align*}
		and so 
		\begin{align*}
			\int_{0}^{\infty} \int_{\mathbb{R}^3_+}  \frac{|\chi(z+t) - \chi(z)|^2}{t^2} |g_0(\bdd{x}, z)|^2 \d{\bdd{x}} \d{z} \d{t} &\lesssim \int_{\mathbb{R}_+^3} \| \chi'(\cdot + z)\|_{2, \mathbb{R}_+}^2 |g_0(\bdd{x}, z)|^2 \d{\bdd{x}} \d{z} \\
			&\leq \|\chi'\|_{2, \mathbb{R}_+}^2 \int_{0}^{2} \int_{\mathbb{R}^2} |g_0(\bdd{x}, z)|^2 \d{\bdd{x}} \d{z}.
		\end{align*}
		Applying Young's inequality to the convolution form of $g_0$,  \cref{eq:proof:tilde-e0-rewrite} then gives
		\begin{align*}
			\int_{0}^{2} \int_{\mathbb{R}^2} |g_0(\bdd{x}, z)|^2 \d{\bdd{x}} \d{z} \leq 2 \| b \|_{1, \mathbb{R}^2}^2 \|f\|_{2, \mathbb{R}^2}^2.
		\end{align*}
		Inequality \cref{eq:tilde-e0-cont-h12l2} now follows on collecting results and applying \cref{eq:wsp-tensorized-norm}.
	\end{proof}
   	We shall also need the stability of the lifting of the derivative of a smooth function.   	
	\begin{lemma}
		Let $\chi \in C^{\infty}_c(\mathbb{R})$ and $b \in C^{\infty}_c(\mathbb{R}^2)$. For $1 < p < \infty$, there holds
		\begin{align}
			\label{eq:tilde-e0-derivative-cont}
			\sum_{i=1}^{2} \| \tilde{\mathcal{E}}_0(\partial_i f) \|_{p, \mathbb{R}^3_+} \lesssim_{\chi, b, p} \|f\|_{1-\frac{1}{p}, p, \mathbb{R}^2} \qquad \forall f \in C_c^{\infty}(\mathbb{R}^2). 
		\end{align}
	\end{lemma}
	\begin{proof}
		Let $1 < p < \infty$, $f \in C_c^{\infty}(\mathbb{R}^2)$, and $i \in \{1,2\}$. Integrating by parts gives
		\begin{align*}
			\tilde{\mathcal{E}}_0(\partial_i f)(\bdd{x}, z) 
			=  \chi(z) \int_{\mathbb{R}^2} b(\bdd{y}) (\partial_i f)(\bdd{x} + z \bdd{y}) \d{\bdd{y}}
			&= \frac{\chi(z)}{z} \int_{\mathbb{R}^2} b(\bdd{y}) \partial_{y_i} \{ f(\bdd{x} + z \bdd{y}) \} \d{\bdd{y}} \\ 
			&= -\frac{\chi(z)}{z} \int_{\mathbb{R}^2} (\partial_i b)(\bdd{y}) f(\bdd{x} + z \bdd{y}) \d{\bdd{y}}.
		\end{align*}
		Since $b \in C^{\infty}_c(\mathbb{R}^2)$, there holds $\int_{\mathbb{R}} (\partial_i b)(\bdd{y}) \d{y_i} = 0$, and so
		\begin{align*}
			\tilde{\mathcal{E}}_0(\partial_i f)(\bdd{x}, z) 
			&= \chi(z) \int_{\mathbb{R}^2} (\partial_i b)(\bdd{y}) \frac{ f(\bdd{x} + z(\bdd{y} - y_i \bdd{e}_i)) - f(\bdd{x} + z\bdd{y}) }{ z } \d{\bdd{y}}.
		\end{align*}
		Applying H\"{o}lder's inequality, we obtain
		\begin{align*}
			|\tilde{\mathcal{E}}_0(\partial_i f)(\bdd{x}, z)|^p \lesssim_{\chi, b, p}  \int_{\mathbb{R}^2} \left|y_i (\partial_i b)(\bdd{y})\right| \left| \frac{ f(\bdd{x} + z(\bdd{y} - y_i \bdd{e}_i)) - f(\bdd{x} + z\bdd{y}) }{ y_i z  } \right|^p \d{\bdd{y}}.
		\end{align*}
		Integrating over $\mathbb{R}^3_+$ then gives
		\begin{align*}
			\| \tilde{\mathcal{E}}_0(\partial_i f) \|_{p, \mathbb{R}^3_+}^p
			&\lesssim_{\chi, b, p} \int_{\mathbb{R}^5_+} \left|y_i (\partial_i b)(\bdd{y})\right| \left| \frac{ f(\bdd{x} + z(\bdd{y} - y_i \bdd{e}_i)) - f(\bdd{x} + z\bdd{y}) }{ y_i z  } \right|^p \d{\bdd{y}} \d{\bdd{x}} \d{z} \\
			&\leftstackrel{\tilde{x}_j = x_j + (1-\delta_{ij}) y_j z }{=} \int_{\mathbb{R}^5_+} \left|y_i (\partial_i b)(\bdd{y})\right| \left| \frac{ f(\tilde{\bdd{x}}) - f( \tilde{\bdd{x}} +  y_i z \bdd{e}_i ) }{ y_i z } \right|^p \d{\bdd{y}} \d{\tilde{\bdd{x}}} \d{z} \\
			&\leftstackrel{t = y_i z}{\lesssim}  \| \partial_i b\|_{1, \mathbb{R}^2} \int_{\mathbb{R}^3_+} \frac{ |f(\tilde{\bdd{x}}) - f(\tilde{\bdd{x}} + t\bdd{e}_i)|^p }{ t^p  }  \d{\tilde{\bdd{x}}} \d{t}.
		\end{align*}
		Inequality \cref{eq:tilde-e0-derivative-cont} now follows from summing over $i$ and applying \cref{eq:wsp-tensorized-norm}.
	\end{proof}
    
    \subsection{Continuity of $\tilde{\mathcal{E}}_k$}
    
    We now show how the continuity of the operator $\tilde{\mathcal{E}}_{0}$ can be used to deduce the continuity of $\tilde{\mathcal{E}}_k$ for $k \in \mathbb{N}_0$. We begin with a partial result.
    \begin{lemma}
   		Let $\chi \in C^{\infty}_c(\mathbb{R})$ and $b \in C^{\infty}_c(\mathbb{R}^2)$ be as in \cref{thm:tilde-em-cont-high} and $k \in \mathbb{N}_0$ be given. Then, for $1 < p < \infty$, there holds
    	\begin{align}
    		\label{eq:tilde-em-lp-continuity}
   			\| \tilde{\mathcal{E}}_k(f) \|_{p, \mathbb{R}^3_+} \lesssim_{\chi, b, k, p} \|f\|_{p, \mathbb{R}^2} \qquad \forall f \in C^{\infty}_c(\mathbb{R}^2),
    	\end{align}
    	and for $1/p < s < 1$ or $(s, p) = (\frac{1}{2}, 2)$, there holds
    	\begin{align}
    		\label{eq:tilde-em-cont-high-partial}
    		\| \tilde{\mathcal{E}}_k(f) \|_{s, p, \mathbb{R}^3_+} \lesssim_{\chi, b, k, s, p} \|f\|_{s-\frac{1}{p}, p, \mathbb{R}^2} \qquad \forall f \in C^{\infty}_c(\mathbb{R}^2). 
    	\end{align}  
   	\end{lemma}
   	\begin{proof}
   		Let $k \in \mathbb{N}_0$, $1 < p < \infty$, and $f \in C^{\infty}_c(\mathbb{R}^2)$. Since the function $\tilde{\chi} := z^k \chi \in C^{\infty}_c(\mathbb{R})$ with $\supp \tilde{\chi} = \supp \chi$, we have $\tilde{\mathcal{E}}_{k}[\chi, b](f) = \tilde{\mathcal{E}}_0[\tilde{\chi}, b](f)$. Consequently, it suffices to prove \cref{eq:tilde-em-lp-continuity,eq:tilde-em-cont-high-partial} in the case $k=0$. To this end, we apply Jensen's inequality to the identity \cref{eq:proof:tilde-e0-rewrite} to obtain
   		\begin{align*}
   			\| \tilde{\mathcal{E}}_0(f) \|_{p, \mathbb{R}^3_+}^p 
   			&\leq \|f\|_{p, \mathbb{R}^2}^p \int_{\mathbb{R}} |\chi(z)|^p \left( \int_{\mathbb{R}^2} \left| z^{-2} b\left(\frac{\bdd{x}}{z}\right) \right| \d{\bdd{x}} \right)^p \d{z} \\
   			&= \|f\|_{p, \mathbb{R}^2}^p \|b\|_{1, \mathbb{R}^2}^p \|\chi(z)\|_{1, \mathbb{R}}^p,
   		\end{align*}
   		and \cref{eq:tilde-em-lp-continuity} follows. Inequality \cref{eq:tilde-em-cont-high-partial} for $1/p < s < 1$ is an immediate consequence of \cref{eq:tilde-e0-cont-lowreg}, while the case  $(s, p) = (\frac{1}{2}, 2)$ follows from \cref{lem:tilde-e0-cont-h12l2}.
   	\end{proof}	
    
    For more precise results, we shall show the effect of taking partial derivatives of $\tilde{\mathcal{E}}_k(f)$ on the index $k$ and on the function $f$. To this end, we recall an integration-by-parts formula for tensors. Given two $d$-dimensional tensors $S$ and $T$, let $S : T$ denote the usual tensor contraction
    \begin{align*}
    	S : T := S_{i_1 i_2 \cdots i_d} T_{i_1 i_2 \cdots i_d},
    \end{align*}
    where we are using Einstein summation notation. Given a $d$-dimensional tensor $S$ with $d \geq 0$ and $k \geq 0$, let $D^k S$ denote the $k$-th derivative tensor of $S$:
    \begin{align*}
    	(D^k S)_{i_1 i_2 \cdots i_{d+k}} := \partial_{i_{d+1}} \partial_{i_{d+2}} \cdots \partial_{i_{d+k}} S_{i_1 i_2 \cdots i_d},
    \end{align*}
    and let $\dive S$ denote the $(d-1)$-dimensional tensor given by
    \begin{align*}
    	(\dive S)_{i_1 i_2 \cdots i_{d-1}} := \partial_{j} S_{i_1 i_2 \cdots i_{d-1} j},
    \end{align*}
    while $\dive^k S$, $0 \leq k \leq d$, denotes $k$ applications of $\dive$ to $S$. With this notation, we have the following integration by parts formula for symmetric, smooth, compactly supported tensors $S$ and $T$ of dimension $d$ and $0 \leq k \leq d$, respectively:
    \begin{align*}
    	\int_{\mathbb{R}^2} S : D^{d-k} T  \d{\bdd{x}} = (-1)^{d-k} \int_{\mathbb{R}^2} \dive^{d-k} S : T \d{\bdd{x}}. 
    \end{align*}
    With this notation in hand, we have the following identity that shows that the derivatives of $\tilde{\mathcal{E}}_k(f)$ are linear combinations of liftings of derivatives of $f$.
    \begin{lemma}
    	Let $\chi \in C^{\infty}_c(\mathbb{R})$, $b \in C^{\infty}_c(\mathbb{R}^2)$, and $k \in \mathbb{N}_0$ be given.
    	For all $\alpha \in \mathbb{N}_0^{3}$ and $f \in C^{\infty}_c(\mathbb{R}^2)$, there holds
    	\begin{multline}
    		\label{eq:tilde-em-derivative-id}
    		D^{\alpha} \tilde{\mathcal{E}}_k(f)(\bdd{x}, z) \\ = \sum_{i=0}^{\alpha_3} \chi_i(z) z^{\max\{k+i-|\alpha|, 0\}} \int_{\mathbb{R}^2} B_{ki\alpha}(\bdd{y}) : (D^{\max\{ |\alpha|-k-i, 0\}} f)(\bdd{x} + z\bdd{y}) \d{\bdd{y}} 
    	\end{multline}
    	for suitable $\chi_i \in C_c^{\infty}(\mathbb{R})$ and $\max\{ |\alpha|-i-k, 0\}$-dimensional tensors $B_{ki\alpha}$ with entries in  $C^{\infty}_c(\mathbb{R}^2)$.
    \end{lemma}
    \begin{proof}
    	Let $f \in C^{\infty}_c(\mathbb{R}^2)$ and let $g_k$ be defined as in \cref{eq:g-def}. For integers $m \geq k$, there holds
    	\begin{align*}
    		\partial_z^m g_k(\bdd{x}, z) &= \sum_{j=0}^{k} c_{kmj} z^{k-j} \int_{\mathbb{R}^2} b(\bdd{y}) (D^{m-j} f)(\bdd{x} + z\bdd{y}) : \bdd{y}^{\otimes m-j} \d{\bdd{y}} \\
    		&=  \sum_{j=0}^{k} c_{kmj} \int_{\mathbb{R}^2}  (b(\bdd{y}) \bdd{y}^{\otimes m-j}) : D^{k-j}_{\bdd{y}} \{ (D^{m-k} f)(\bdd{x} + z\bdd{y}) \} \d{\bdd{y}} \\
    		&=  \int_{\mathbb{R}^2} \left\{  \sum_{j=0}^{k} (-1)^{k-j} c_{kmj} \dive^{k-j} (b(\bdd{y})  \bdd{y}^{\otimes m-j}) \right\} : (D^{m-k} f)(\bdd{x} + z\bdd{y})  \d{\bdd{y}} \\
    		&=: \int_{\mathbb{R}^2} B_{km}(\bdd{y}) :  (D^{m-k} f)(\bdd{x} + z\bdd{y})  \d{\bdd{y}},
    	\end{align*}
    	where $c_{kmj}$ are suitable constants, $\bdd{y}^{\otimes n}$ is the tensor product of $n$ copies of $\bdd{y}$, and $D_{\bdd{y}}$ denotes the derivative operator with respect to $\bdd{y}$. For $0 \leq m < k$, there holds
    	\begin{align*}
    		\partial_z^m g_k(\bdd{x}, z) &= \sum_{j=0}^{m} c_{kmj} z^{k-j} \int_{\mathbb{R}^2} b(\bdd{y}) (D^{m-j} f)(\bdd{x} + z\bdd{y}) : \bdd{y}^{\otimes m-j} \d{\bdd{y}} \\
    		&=  \sum_{j=0}^{m} c_{kmj} z^{k-m} \int_{\mathbb{R}^2}  (b(\bdd{y}) \bdd{y}^{\otimes m-j}) : D^{m-j}_{\bdd{y}} \{ f(\bdd{x} + z\bdd{y}) \}  \d{\bdd{y}} \\
    		&= z^{k-m} \int_{\mathbb{R}^2} \left\{  -\sum_{j=0}^{m} (-1)^{m-j} c_{kmj} \dive^{m-j} (b(\bdd{y})  \bdd{y}^{\otimes m-j}) \right\} : f(\bdd{x} + z\bdd{y})  \d{\bdd{y}} \\
    		&=: z^{k-m} \int_{\mathbb{R}^2} B_{km}(\bdd{y}) : f(\bdd{x} + z\bdd{y})  \d{\bdd{y}}.
    	\end{align*}
    	Consequently, there holds
    	\begin{align*}
    		\partial_z^m g_k(\bdd{x}, z) &= z^{\max\{ k-m, 0 \}} \int_{\mathbb{R}^2} B_{km}(\bdd{y}) : (D^{\max\{ m-k, 0\}} f)(\bdd{x} + z\bdd{y}) \d{\bdd{y}} \qquad \forall m \in \mathbb{N}_0.
    	\end{align*}
    	Now let $\beta \in \mathbb{N}_0^2$ with $|\beta| \geq k$. Then, there holds
    	\begin{multline*}
    		D_{\bdd{x}}^{\beta} g_k(\bdd{x}, z) = z^{k} \int_{\mathbb{R}^2} b(\bdd{y}) (D^{\beta} f)(\bdd{x} + z\bdd{y}) \d{\bdd{y}} = \int_{\mathbb{R}^2} b(\bdd{y}) D^{\tilde{\beta}}_{\bdd{y}} \{ (D^{\beta-\tilde{\beta}}f)(\bdd{x} + z\bdd{y}) \} \d{\bdd{y}} \\
    		= (-1)^{|\tilde{\beta}|}\int_{\mathbb{R}^2} (D^{\tilde{\beta}} b)(\bdd{y}) (D^{\beta-\tilde{\beta}}f)(\bdd{x} + z\bdd{y}) \d{\bdd{y}} =: \int_{\mathbb{R}^2} {B}_{k \beta}(\bdd{y}) (D^{\beta-\tilde{\beta}}f)(\bdd{x} + z\bdd{y}) \d{\bdd{y}}, 
    	\end{multline*}
    	where $D_{\bdd{x}}^{\beta} := \partial_{x_1}^{\beta_1} \partial_{x_2}^{\beta_2}$ and $\tilde{\beta} \in \mathbb{N}_0^2$ is any fixed multi-index such that $|\tilde{\beta}| = k$ and $\beta - \tilde{\beta} \in \mathbb{N}_0^2$. Similar arguments show that for $\beta \in \mathbb{N}_0^2$ with $|\beta| < k$, there holds
    	\begin{align*}
    		D_{\bdd{x}}^{\beta} g_k(\bdd{x}, z) &= (-1)^{|\beta|} z^{k-|\beta|} \int_{\mathbb{R}^2} (D^{\beta} b)(\bdd{y}) f(\bdd{x} + z\bdd{y}) \d{\bdd{y}} \\
    		&=: z^{k-|\beta|} \int_{\mathbb{R}^2} {B}_{k \beta}(\bdd{y}) f(\bdd{x} + z\bdd{y}) \d{\bdd{y}}.
    	\end{align*}	
    	Collecting results, for any $\alpha \in \mathbb{N}_0^3$, there holds
    	\begin{align*}
    		D^{\alpha} g_k(\bdd{x}, z) = z^{\max\{k-|\alpha|, 0\}} \int_{\mathbb{R}^2} B_{k\alpha}(\bdd{y}) : (D^{\max\{ |\alpha|-k, 0\}} f)(\bdd{x} + z\bdd{y}) \d{\bdd{y}} 
    	\end{align*}
    	for suitable $\max\{ |\alpha|-k, 0\}$-dimensional tensors $B_{k\alpha}$ with entries in  $C^{\infty}_c(\mathbb{R}^2)$. Equality \cref{eq:tilde-em-derivative-id} now follows from the product rule.
    \end{proof}
	
	\subsection{Proof of \cref{thm:tilde-em-cont-high}}
	\label{sec:proof-tilde-em-cont-high}
    	Let $k \in \mathbb{N}_0$, $1 < p < \infty$, and $f \in C^{\infty}_c(\mathbb{R}^2)$. For $\alpha \in \mathbb{N}_0^3$, \cref{eq:tilde-em-derivative-id} gives
    	\begin{align}
    		\label{eq:proof:deriv-tildeem-norms}
    		\| D^{\alpha} \tilde{\mathcal{E}}_k(f) \|_{\sigma, p, \mathbb{R}^3_+} \leq \sum_{ \substack{0 \leq i \leq \alpha_3 \\ \beta \in \mathbb{N}_0^2 \\ |\beta| = \max\{|\alpha|-k-i, 0\} } } \| \tilde{\mathcal{E}}_{ \max\{k+i-|\alpha|, 0\} }[\chi_i, b_{ki\beta}](D^{\beta} f) \|_{\sigma, p, \mathbb{R}^3_+},
    	\end{align}
    	where $\chi_i \in C^{\infty}_c(\mathbb{R})$ and $b_{ki\beta} \in C^{\infty}_c(\mathbb{R}^2)$ are suitable functions depending on $\chi$ and $b$ respectively and $0 \leq \sigma < 1$. 
    	
    	\noindent \textbf{Step 1: $L^p$ bounds on derivatives. } For $k + i - |\alpha| \geq 0$ (so that $|\beta| = 0$), \cref{eq:tilde-em-lp-continuity} gives
    	\begin{align*}
    		\|\tilde{\mathcal{E}}_{ \max\{k+i-|\alpha|, 0\} }[\chi_i, b_{ki\beta}](D^{\beta} f) \|_{p, \mathbb{R}^3_+} = \|\tilde{\mathcal{E}}_{ k+i-|\alpha| }[\chi_i, b_{ki\beta}](f) \|_{p, \mathbb{R}^3_+} \lesssim_{\chi, b, k, p} \|f\|_{p, \mathbb{R}^2}.
    	\end{align*}
    	For $k + i - |\alpha| < 0$ (so that $|\alpha| \geq k$ and $|\beta| \geq 1$), there exists $j \in \{1, 2\}$ such that $\beta_j \geq 1$, and so we apply \cref{eq:tilde-e0-derivative-cont} to obtain
    	\begin{align*}
    		\|\tilde{\mathcal{E}}_{  0 }[\chi_i, b_{ki\beta}](D^{\beta} f) \|_{p, \mathbb{R}^3_+} 
    		= \|\tilde{\mathcal{E}}_{ 0 }[\chi_i, b_{ki\beta}](\partial_j D^{\beta - \unitvec{e}_j} f) \|_{p, \mathbb{R}^3_+} 
    		&\lesssim_{\chi, b, k, p} \| D^{\beta - \unitvec{e}_j} f \|_{1-\frac{1}{p}, p, \mathbb{R}^2} \\
    		 &\leq \|f\|_{|\alpha|-k-i-\frac{1}{p}, p, \mathbb{R}^2}.
    	\end{align*}
    	Consequently, for all $f \in C^{\infty}_c(\mathbb{R}^2)$, there holds
    	\begin{align}
    		\label{eq:proof:tilde-em-integer-bound}
    		\| \tilde{\mathcal{E}}_k(f) \|_{m, p, \mathbb{R}^3_+} \lesssim_{\chi, b, k, m, p} \|f\|_{m-k-\frac{1}{p}, p, \mathbb{R}^2}, \qquad m \in \{k+1, k+2, \ldots\}.
    	\end{align}
    	By density, \cref{eq:proof:tilde-em-integer-bound} holds for all $f \in W^{m-k-1/p, p}(\mathbb{R}^2)$.
    	
    	\noindent \textbf{Step 2: The case $s \geq k+1$. } Inequality \cref{eq:tilde-em-cont-high} for real  $s \geq k+1$ with $(s, p) \in \mathcal{A}_k$ follows from \cref{eq:proof:tilde-em-integer-bound} using a standard interpolation argument.
    	
    	\noindent \textbf{Step 3: The case $k+1/p \leq s < k+1$. } For $s = k + \sigma$, where $1/p < \sigma < 1$ or $(\sigma, p) = (1/2, 2)$, we take $|\alpha| = k$ in \cref{eq:proof:deriv-tildeem-norms} and apply \cref{eq:tilde-em-cont-high-partial} to obtain
    	\begin{align*}
    		|\tilde{\mathcal{E}}_k(f)|_{s, p, \mathbb{R}^3_+} \leq  \sum_{ 0 \leq i \leq \alpha_3   } \| \tilde{\mathcal{E}}_{ i }[\chi_i, b_{ki\beta}](f) \|_{\sigma, p, \mathbb{R}^3_+} \lesssim_{\chi, b, k, p, s} \|f\|_{\sigma-\frac{1}{p}, p, \mathbb{R}^2},
    	\end{align*}
    	which completes the proof. \proofbox

    \section{Weighted $L^p$ continuity of whole-space operators}
    \label{sec:whole-space-lp-cont}
    
    In the previous section in \cref{thm:tilde-em-cont-high}, we established that the lifting operators $\tilde{\mathcal{E}}_k$ are continuous from $W^{s-k-1/p, p}(\mathbb{R}^2)$ to $W^{s, p}(\mathbb{R}^3_+)$ provided that $s > k + 1/p$. We now turn to the stability of the operator $\tilde{\mathcal{E}}_k$ with respect to lower-order Sobolev spaces. In particular, we seek to obtain bounds on $\| \tilde{\mathcal{E}}_k(f)\|_{s, p, \mathcal{O}_1}$ for $0 \leq s < k+1/p$, where $\mathcal{O}_1 := (0, \infty)^3 \supset \reftet$ is the first octant. It turns out that one suitable space for the lifted function $f$ is a weighted $L^p$ space. Let $\mathcal{Q}_1 = (0, \infty)^2 \supset \reftri$ denote the first quadrant and let $\rho \in L^{\infty}(\mathcal{Q}_1)$ be a weight function that satisfying $\rho > 0$ almost everywhere. Then, for $1 < p < \infty$, define
    \begin{align}
    	\label{eq:weighted-lp-space-def}
    	L^p(\mathcal{Q}_1; \rho \d{\bdd{x}}) := \left\{ f \text{ measurable} : \int_{\mathcal{Q}_1} |f(\bdd{x})|^p \rho(\bdd{x}) \d{\bdd{x}} < \infty \right\}.
   	\end{align}
   	The weights that will appear in our estimates are powers of $\omega_1$ and $\omega_2$ defined in \cref{eq:omegai-def}. In particular, the main result of this section is as follows.	
	\begin{theorem}
		\label{thm:tilde-em-cont-low}
		Let  $\chi \in C^{\infty}_c(\mathbb{R})$ and $b \in C^{\infty}_c(\mathbb{R}^2)$ be as in \cref{thm:tilde-em-cont-high} and $k \in \mathbb{N}_0$ be given. For $1 < p < \infty$ and $0 \leq s < k + 1/p$, there holds
		\begin{align}
			\label{eq:tilde-em-cont-low}
			\| \tilde{\mathcal{E}}_k(f)\|_{s, p, \mathcal{O}_1} \lesssim_{\chi, b, k, s, p} \| \omega_1^{\frac{1}{p} + k - s} f\|_{p, \mathcal{Q}_1} \qquad \forall f \in L^p(\mathcal{Q}_1, \omega_1^{1 + (k-s)p} \d{\bdd{x}}).
		\end{align}
	\end{theorem}
	The proof proceeds in several steps and appears in \cref{sec:proof-tilde-em-cont-low}.
	
	\subsection{Auxiliary results}
	
	We begin by recording a number of technical lemmas. Throughout the rest of the section we use the notation $\fint_{\mathcal{O}} f d{\bdd{x}} := |\mathcal{O}|^{-1} \int _{\mathcal{O}} f d{\bdd{x}}$.
	\begin{lemma}
		For $1 \leq p < \infty$ and $0 < h < \infty$, there holds
		\begin{align}
			\label{eq:work:lp-norm-average-1d}
			\int_{0}^{\infty} \left| \fint_{x}^{x+h} f(y) \d{y} \right|^p \d{x} \leq h^{p-1} \int_{0}^{\infty} |f(x)|^p \d{x} \qquad \forall f \text{ measurable}.
		\end{align}
	\end{lemma}
	\begin{proof}
		The result follows on applying H\"{o}lder's inequality and changing the order of integration:
		\begin{align*}
			\int_{0}^{\infty} \left| \frac{1}{h} \int_{x}^{x+h} f(y) \d{y} \right|^p \d{x}  &\leq h^{p-2} \int_{0}^{\infty} \int_{x}^{x+h} |f(y)|^p \d{y} \\ 
			&= h^{p-2} \left( \int_{0}^{h} \int_{0}^{y} + \int_{h}^{\infty} \int_{y-h}^{y} \right) |f(y)|^p \d{x} \d{y}.
		\end{align*}
	\end{proof}
	
	\begin{lemma}
		Let $1 < p < \infty$, $0 < s < 1$, and $0 \leq a \leq \infty$. Then, there holds
		\begin{align}
			\label{eq:fraction-by-weighted-derivative-interval}
			\int_{(0, a)^2} \frac{|f(x) - f(y)|^p}{|x-y|^{1+sp}} \d{x} \d{y} \lesssim_{s, p} \int_{(0, a)} x^{(1-s)p} |f'(x)|^p \d{x} \qquad \forall f \in W^{1, p}_{\mathrm{loc}}((0, a)).
		\end{align}
	\end{lemma}
	\begin{proof}
		The proof follows the same arguments as those used in the proof of \cite[Theorem 1.28]{Leoni23}, which considers the case $a = \infty$. The full details are given below. 
		
		By symmetry, there holds
		\begin{align*}
			\int_{(0, a)^2} \frac{|f(x) - f(y)|^p}{|x-y|^{1+sp}} \d{x} \d{y} &= 2 \int_{0}^{a} \int_{y}^{a} \frac{|f(x) - f(y)|^p}{(x-y)^{1+sp}} \d{x} \d{y} \\
			&= 2 \int_{0}^{a} \int_{y}^{a} \frac{1}{(x-y)^{1+sp}} \left| \int_{y}^{x} f'(t) \d{t} \right|^p \d{x} \d{y}.
		\end{align*}
		Performing a change of variable and applying Hardy's inequality \cite[Theorem 1.3]{Leoni23}, we obtain
		\begin{align*}
			\int_{y}^{a} \frac{1}{(x-y)^{1+sp}} \left| \int_{y}^{x} f'(t) \d{t} \right|^p \d{x}  &\leftstackrel{\substack{\tilde{x}=x-y \\ \tau = t-y}}{\leq} \int_{0}^{a-y} \frac{1}{\tilde{x}^{1+sp}} \left( \int_{0}^{\tilde{x}}  |\tau f'(y + \tau)| \frac{\d{\tau}}{\tau} \right)^p \d{\tilde{x}} \\
			&\leq \frac{1}{s^p}\int_{0}^{a-y} \frac{|f'(y + \tilde{x})|^p}{\tilde{x}^{1+(s-1)p}}    \d{\tilde{x}} \\
			&\leftstackrel{x = \tilde{x} + y}{=} \frac{1}{s^p}\int_{y}^{a} \frac{|f'(x)|^p}{(x-y)^{1+(s-1)p}} \d{x}.
		\end{align*}
		Thus,
		\begin{align*}
			\int_{(0, a)^2} \frac{|f(x) - f(y)|^p}{|x-y|^{1+sp}} \d{x} \d{y} &\leq \frac{2}{s^p} \int_{0}^{a} \int_{y}^{a} \frac{|f'(x)|^p}{(x-y)^{1+(s-1)p}}    \d{x} \d{y} \\
			&= \frac{2}{s^p}\int_{0}^{a} |f'(x)|^p \int_{0}^{x} \frac{1}{(x-y)^{1+(s-1)p}}  \d{y} \d{x} \\
			&= \frac{2}{s^p(1-s)p} \int_{0}^{a} x^{(1-s)p}  |f'(x)|^p \d{x},
		\end{align*}
		which completes the proof.
	\end{proof}
	
	\begin{lemma}
		Let $1 < p < \infty$ and $0 < s < 1/p$. For all $f \in L^p(\mathcal{Q}_1; \omega_1^{1-sp} \d{\bdd{x}})$, there holds
		\begin{align}
			\label{eq:inverse-weighted-z-by-weighted-lp}
			\int_{ \mathcal{Q}_1 } \int_{0}^{2} \int_{x_1}^{x_1 + z} \int_{x_2}^{x_2 + z} \frac{1}{z^{2+sp}}  |f(\bdd{y})|^p \d{y_2} \d{y_1}  \d{z} \d{\bdd{x}} \lesssim_{s, p} \| \omega_1^{\frac{1}{p}-s} f \|_{p, \mathcal{Q}_1}^{p}.
		\end{align} 
	\end{lemma}
	\begin{proof}
		Applying \cref{eq:work:lp-norm-average-1d}  and using that $0 < z < 2$ gives
		\begin{align*}
			\int_{0}^{\infty} \fint_{x_2}^{x_2 + z} \left( \int_{x_1}^{x_1+z} |f(\bdd{y})|^p \d{y_1} \right) \d{y_2} \d{x_2} \lesssim_{p} \int_{0}^{\infty}  \int_{x_1}^{x_1 + z} |f(y_1, x_2)|^p \d{y_1} \d{x_2}.
		\end{align*}
		Moreover, there holds
		\begin{align*}
			&\int_{0}^{\infty} \int_{0}^{2} \int_{x_1}^{x_1 + z} \frac{1}{z^{1+sp}} |f(y_1, x_2)|^p \d{y_1} \d{z} \d{x_1}  \\
			&\qquad = \int_{0}^{\infty} \int_{x_1}^{x_1+2}  |f(y_1, x_2)|^p \int_{y_1-x_1}^{2} \frac{1}{z^{1+sp}} \d{z} \d{y_1} \d{x_1}  \\
			&\qquad\lesssim_{s, p} \int_{0}^{\infty} \int_{x_1}^{x_1+2} (y_1-x_1)^{-sp} |f(y_1, x_2)|^p \d{y_1} \d{x_1}  \\
			&\qquad= \left( \int_{0}^{2} \int_{0}^{y_1} + \int_{2}^{\infty} \int_{y_1-2}^{y_1} \right) (y_1-x_1)^{-sp} |f(y_1, x_2)|^p \d{x_1} \d{y_1}  \\
			&\qquad\lesssim_{s, p} \int_{0}^{\infty} \min\{ y_1, 2 \}^{1-sp} |f(y_1, x_2)|^p \d{y_1}.
		\end{align*}
		The result now follows on integrating over $0 < x_2 < \infty$.
	\end{proof}

	\begin{lemma}
		Let  $\chi \in C^{\infty}_c(\mathbb{R})$ and $b \in C^{\infty}_c(\mathbb{R}^2)$ be as in \cref{thm:tilde-em-cont-high} and $1 < p < \infty$. Let $k \in \mathbb{N}_0$ and $f \in L^p(\mathcal{Q}_1; \omega_1^{1+kp} \d{\bdd{x}})$. 	For $0 < t < 2$, there holds
		\begin{align}
			\label{eq:weighted-3d-lp-gen-t}
			\int_{0}^{t} \int_{\mathcal{Q}_1} |g_k(\bdd{x}, z)|^p \d{\bdd{x}} \d{z} \lesssim_{b, k, p} \int_{\mathcal{Q}_1} \min\{x_1, t\}^{1+kp} |f(\bdd{x})|^p \d{\bdd{x}}, 
		\end{align}
		where $g_k$ is defined in \cref{eq:g-def}
	\end{lemma}
	\begin{proof}
		Let $k \in \mathbb{N}_0$, $1 < p < \infty$, and $f \in L^p(\mathcal{Q}_1; \omega_1^{1+kp} \d{\bdd{x}})$ be given. Let $z \in (0, t)$. Then, for $\bdd{x} \in \mathcal{Q}_1$ and $\bdd{y} \in (0, 1)^2$, there holds $z \leq \min\{ x_1 + z y_1, t \}/y_1$, and so
		\begin{align*}
			|g_k(\bdd{x}, z)| 
			&\leq   \int_{(0, 1)^2} \min\{ x_1 + z y_1, t \}^k |y_1^{-k} b(\bdd{y})|  |f(\bdd{x} + z\bdd{y})| \d{\bdd{y}} \\ 
			&\leftstackrel{\bdd{u} = \bdd{x} + z\bdd{y}}{\lesssim_{b,k}}  \fint_{x_2}^{x_2 + z} \fint_{x_1}^{x_1 + z} \min\{ u_1, t \}^k |f(\bdd{u})| \d{u_1} \d{u_2}  .
		\end{align*}
		Integrating over $x_2 \in (0, \infty)$ and applying \cref{eq:work:lp-norm-average-1d} to the function
		\begin{align*}
			\tilde{f}(u_2; x_1, z) = \fint_{x_1}^{x_1+z}  |\tilde{\omega}_1^k f(\bdd{u})| \d{u_1}, \qquad \text{where } \tilde{\omega}_1(\bdd{u}) :=  \min\{ u_1, t \}^k
		\end{align*}
		and using that $0 < z < t < 2$ gives
		\begin{align*}
			\int_{0}^{\infty} |g_k(\bdd{x}, z) |^p \d{x_2} 
			&\leq \int_{0}^{\infty} \left( \fint_{x_2}^{x_2 + z} \tilde{f}(u_2; x_1, z) \d{u_2} \right)^p \d{x_2} \\
			&\leq 2^{p-1} \int_{0}^{\infty}  \left( \fint_{x_1}^{x_1 + z}   |(\tilde{\omega}_1^k f)(u_1, x_2)| \d{u_1} \right)^p \d{x_2}.
		\end{align*}
		Hardy's inequality \cite[Theorem 327]{Hardy52} then shows that, for  $0 < x_2 < \infty$, there holds
		\begin{align*}
			\int_{0}^{t} \left( \fint_{x_1}^{x_1 + z} |(\tilde{\omega}_1^k f)(u_1, x_2)| \d{u_1} \right)^p  \d{z} \quad 
			&\leftstackrel{v = x_1 + z}{=} \int_{x_1}^{x_1 + t} \left( \fint_{x_1}^{v} |(\tilde{\omega}_1^k f) (u_1, x_2)| \d{u_1} \right)^p  \d{v} \\
			&\lesssim_p \int_{x_1}^{x_1 + t} |(\tilde{\omega}_1^k f)(v, x_2)|^p \d{v},
		\end{align*}
		and so 
		\begin{align*}
			\int_{0}^{t} \int_{0}^{\infty} |g_k(\bdd{x}, z) |^p \d{x_2} \d{z} \lesssim_{b, k, p} \int_{0}^{\infty} \int_{x_1}^{x_1 + t}  |(\tilde{\omega}_1^k f)(v, x_2)|^p \d{v} \d{x_2}.
		\end{align*}
		Integrating over $x_1$ and changing the order of integration gives
		\begin{align*}
			\int_{0}^{t} \int_{\mathcal{Q}_1} |g_k(\bdd{x}, z)|^p \d{\bdd{x}} \d{z} 
			&\lesssim_{b, k, p} \int_{\mathcal{Q}_1} \int_{x_1}^{x_1 + t} |(\tilde{\omega}_1^k f)(v, x_2)|^p \d{v} \d{\bdd{x}}  \\
			&= \int_{0}^{\infty} \left( \int_{0}^{t} \int_{0}^{v} + \int_{t}^{\infty} \int_{v-t}^{v} \right) |(\tilde{\omega}_1^{k} f)(v, x_2)|^p \d{x_1} \d{v} \d{x_2}  \\
			&\leq \int_{\mathcal{Q}_1} \tilde{\omega}_1(v, x_2)^{1+kp} |f(v, x_2)|^p \d{v} \d{x_2},
		\end{align*}
		which completes the proof.
	\end{proof}

	\subsection{Continuity of $\tilde{\mathcal{E}}_0$}
	
	In this section, we prove \cref{thm:tilde-em-cont-low} in the case $k=0$. We will utilize the following equivalent norm on $W^{s, p}(\mathcal{O}_1)$.
	
	\begin{lemma}
		For all $p \in (1,\infty)$, $s \in (0, 1)$, and $f \in W^{s, p}(\mathcal{O}_1)$,  there holds
		\begin{align}
			\label{eq:wsp-o1-norm-equiv}
			\| f \|_{s, p, \mathcal{O}_1}^p \approx_{s, p} \| f \|_{p, \mathcal{O}_1}^p + \sum_{i=1}^{3} \int_{0}^{1} \int_{\mathcal{O}_1} \frac{ |f(\bdd{x} + t\unitvec{e}_i) - f(\bdd{x})|^p }{t^{1+sp}} \d{\bdd{x}} \d{t}.
		\end{align}
	\end{lemma}	
	\begin{proof}
		Let $f \in W^{s, p}(\mathcal{O}_1)$. Thanks to \cite[Theorem 6.38]{Leoni23}, there holds
		\begin{align*}
			| f |_{s, p, \mathcal{O}_1}^p \approx_{s, p} \sum_{i=1}^{3} \int_{0}^{\infty} \int_{\mathcal{O}_1} \frac{ |f(\bdd{x} + t\unitvec{e}_i) - f(\bdd{x})|^p }{t^{1+sp}} \d{\bdd{x}} \d{t},
		\end{align*}
		and \cref{eq:wsp-o1-norm-equiv} now follows on noting that
		\begin{align*}
			\sum_{i=1}^{3} \int_{1}^{\infty} \int_{\mathcal{O}_1} \frac{ |f(\bdd{x} + t\unitvec{e}_i) - f(\bdd{x})|^p }{t^{1+sp}} \d{\bdd{x}} \d{t} \lesssim_{s, p} \|f\|_{p, \mathcal{O}_1}^{p}.
		\end{align*}
	\end{proof}		
	
	We now estimate each term in \cref{eq:wsp-o1-norm-equiv}. The first result deals with terms involving translations in the first two coordinate directions.	
	\begin{lemma}
		Let $\chi \in C^{\infty}_c(\mathbb{R})$ and $b \in C^{\infty}_c(\mathbb{R}^2)$ be as in \cref{thm:tilde-em-cont-high}.
		For $1 < p < \infty$, $0 < s < 1/p$, and $1 \leq i \leq 2$, there holds
		\begin{align}
			\label{eq:weighted-3d-x-dir-gen}
			\int_{0}^{1} \int_{\mathcal{O}_1} \frac{|\tilde{\mathcal{E}}_0(f)(\bdd{x} + t\bdd{e}_i, z) - \tilde{\mathcal{E}}_0(f)(\bdd{x}, z)|^p}{t^{1+sp}} \d{\bdd{x}} \d{z} \d{t} \lesssim_{\chi, b, s, p} \| \omega_1^{\frac{1}{p} - s} f \|_{p, \mathcal{Q}_1}^p  
		\end{align}
		for all $f \in C^{\infty}_c(\mathcal{Q}_1)$.
	\end{lemma}
	\begin{proof}
		Let $1 < p < \infty$, $0 < s < 1/p$, and $f \in C^{\infty}_c(\mathcal{Q}_1)$ be given. Let $g_0(\cdot,\cdot)$ be as in \cref{eq:g-def} with $k=0$ and let $\tilde{g}(\bdd{x}, z, t) := g_0(\bdd{x} + t\unitvec{e}_i, z) - g_0(\bdd{x}, z)$.
		
		\noindent \textbf{Step 1. } Let $1 \leq i \leq 2$. We will show that
		\begin{align}
			\label{eq:proof:weighted-3d-x-dir-v}
			\int_{0}^{1} \int_{0}^{2} \int_{ \mathcal{Q}_1 } \frac{|\tilde{g}(\bdd{x}, z, t)|^p}{t^{1+sp}} \d{\bdd{x}} \d{z} \d{t} \lesssim_{b, s, p} \| \omega_1^{1/p - s} f \|_{p, \mathcal{Q}_1}^p.
		\end{align}
		We begin by decomposing the above integral into two terms:
		\begin{multline*}
			\int_{0}^{1} \int_{0}^{2} \int_{ \mathcal{Q}_1 } \frac{|\tilde{g}(\bdd{x}, z, t)|^p}{t^{1+sp}} \d{\bdd{x}} \d{z} \d{t} \\
			= \left( \int_{0}^{1} \int_{0}^{t} \int_{ \mathcal{Q}_1 } + \int_{0}^{1} \int_{t}^{2} \int_{ \mathcal{Q}_1 }  \right) \frac{|\tilde{g}(\bdd{x}, z, t)|^p}{t^{1+sp}} \d{\bdd{x}} \d{z} \d{t}  
			=: A_i + B_i.
		\end{multline*}
		
		\noindent \textbf{Part (a): $A_i$. } Let $0 < t < 1$. Then,  $f(\cdot + t\unitvec{e}_i) - f(\cdot) \in L^p(\mathcal{Q}_1, \omega_1\d{\bdd{x}})$ and
		\begin{align*}
			\tilde{g}(\bdd{x}, z, t) = \int_{ \mathcal{Q}_1 } b(\bdd{y}) \left[ f(\bdd{x} + z\bdd{y} + t\unitvec{e}_i) - f(\bdd{x} + z\bdd{y})\right]  \d{\bdd{y}}.
		\end{align*}
		Integrating \cref{eq:weighted-3d-lp-gen-t} over $0 < t < 1$ then gives
		\begin{align*}
			A_i \lesssim_{b, p} \int_{\mathcal{Q}_1 }  \int_{0}^{1} \min\{x_1, t\} \frac{|f(\bdd{x} + t\unitvec{e}_i)|^p + |f(\bdd{x})|^p}{t^{1+sp}} \d{t} \d{\bdd{x}}.
		\end{align*}
		For $i = 1$, there holds
		\begin{align*}
			\int_{0}^{\infty} \int_{0}^{1} t^{-sp} |f(\bdd{x} + t\bdd{e}_1)|^p \d{t} \d{x_1} 
			& \quad \leftstackrel{\tilde{x}_1 = x_1 + t}{=} \int_{0}^{1} \int_{t}^{\infty} t^{-sp} |f(\tilde{x}_1, x_2)|^p \d{\tilde{x}_1} \d{t} \\
			& = \left( \int_{0}^{1} \int_{0}^{\tilde{x}_1} + \int_{1}^{\infty} \int_{0}^{1} \right) t^{-sp} |f(\tilde{x}_1, x_2)|^{p} \d{t} \d{\tilde{x}_1} \\
			& \lesssim_{s,p} \int_{0}^{\infty} \min\{\tilde{x}_1, 1\}^{1-sp} |f(\tilde{x}_1, x_2)|^{p} \d{\tilde{x}_1}.
		\end{align*}
		On the other hand, note that for any $0 \leq u < \infty$ and $\tilde{f}(\bdd{x}) = f(\bdd{x} + u\unitvec{e}_2)$, there holds 
		\begin{align*}
			& \int_{0}^{\infty} \int_{0}^{1} \min\{x_1, t\} t^{-(1+sp)} |\tilde{f}(\bdd{x})|^p \d{t} \d{x_1} \\
			&\qquad = \left( \int_{0}^{1} \int_{0}^{x_1} + \int_{1}^{\infty} \int_{0}^{1} \right) t^{-sp} |\tilde{f}(\bdd{x})|^p \d{t} \d{x_1} + \int_{0}^{1} \int_{x_1}^{1}  t^{-(1+sp)} x_1 |\tilde{f}(\bdd{x})|^p \d{t} \d{x_1} \\
			&\qquad\lesssim_{s,p} \int_{0}^{\infty} \min\{x_1, 1\}^{1-sp} |\tilde{f}(\bdd{x})|^p  \d{x_1} + \int_{0}^{1} ( x_1^{-sp} - 1) x_1 |\tilde{f}(\bdd{x})|^{p} \d{x_1} \\
			&\qquad\lesssim \int_{0}^{\infty} \min\{x_1, 1\}^{1-sp} |f(x_1, x_2 + u)|^p \d{x_1}.
		\end{align*}
		The bound $A_i \lesssim_{b, s, p} \| \omega_1^{\frac{1}{p} - s} f \|_{p, \mathcal{Q}_1}^p$ now follows on performing a change of variables and collecting results.
		
		\noindent \textbf{Part (b): $B_i$. } Using identity \cref{eq:proof:tilde-e0-rewrite}, we obtain
		\begin{align*}
			\tilde{g}(\bdd{x}, z, t) &= z^{-2} \int_{\mathbb{R}^2} \left[ b\left(\frac{\bdd{y}-\bdd{x}-t\unitvec{e}_i}{z} \right) - b\left( \frac{\bdd{y}-\bdd{x}}{z} \right) \right] f(\bdd{y}) \d{\bdd{y}} \\
			&= -z^{-3} \int_{\mathbb{R}^2} \int_{0}^{t} (\partial_i b)\left( \frac{\bdd{y}-\bdd{x}-r\unitvec{e}_i}{z} \right) f(\bdd{y}) \d{r}  \d{\bdd{y}}.
		\end{align*}
		Writing $z^{-2} |\partial_i b| = (z^{-2} |\partial_i b|)^{1-1/p} (z^{-2} |\partial_i b|)^{1/p}$ and applying H\"{o}lder's inequality gives
		\begin{align*}
			|\tilde{g}(\bdd{x}, z, t)|^p 
			&\leq \left( \int_{0}^{t} \int_{\mathbb{R}^2} \frac{1}{z^2} |\partial_i b|\left(\frac{\bdd{y}-\bdd{x}-r\unitvec{e}_i}{z} \right) \d{\bdd{y}} \d{r} \right)^{p-1} \\
			&\qquad \times \frac{1}{z^{p+2}} \int_{0}^{t} \int_{\mathbb{R}^2} |\partial_i b| \left( \frac{\bdd{y}-\bdd{x}-r\unitvec{e}_i}{z} \right) |f(\bdd{y})|^p \d{\bdd{y}} \d{r} \\
			&\lesssim_{b, p} \frac{t^{p-1}}{z^{p+2}} \int_{0}^{t} \int_{\mathbb{R}^2} |\partial_i b| \left( \frac{\bdd{y}-\bdd{x}-r\unitvec{e}_i}{z} \right) |f(\bdd{y})|^p \d{\bdd{y}} \d{r}. 
		\end{align*}
		Integrating over $\bdd{x}$ gives
		\begin{align*}
			\int_{ \mathcal{Q}_1 } |\tilde{g}(\bdd{x}, z, t)|^p  \d{\bdd{x}} 
			&\quad \leftstackrel{\tilde{\bdd{x}} = \bdd{x}+r\unitvec{e}_i}{\leq} \frac{t^{p-1}}{z^{p+2}} \int_{0}^{t} \int_{ \mathcal{Q}_1 } \int_{\mathbb{R}^2} |\partial_i b| \left( \frac{\bdd{y}-\tilde{\bdd{x}}}{z} \right) |f(\bdd{y})|^p \d{\bdd{y}} \d{\tilde{\bdd{x}}} \d{r}  \\
			&\quad\leq \frac{t^{p}}{z^{p+2}} \int_{ \mathcal{Q}_1 } \int_{\mathbb{R}^2} |\partial_i b| \left( \frac{\bdd{y}-\tilde{\bdd{x}}}{z} \right) |f(\bdd{y})|^p \d{\bdd{y}} \d{\tilde{\bdd{x}}} \\
			&\quad\lesssim_{b} \frac{t^{p}}{ z^{p + 2}} \int_{ \mathcal{Q}_1 } \int_{x_1}^{x_1 + z} \int_{x_2}^{x_2 + z}  |f(\bdd{y})|^p \d{y_2} \d{y_1} \d{\bdd{x}}.
		\end{align*}
		Integrating over $z$ and $t$, we obtain
		\begin{align*}
			B_i &\lesssim_{b, p} \int_{0}^{1} \int_{t}^{2}  \frac{t^{(1-s)p - 1}}{z^{p+2}} \int_{ \mathcal{Q}_1 } \int_{x_1}^{x_1 + z} \int_{x_2}^{x_2 + z} |f(\bdd{y})|^p \d{y_2} \d{y_1} \d{\bdd{x}} \d{z} \d{t} \\
			&= \int_{0}^{2} \int_{0}^{z} \frac{t^{(1-s)p - 1}}{z^{p+2}} \int_{ \mathcal{Q}_1 } \int_{x_1}^{x_1 + z} \int_{x_2}^{x_2 + z}  |f(\bdd{y})|^p \d{y_2} \d{y_1} \d{\bdd{x}} \d{t} \d{z} \\
			&= \frac{1}{p(1-s)} \int_{ \mathcal{Q}_1 } \int_{0}^{2} \int_{x_1}^{x_1 + z} \int_{x_2}^{x_2 + z} \frac{1}{z^{2+sp}}  |f(\bdd{y})|^p \d{y_2} \d{y_1}  \d{z} \d{\bdd{x}}.
		\end{align*}
		Applying \cref{eq:inverse-weighted-z-by-weighted-lp}, we obtain $B_i \lesssim_{b, s, p} \| \omega_1^{\frac{1}{p}-s} f \|_{p, \mathcal{Q}_1}^p$, which completes the proof of \cref{eq:proof:weighted-3d-x-dir-v}.
		
		\noindent \textbf{Step 2. } Since $\supp \chi \subset B(0, 2)$, there holds
		\begin{multline*}
			\int_{0}^{1} \int_{\mathcal{O}_1} \frac{|\tilde{\mathcal{E}}_0(f)(\bdd{x} + t\unitvec{e}_i, z) - \tilde{\mathcal{E}}_0(f)(\bdd{x}, z)|^p}{t^{1+sp}} \d{\bdd{x}} \d{z} \d{t}  \\ 
			= \int_{0}^{1} \int_{0}^{2} \int_{\mathcal{Q}_1} |\chi(z)|^p \frac{|\tilde{g}(\bdd{x}, z, t)|^p}{t^{1+sp}} \d{\bdd{x}} \d{z} \d{t} 
			\lesssim_{\chi, b, s, p} \int_{ \mathcal{Q}_1 }  \min\{x_1, 1\}^{1-sp} f(\bdd{x}) \d{\bdd{x}},
		\end{multline*}
		which completes the proof.	
	\end{proof}

	The next result deals with the term involving a translation in the $z$-direction.	
	\begin{lemma}
		Let $\chi \in C^{\infty}_c(\mathbb{R})$ and $b \in C^{\infty}_c(\mathbb{R}^2)$ be as in \cref{thm:tilde-em-cont-high}. For $1 < p < \infty$ and $0 < s < 1/p$, there holds
		\begin{align}
			\label{eq:weighted-3d-y-dir}
			\int_{0}^{1} \int_{\mathcal{O}_1} \frac{|\tilde{\mathcal{E}}_0(f)(\bdd{x}, z+t) - \tilde{\mathcal{E}}_0(f)(\bdd{x}, z)|^p}{t^{1+sp}} \d{\bdd{x}} \d{z} \d{t} \lesssim_{\chi, b, s, p} \| \omega_1^{\frac{1}{p}-s} f \|_{p, \mathcal{Q}_1}^p
		\end{align}
		for all $f \in C^{\infty}_c(\mathcal{Q}_1)$.
	\end{lemma}
	\begin{proof}
		Let $1 < p < \infty$, $0 < s < 1/p$, and $f \in C^{\infty}_c(\mathcal{Q}_1)$ be given. Let $g_0(\cdot, \cdot)$ be defined as in \cref{eq:g-def}.

		\noindent \textbf{Step 1. } We will first show that
		\begin{align}
			\label{eq:proof:weighted-3d-y-dir-vxz}
			\int_{0}^{1} \int_{0}^{2} \int_{\mathcal{Q}} \frac{|g_0(\bdd{x}, z+t) - g_0(\bdd{x}, z)|^p}{t^{1+sp}} \d{\bdd{x}} \d{z} \d{t} \lesssim_{b, s, p} \| \omega_1^{\frac{1}{p}-s} f \|_{p, \mathcal{Q}_1}^p.
		\end{align}
		Applying \cref{eq:fraction-by-weighted-derivative-interval} gives
		\begin{align}
			\label{eq:proof:z-semi-by-partialz}
			\int_{0}^{1} \int_{0}^{2} \frac{|g_0(\bdd{x}, z+t) - g_0(\bdd{x}, z)|^p}{t^{1+sp}} \d{z} \d{t} &\leq \int_{0}^{2} z^{(1-s)p} |\partial_z g_0(\bdd{x}, z)|^p \d{z}.
		\end{align}
		Applying identity \cref{eq:proof:tilde-e0-rewrite}, we obtain
		\begin{align*}
			\partial_z g_0(\bdd{x}, z) &= \int_{\mathbb{R}^2} \partial_z \left\{ z^{-2} b\left( \frac{\bdd{y}-\bdd{x}}{z} \right) \right\} f(\bdd{y}) \d{\bdd{y}} \\
			&= -\int_{\mathbb{R}^2} \left\{ 2 z^{-3} b\left( \frac{\bdd{y}-\bdd{x}}{z} \right) + z^{-4} Db\left( \frac{\bdd{y}-\bdd{x}}{z} \right) \cdot (\bdd{y}-\bdd{x}) \right\} f(\bdd{y}) \d{y_1} \d{y_2}.
		\end{align*}
		For $\bdd{y} \in (x_1, x_1+z) \times (x_2, x_2 + z)$, there holds
		\begin{align*}
			\left| 2z^{-3} b\left( \frac{\bdd{y}-\bdd{x}}{z} \right) + z^{-4} Db\left( \frac{\bdd{y}-\bdd{x}}{z} \right) \cdot (\bdd{y}-\bdd{x}) \right| &\lesssim_{b} z^{-3},
		\end{align*}
		where we used that $b$ and $D b$ are uniformly bounded. Since $\supp b \subset (0, 1)^2$, we obtain
		\begin{align*}
			\int_{0}^{2} z^{(1-s)p}|\partial_z g_0(\bdd{x}, z) | \d{z} \lesssim_{b} \int_{0}^{2} \frac{1}{z^{2+sp}}  \int_{x_2}^{x_2 + z} \int_{x_1}^{x_1+z}  |f(\bdd{y})| \d{y_1} \d{y_2} \d{z}.
		\end{align*}
		Inequality \cref{eq:proof:weighted-3d-y-dir-vxz} now follows on integrating \cref{eq:proof:z-semi-by-partialz} over $\bdd{x} \in \mathcal{Q}_1$ and applying \cref{eq:inverse-weighted-z-by-weighted-lp}.
		
		\noindent \textbf{Step 2. } For $0 < t < 2$ and $\bdd{x} \in \mathcal{Q}_1$, we add and subtract $\chi(z+t)g_0(\bdd{x}, t)$ to obtain
		\begin{multline*}
			|\tilde{\mathcal{E}}_0(f)(\bdd{x}, z+t) - \tilde{\mathcal{E}}_0(f)(\bdd{x}, z)|^p \lesssim_{p} |\chi(z + t)|^p |g_0(\bdd{x}, z + t) - g_0(\bdd{x}, t)|^p \\ + |\chi(z+t) - \chi(z)|^p | g_0(\bdd{x}, z)|^p.
		\end{multline*}
		For the first term, we use that $\supp \chi \in (-2, 2)$ and apply \cref{eq:proof:weighted-3d-y-dir-vxz} to obtain
		\begin{multline*}
			\int_{0}^{1} \int_{ \mathcal{O}_1 } |\chi(z + t)|^p \frac{|g_0(\bdd{x}, z + t) - g_0(\bdd{x}, t)|^p}{t^{1+sp}} \d{\bdd{x}} \d{z} \d{t} \\
			\lesssim_{\chi, p}  \int_{0}^{1} \int_{0}^{2} \int_{ \mathcal{Q}_1} \frac{|g_0(\bdd{x}, z + t) - g_0(\bdd{x}, t)|^p}{t^{1+sp}} \d{\bdd{x}} \d{z} \d{t} 
			\lesssim_{b, s, p} \| \omega_1^{\frac{1}{p}-s} f\|_{p, \mathcal{Q}_1}^p.
		\end{multline*}
		For the second term, we again use that $\supp \chi \in (-2, 2)$ as well as the assumption $0 < s < 1/p$:
		\begin{align*}
			&\int_{0}^{1} \int_{ \mathcal{O}_1 }  \frac{|\chi(z+t) - \chi(z)|^p}{t^{1+sp}} | g_0(\bdd{x}, z)|^p \d{\bdd{x}} \d{z} \d{t} \\
			&\qquad = \int_{0}^{1} \int_{0}^{2}  \int_{ \mathcal{Q}_1 } \frac{1}{t^{1+sp}} \left| \int_{z}^{z+t} \chi'(r) \d{r} \right|^p | g_0(\bdd{x}, z)|^p \d{\bdd{x}} \d{z} \d{t} \\
			&\qquad \lesssim_{\chi, p} \int_{0}^{1} \int_{0}^{2}  \int_{ \mathcal{Q}_1 } t^{(1-s)p-1} | g_0(\bdd{x}, z)|^p \d{\bdd{x}} \d{z} \d{t} \\
			&\qquad \lesssim_{s, p} \int_{0}^{2}  \int_{ \mathcal{Q}_1 } | g_0(\bdd{x}, z)|^p \d{\bdd{x}} \d{z}.
		\end{align*}
		Inequality \cref{eq:weighted-3d-y-dir} now follows from \cref{eq:weighted-3d-lp-gen-t}.
	\end{proof}
	
	We now obtain \cref{thm:tilde-em-cont-low} in the case $k=0$.
	\begin{lemma}
		\label{lem:tilde-e0-cont-low}
		Let  $\chi \in C^{\infty}_c(\mathbb{R})$ and $b \in C^{\infty}_c(\mathbb{R}^2)$ be as in \cref{thm:tilde-em-cont-high}. For $1 < p < \infty$ and $0 \leq s < 1/p$, there holds
		\begin{align}
			\label{eq:tilde-e0-cont-low}
			\| \tilde{\mathcal{E}}_0(f)\|_{s, p, \mathcal{O}_1} \lesssim_{\chi, b, k, s, p} \| \omega_1^{\frac{1}{p} - s} f\|_{p, \mathcal{Q}_1} \qquad \forall f \in L^p(\mathcal{Q}_1, \omega_1^{1 - sp} \d{\bdd{x}}).
		\end{align}
	\end{lemma}	
	\begin{proof}
		The case $s = 0$ follows on taking $t=2$ in \cref{eq:weighted-3d-lp-gen-t} and using the fact that $\| \tilde{\mathcal{E}}_0(f) \|_{p, \mathcal{O}_1} \lesssim_{\chi, p} \| g \|_{p, \mathcal{Q}_1 \times (0, 2)}$, where $g$ is defined in \cref{eq:g-def}. The case $0 < s < 1/p$ follows from the norm equivalence \cref{eq:wsp-o1-norm-equiv}, the bounds \cref{eq:weighted-3d-x-dir-gen,eq:weighted-3d-y-dir}, and the density of $C^{\infty}_c(\mathcal{O}_1)$ in $L^p(\mathcal{O}_1, \omega_1^{1-sp} \d{\bdd{x}})$.
	\end{proof}
	
	\subsection{Proof of \cref{thm:tilde-em-cont-low}} 
	\label{sec:proof-tilde-em-cont-low}
	
	Let $k \in \mathbb{N}_0$, $1 < p < \infty$.
	
	\noindent \textbf{Step 1: $s=0$. } Taking $t=2$ in \cref{eq:weighted-3d-lp-gen-t} and using the fact that $\| \tilde{\mathcal{E}}_k(f) \|_{p, \mathcal{O}_1} \lesssim_{\chi, p} \| g_k \|_{p, \mathcal{Q}_1 \times (0, 2)}$, where $g_k$ is defined in \cref{eq:g-def}, we obtain \cref{eq:tilde-em-cont-low} in the case $s = 0$.
	
	\noindent \textbf{Step 2: $s \in \{1, 2, \ldots, k\}$. } Let $f \in C^{\infty}_c(\mathcal{Q}_1)$. Applying \cref{eq:tilde-em-derivative-id} with $|\alpha| \leq k$, we obtain
	\begin{align}
		\label{eq:proof:tilde-ek-sigma-norm-deriv}
		\| D^{\alpha} \tilde{\mathcal{E}}_k(f)\|_{\sigma, p, \mathcal{O}_1} \leq \sum_{0 \leq i \leq \alpha_3} \| \tilde{\mathcal{E}}_{k+i-|\alpha|}[\chi_i, b_{ki}](f)\|_{\sigma, p, \mathcal{O}_1},
	\end{align}
	where $\chi_i \in C^{\infty}_c(\mathbb{R})$ and $b_{ki} \in C^{\infty}_c(\mathbb{R}^2)$ are suitable functions depending on $\chi$ and $b$ respectively and $0 \leq \sigma < 1$. Applying \cref{eq:tilde-em-cont-low} with $s=0$ then gives
	\begin{align*}
		| \tilde{\mathcal{E}}_k(f)|_{s, p, \mathcal{O}_1} \lesssim_{\chi, b, k, s, p} \| \omega_1^{\frac{1}{p} + k - s} f \|_{p, \mathcal{Q}_1},
	\end{align*}
	where we used that $k + i - |\alpha| \geq k-m$ for $0 \leq i \leq \alpha_3$. By density, \cref{eq:tilde-em-cont-low} holds for $s \in \{0, 1, \ldots, k\}$.
	
	\noindent \textbf{Step 3: $0 \leq s \leq k$. } This case follows from  interpolating Step 2 (see e.g. \cite[Theorem 14.2.3]{Brenner08} and \cite[Theorem 5.4.1]{BerLof76}).
	
	\noindent \textbf{Step 4: $k < s < k+1/p$. } Let $\sigma = s - k$ so that $0 < \sigma < 1/p$. Setting $\tilde{\chi}_{ik\alpha}(z) := z^{k+i-|\alpha|} \chi_i \in C^{\infty_c}(\mathcal{Q}_1)$ so that $\supp \tilde{\chi}_{ik\alpha} \in (-2, 2)$ and applying \cref{eq:proof:tilde-ek-sigma-norm-deriv,eq:tilde-e0-cont-low} then gives
	\begin{align*}
		\| D^{\alpha} \tilde{\mathcal{E}}_k(f)\|_{\sigma, p, \mathcal{O}_1} \leq \sum_{0 \leq i \leq \alpha_3}  \| \tilde{\mathcal{E}}_{0}[\tilde{\chi}_{ik\alpha}, b_{ki}](f)\|_{\sigma, p, \mathcal{O}_1} \lesssim_{\chi, b, k, \sigma, p} \| \omega_1^{\frac{1}{p}-\sigma} f \|_{p, \mathcal{Q}_1}.
	\end{align*}
	Inequality \cref{eq:tilde-em-cont-low} now follows. \hfill \proofbox

	\section{Continuity of fundamental operators}
	\label{sec:cont}
	
	In this section, we prove the continuity and interpolation properties of the four fundamental operators $\mathcal{E}_k^{[1]}$ defined in \cref{eq:ek1-def}, $\mathcal{M}_{k, r}^{[1]}$ defined in \cref{eq:mkr-def}, $\mathcal{S}_{k, r}^{[1]}$ defined in \cref{eq:skr-def}, and $\mathcal{R}_{k, r}^{[1]}$ defined in \cref{eq:rkr-def}. We begin with the properties of $\mathcal{E}_k^{[1]}$, which rely on the results of \cref{sec:whole-space-cont}. Then, in \cref{sec:weighted-cont}, we show that the four fundamental operators are continuous from weighted $L^p$ spaces \cref{eq:weighted-lp-space-def} to $W^{s, p}(\reftet)$ for small $s$, which will be useful for the analysis of $\mathcal{M}_{k, r}^{[1]}$, $\mathcal{S}_{k, r}^{[1]}$, and $\mathcal{R}_{k, r}^{[1]}$. This section concludes with the proofs of \cref{lem:mmr-gamma1-cont-high,lem:smr-gamma1-cont-high,lem:rmr-gamma1-cont-high}.
	
	\subsection{Proof of \cref{lem:em-gamma1-cont-high}}
	\label{sec:proof-em-gamma1-cont-high}
	
	\textbf{Step 1: Continuity \cref{eq:em-gamma1-continuity-high}. } Let $\tilde{b}$ denote the extension by zero of $b$ to $\mathbb{R}^2$ and let $\chi \in C^{\infty}_c(\mathbb{R})$ with $\chi \equiv 1$ on $(-1, 1)$ and $\supp \chi \in (-2, 2)$. Let $f \in W^{s-k-\frac{1}{p}, p}(\reftri)$ be given and let $\tilde{f}$ denote a bounded extension $f$ to  $\mathbb{R}^2$ satisfying $\|\tilde{f}\|_{s, p, \mathbb{R}^2} \lesssim_{s, p} \|f\|_{s, p, \reftri}$; see e.g. \cite{Devore93}. Thanks to the identity
	\begin{align}
		\label{eq:proof:em-tilde-em-id}
		\mathcal{E}_k(f) =  \frac{(-1)^k}{k!} \tilde{\mathcal{E}}_k(f)[\chi, \tilde{b}](f) \qquad \text{on } K,
	\end{align}
	where $\tilde{E}_k$ is defined in \cref{eq:tilde-ek-def}, inequality
	\cref{eq:em-gamma1-continuity-high} immediately follows from \cref{eq:tilde-em-cont-high} and the smoothness of the mapping $\mathfrak{I}_1$ defined in \cref{eq:ek1-gamma1-def}.
	
	\noindent \textbf{Step 2: Trace property \cref{eq:em-gamma1-interp}. } Direct computation shows that \cref{eq:em-gamma1-interp} holds.

	\noindent \textbf{Step 3: Polynomial preservation. } If $f \in \mathcal{P}_N(\Gamma_1)$, $N \in \mathbb{N}_0$, then direct inspection reveals that $\mathcal{E}_k^{[1]}(f) \in \mathcal{P}_{N+k}(\reftet)$. \proofbox
	
	\subsection{Weighted continuity}
	\label{sec:weighted-cont}
	
	We begin with the continuity of $\mathcal{E}_k^{[1]}$.
	\begin{lemma}
		Let $b \in C_c^{\infty}(\reftri)$, $k \in \mathbb{N}_0$, $1 < p < \infty$, and $0 \leq s < k+1/p$ or $(s, p) = (k+\frac{1}{2}, 2)$. Then, for all $t_1, t_2, t_3 \in [0, \infty)$ such that $t_1 + t_2 + t_3 = k - s + 1/p$, there holds
		\begin{align}
			\label{eq:em-gamma1-continuity-low}
			\| \mathcal{E}_k^{[1]}(f) \|_{s, p, \reftet} \lesssim_{b, k, s, p}  \| \omega_1^{t_1} \omega_2^{t_2} \omega_3^{t_3} f \|_{p, \reftri} \qquad \forall f \in L^p(\reftri; (\omega_1^{t_1} \omega_2^{t_2} \omega_3^{t_3})^p \d{\bdd{x}}),
		\end{align}
		where $\omega_i$ are defined in \cref{eq:omegai-def}.
	\end{lemma}
	\begin{proof}
		Let $t = k-s+1/p$.
		
		\textbf{Step 1: $t_2 = t_3 = 0$. } Let $\tilde{b}$ denote the extension by zero of $b$ to $\mathbb{R}^2$ and let $\chi \in C^{\infty}_c(\mathbb{R})$ with $\chi \equiv 1$ on $(-1, 1)$ and $\supp \chi \in (-2, 2)$. Let $f \in C^{\infty}_c(\reftri)$ be given and let $\tilde{f}$ denote the extension by zero of $f$ to  $\mathbb{R}^2$. Thanks to the identity \cref{eq:proof:em-tilde-em-id}, \cref{eq:em-gamma1-continuity-low} with $t_2 = t_3 = 0$ follows from \cref{eq:tilde-em-cont-low} and a standard density argument. The case $(s, p) = (k+\frac{1}{2}, 2)$ follows from a similar argument using \cref{eq:tilde-em-cont-high}.
		
		\textbf{Step 2: $t_1 = t_3 = 0$. } We define transformations $\mathfrak{F}_1 : \reftri \to \reftri$ and $\mathfrak{G}_1 : \reftet \to \reftet$ as follows:
		\begin{align}
			\label{eq:proof:frakf1-def}
			\mathfrak{F}_1(\bdd{x}) := (x_2, x_1) \quad \text{and} \quad \mathfrak{G}_1(\bdd{x}, z) := (x_2, x_1, z) \qquad (\bdd{x}, z) \in \reftet.
		\end{align}
		Then, a change of variable shows that $\mathcal{E}_k^{[1]}(f) \circ \mathfrak{G}_1 = \mathcal{E}_k^{[1]}[b \circ \mathfrak{F}_1](f \circ \mathfrak{F}_1)$, 	and so
		\begin{align*}
			\| \mathcal{E}_k^{[1]}(f) \|_{s, p, \reftet} = \| \mathcal{E}_k^{[1]}(f) \circ \mathfrak{G}_1 \|_{s, p, \reftet} \lesssim_{b, k, s, p} \| \omega_1^{t} (f \circ \mathfrak{F}_1) \|_{p, \reftri} = \| \omega_2^t f \|_{p, \reftri},
		\end{align*}
		where we applied Step 1 in the middle inequality.
	
		\noindent \textbf{Step 3: $t_3 = 0$. } Applying Steps 1 and 2 and interpolating between $L^p(\reftri; \omega_1^{tp} d\bdd{x})$ and $L^p(\reftri; \omega_2^{tp} d\bdd{x})$ (see e.g. \cite[Theorem 5.4.1]{BerLof76}) then gives \cref{eq:em-gamma1-continuity-low}.
	
		\noindent \textbf{Step 4: $t_1 = t_2 = 0$. } We define transformations $\mathfrak{F}_2 : \reftri \to \reftri$ and $\mathfrak{G}_2 : \reftet \to \reftet$ as follows:
		\begin{align}
			\label{eq:proof:frakf2-def}
			\mathfrak{F}_2(\bdd{x}) := (x_2, 1-x_1-x_2) \quad \text{and} \quad \mathfrak{G}_2(\bdd{x}, z) := (1-x_1-x_2-z, x_1, z) \qquad (\bdd{x}, z) \in \reftet.
		\end{align}
		A change of variables then gives $\mathcal{E}_k^{[1]}(f) \circ \mathfrak{G}_2 = \mathcal{E}_k^{[1]}[b \circ \mathfrak{F}_2](f \circ \mathfrak{F}_2)$,
		and so
		\begin{align*}
			\| \mathcal{E}_k^{[1]}(f) \|_{t, p, \reftet} \lesssim \| \mathcal{E}_k^{[1]}(f) \circ \mathfrak{G}_2 \|_{t, p, \reftet} \lesssim_{b, k, t, p} \| \omega_1^t (f \circ \mathfrak{F}_2) \|_{p, \reftri} = \| \omega_3^t f \|_{p, \reftri},
		\end{align*}
		where we applied Step 1 in the middle inequality.
	
		\noindent \textbf{Step 5: General case. } Applying Steps 3 and 4 and interpolating between $L^p(\reftri; (\omega_1^{r_1} \omega_2^{r_2})^p d\bdd{x})$ with $r_1, r_2 \in \mathbb{R}_+$ with $r_1 + r_2 = t$ and $L^p(\reftri; \omega_3^{tp}d\bdd{x})$ (see e.g. \cite[Theorem 5.4.1]{BerLof76}) gives
		\cref{eq:em-gamma1-continuity-low}.
	\end{proof}
	
	We now turn to the continuity of $\mathcal{M}_{k, r}^{[1]}$.
	\begin{lemma}
		Let $b \in C_c^{\infty}(\reftri)$, $k, r \in \mathbb{N}_0$, $1 < p < \infty$, and $0 \leq s < k+1/p$ or $(s, p) = (k+\frac{1}{2}, 2)$. Then, for all $t_1, t_2, t_3 \in [0, \infty)$ such that $t_1 + t_2 + t_3 = k - s + 1/p$, there holds
		\begin{align}
			\label{eq:mmr-gamma1-continuity-low}
			\| \mathcal{M}_{k, r}^{[1]}(f) \|_{s, p, \reftet} \lesssim_{b, k, r, s, p}  \| \omega_1^{t_1} \omega_2^{t_2} \omega_3^{t_3} f \|_{p, \reftri} \qquad \forall f \in L^p(\reftri; (\omega_1^{t_1} \omega_2^{t_2} \omega_3^{t_3})^p \d{\bdd{x}}),
		\end{align}
		where $\omega_i$ are defined in \cref{eq:omegai-def}.
	\end{lemma}
	\begin{proof}
		Let $0 \leq s < k + 1/p$ or $(s, p) = (k+\frac{1}{2}, 2)$. We proceed by induction on $r$. The case $r=0$ follows from \cref{eq:em-gamma1-continuity-low}, so assume that \cref{eq:mmr-gamma1-continuity-low} holds for some $r \in \mathbb{N}_0$. Direction computation gives
		\begin{align*}
			\mathcal{M}_{k, r+1}^{[1]}(f)(\bdd{x}, z) - \mathcal{M}_{k, r}^{[1]}(f)(\bdd{x}, z) &= x_2^r \frac{(-z)^k}{k!} \int_{\reftri} b(\bdd{y}) \frac{f(\bdd{x} + z\bdd{y})}{(x_2 + z y_2)^r} \left( \frac{x_2}{x_2 + z y_2} - 1 \right)  \d{\bdd{y}}  \\
			&=  x_2^r \frac{(-z)^{k+1}}{k!} \int_{\reftri} y_2 b(\bdd{y}) \frac{f(\bdd{x} + z\bdd{y})}{(x_2 + z y_2)^{r+1}}   \d{\bdd{y}}  \\
			&= (k+1) \mathcal{M}_{k+1, r}^{[1]}[\omega_2 b](\omega_2^{-1} f)(\bdd{x}, z),
		\end{align*}
		which leads to the following identity
		\begin{align}
			\label{eq:proof:mmr-id}
			\mathcal{M}_{k, r+1}^{[1]}(f) = (k+1) \mathcal{M}_{k+1, r}^{[1]}[\omega_2 b](\omega_2^{-1} f) + \mathcal{M}_{k, r}^{[1]}(f).
		\end{align}
		Consequently, there holds
		\begin{align*}
			\| \mathcal{M}_{k, r+1}^{[1]}(f) \|_{s, p, \reftet} &\leq (k+1) \| \mathcal{M}_{k+1, r}^{[1]}[\omega_2 b](\omega_2^{-1} f) \|_{s, p, \reftet} + \| \mathcal{M}_{k, r}^{[1]}(f) \|_{s, p, \reftet}.
		\end{align*}
		Applying \cref{eq:mmr-gamma1-continuity-low} with $\tau_1 = t_1$, $\tau_2 = t_2 + 1$ and $\tau_3 = t_3$ gives 
		\begin{align*}
			\| \mathcal{M}_{k+1, r}^{[1]}[\omega_2 b](\omega_2^{-1} f) \|_{s, p, \reftet} \lesssim_{b, k, r, s, p} \|  \omega_1^{\tau_1} \omega_2^{\tau_2 - 1} \omega_3^{\tau_3} f \|_{p, \reftri} = \|  \omega_1^{t_1} \omega_2^{t_2} \omega_3^{t_3} f \|_{p, \reftri},
		\end{align*} 
		and so $\| \mathcal{M}_{k, r+1}^{[1]}(f) \|_{s, p, \reftet} \lesssim_{b, k, r, s, p} \|  \omega_1^{t_1} \omega_2^{t_2} \omega_3^{t_3} f \|_{p, \reftri}$, which completes the proof.
	\end{proof}

	It will be convenient to define a three-parameter version of $\mathcal{S}_{k, r}^{[1]}$ as follows:
	\begin{align}
		\label{eq:skrq-def}
		\mathcal{S}_{k, r, q}^{[1]}(f)(\bdd{x}, z) &:= x_1^q x_2^r \mathcal{E}_k^{[1]}(\omega_1^{-q} \omega_2^{-r} f)(\bdd{x}, z) 
	\end{align}
	for $k, r, q \in \mathbb{N}_0$. This three-parameter version satisfies the same continuity properties as $\mathcal{E}_k^{[1]}$ and $\mathcal{M}_{k, r}^{[1]}$.	
	\begin{lemma}
		Let $b \in C_c^{\infty}(\reftri)$, $k, r, q \in \mathbb{N}_0$, $1 < p < \infty$, and $0 \leq s < k+1/p$ or $(s, p) = (k+\frac{1}{2}, 2)$. Then, for all $t_1, t_2, t_3 \in [0, \infty)$ such that $t_1 + t_2 + t_3 = k - s + 1/p$, there holds
		\begin{alignat}{2}
			\label{eq:smrq-gamma1-continuity-low}
			\| \mathcal{S}_{k, r, q}^{[1]}(f) \|_{s, p, \reftet}  &\lesssim_{b, k, r, q, s, p}  \| \omega_1^{t_1} \omega_2^{t_2} \omega_3^{t_3} f \|_{p, \reftri} \qquad & &\forall f \in L^p(\reftri; (\omega_1^{t_1} \omega_2^{t_2} \omega_3^{t_3})^p \d{\bdd{x}}), 
		\end{alignat}
		where $\omega_i$ are defined in \cref{eq:omegai-def}.
	\end{lemma}
	\begin{proof}
		Let $0 \leq s < k + 1/p$ or $(s, p) = (k+\frac{1}{2}, 2)$. We proceed by induction on $q$. The case $q=0$ follows from \cref{eq:mmr-gamma1-continuity-low}, so assume that \cref{eq:smrq-gamma1-continuity-low} holds for some $q \in \mathbb{N}_0$. Direct computation gives
		\begin{align*}
			&\mathcal{S}_{k, r, q+1}^{[1]}(f)(\bdd{x}, z) - \mathcal{S}_{k, r, q}^{[1]}(f)(\bdd{x}, z)  \\
			&\qquad = x_1^q x_2^r \frac{(-z)^k}{k!} \int_{\reftri} b(\bdd{y}) \frac{f(\bdd{x} + z\bdd{y})}{(x_1 + z y_1)^q (x_2 + z y_2)^r} \left( \frac{x_1}{x_1 + z y_1} - 1 \right)  \d{\bdd{y}}  \\
			&\qquad=  x_1^q x_2^r \frac{(-z)^{k+1}}{k!} \int_{\reftri} y_1 b(\bdd{y}) \frac{f(\bdd{x} + z\bdd{y})}{(x_1 + z y_1)^{q+1} (x_2 + z y_2)^r}   \d{\bdd{y}}  \\
			&\qquad= (k+1) \mathcal{S}_{k+1, r, q}^{[1]}[\omega_1 b](\omega_1^{-1} f)(\bdd{x}, z),
		\end{align*}
		which leads to the following identity
		\begin{align}
			\label{eq:proof:smrq-id}
			\mathcal{S}_{k, r, q+1}^{[1]}(f) = (k+1) \mathcal{S}_{k+1, r, q}^{[1]}[\omega_1 b](\omega_1^{-1} f) + \mathcal{S}_{k, r, q}^{[1]}(f).
		\end{align}
		Consequently, there holds
		\begin{align*}
			\| \mathcal{S}_{k, r, q+1}^{[1]}(f) \|_{s, p, \reftet} &\leq (k+1) \| \mathcal{S}_{k+1, r, q}^{[1]}[\omega_1 b](\omega_1^{-1} f) \|_{s, p, \reftet} + \| \mathcal{M}_{k, r, q}^{[1]}(f) \|_{s, p, \reftet}.
		\end{align*}
		Applying \cref{eq:smrq-gamma1-continuity-low} with $\tau_1 = t_1 + 1$, $\tau_2 = t_2$ and $\tau_3 = t_3$ gives 
		\begin{align*}
			\| \mathcal{S}_{k+1, r, q}^{[1]}[\omega_1 b](\omega_1^{-1} f) \|_{s, p, \reftet} \lesssim_{b, k, r, q, s, p} \|  \omega_1^{\tau_1 - 1} \omega_2^{\tau_2} \omega_3^{\tau_3} f \|_{p, \reftri} = \|  \omega_1^{t_1} \omega_2^{t_2} \omega_3^{t_3} f \|_{p, \reftri},
		\end{align*} 
		and so $\| \mathcal{S}_{k, r, q+1}^{[1]}(f) \|_{s, p, \reftet} \lesssim_{b, k, r, q, s, p} \|  \omega_1^{t_1} \omega_2^{t_2} \omega_3^{t_3} f \|_{p, \reftri}$.
	\end{proof}
	
	\subsection{Proof of \cref{lem:mmr-gamma1-cont-high}}
	\label{sec:proof-mmr-gamma1-cont-high}
	
	\noindent	\textbf{Step 1: Continuity \cref{eq:mmr-gamma1-continuity-high}. } We first show that \cref{eq:mmr-gamma1-continuity-high} holds with $\Gamma_1$ replaced by $\reftri$ and $\gamma_{12}$ replaced by $\gamma_{2}$, where we recall that the edges of $\reftri$ are labeled as in \cref{fig:reference triangle}: For all $k, r \in \mathbb{N}_0$, $(s, p) \in \mathcal{A}_k \cup \{ (k+\frac{1}{2}, 2) \}$, and $f~\in~W^{s-k-\frac{1}{p},p}(\reftri) \cap W_{\gamma_{2}}^{\min\{s-k- \frac{1}{p}, r\} , p}(\reftri)$, there holds
	\begin{align}
		\label{eq:proof:mmr-gamma1-continuity-high}
		\| \mathcal{M}_{k, r}^{[1]}(f) \|_{s, p, \reftet} \lesssim_{b, k, r, s, p} \begin{cases}
			\|f\|_{2, \reftri} & \text{if } (s, p) = (k+\frac{1}{2}, 2), \\
			\inorm{\gamma_{2}}{f}{s-k-\frac{1}{p}, p, \reftri} & \text{if } k + \frac{1}{p} < s \leq k+r+\frac{1}{p}, \\
			\| f\|_{s-k-\frac{1}{p}, p, \reftri} & \text{if } s > k + r + \frac{1}{p}.
		\end{cases}
	\end{align}
	
	We proceed by induction on $r$. The case $r=0$ follows from \cref{eq:em-gamma1-continuity-high,eq:em-gamma1-continuity-low}, so assume that \cref{eq:proof:mmr-gamma1-continuity-high} holds for some $r \in \mathbb{N}_0$ and all $k \in \mathbb{N}_0$ and $(s, p) \in \mathcal{A}_k$. Let $k \in \mathbb{N}_0$, $(s, p) \in \mathcal{A}_k \cup \{(k+\frac{1}{2}, 2)\}$, and $f \in W^{s-k-\frac{1}{p},p}(\reftri) \cap W_{\gamma_{2}}^{\min\{s-k- \frac{1}{p}, r+1\} , p}(\reftri)$ be given. Thanks to \cref{eq:proof:mmr-id}, there holds
	\begin{align*}
		\| \mathcal{M}_{k, r+1}^{[1]}(f) \|_{s, p, \reftet} \leq (k+1) \| \mathcal{M}_{k+1, r}^{[1]}[\omega_2 b](\omega_2^{-1} f) \|_{s, p, \reftet} + \| \mathcal{M}_{k, r}^{[1]}(f) \|_{s, p, \reftet}.
	\end{align*}
	
	\noindent \textbf{Part (a):  $k+1/p \leq s \leq k+1+1/p$. } Thanks to \cref{lem:omega1-inv-mapping,}, there holds $\omega_2^{-1} f \in L^p(\reftri; \omega_2^{(k-s+1)p+1} \d{\bdd{x}})$ and \cref{eq:omega1-invs-wsp1-bound,eq:omega1-inv-wsp-bound} give
	\begin{align*}
		\| \omega_2^{k-s+1+\frac{1}{p}} \omega_2^{-1} f \|_{p, \reftri} = \| \omega_2^{k-s+\frac{1}{p}} f \|_{p, \reftri} \lesssim_{k,s,p} \inorm{\gamma_2}{f}{s-k-\frac{1}{p}, \reftri}.
	\end{align*}
	Consequently, we apply \cref{eq:mmr-gamma1-continuity-low} to obtain
	\begin{align*}
		\| \mathcal{M}_{k+1, r}^{[1]}[\omega_2 b](\omega_2^{-1} f) \|_{s, p, \reftet} \lesssim_{b, k, r, s, p} \| \omega_2^{k-s+\frac{1}{p}} f \|_{p, \reftri} \lesssim_{k, s, p} \inorm{\gamma_2}{f}{s-k-\frac{1}{p}, p, \reftri}.
	\end{align*} 
	
	\noindent \textbf{Part (b):  $k+1+1/p < s \leq k + r + 1 + 1/p$. } \Cref{lem:omega1-inv-mapping} shows that $\omega_2^{-1} f \in W_{\gamma_2}^{s-k-1-\frac{1}{p}, p}(\reftri)$ and \cref{eq:proof:mmr-gamma1-continuity-high} and \cref{eq:omega1-inv-wsps-bound} then give
	\begin{align*}
		\| \mathcal{M}_{k+1, r}^{[1]}[\omega_2 b](\omega_2^{-1} f) \|_{s, p, \reftet} \lesssim_{b, k, r, s, p} \inorm{\gamma_2}{\omega_2^{-1} f}{s-k-1-\frac{1}{p}, p, \reftri} \lesssim_{k, s, p} \inorm{\gamma_2}{f}{s-k-\frac{1}{p}, p, \reftri}.
	\end{align*} 
	
	\noindent \textbf{Part (c): $s > k + r + 1 + 1/p$. } Thanks to \cref{lem:omega1-inv-mapping}, there holds $\omega_2^{-1} f \in W^{s-k-1-\frac{1}{p}, p}(\reftri) \cap W^{r}_{\gamma_2}(\reftri)$, and so we apply \cref{eq:proof:mmr-gamma1-continuity-high} and \cref{eq:omega1-inv-wsp-bound} to obtain
	\begin{align*}
		\| \mathcal{M}_{k+1, r}^{[1]}[\omega_2 b](\omega_2^{-1} f) \|_{s, p, \reftet} \lesssim_{b, k, r, s, p} \|\omega_2^{-1} f\|_{s-k-1-\frac{1}{p}, p, \reftri} \lesssim_{k, s, p} \|f\|_{s-k-\frac{1}{p}, p, \reftri}.
	\end{align*}
	Inequality \cref{eq:proof:mmr-gamma1-continuity-high} for $r + 1$ now follows from the triangle inequality. The smoothness of the mapping $\mathfrak{I}_1$ defined in \cref{eq:ek1-gamma1-def} then gives \cref{eq:mmr-gamma1-continuity-high}.
	
	\noindent \textbf{Step 2: Trace properties \cref{eq:mmr-gamma1-interp,eq:mmr-gamma2-zero}. } Direct computation shows that \cref{eq:mmr-gamma1-interp,eq:mmr-gamma2-zero} hold.

	\noindent \textbf{Step 3: Polynomial preservation. } 	Suppose that $f \in \mathcal{P}_N(\Gamma_1)$, $N \in \mathbb{N}_0$, satisfies $D_{\Gamma}^{l} f|_{\gamma_{12}} = 0$ for $0 \leq l \leq r-1$. Then, $f \circ \mathfrak{I}_1 = \omega_2^r g$ for some $g \in \mathcal{P}_{N-r}(\reftri)$, and so $\mathcal{M}_{k, r}^{[1]}(f) = x_2^r \mathcal{E}_{k}^{[1]}(g) \in \mathcal{P}_{N+k}(\reftet)$ thanks to \cref{lem:em-gamma1-cont-high}. \hfill \proofbox
	
	\subsection{Proof of \cref{lem:smr-gamma1-cont-high}}
	\label{sec:proof-lem:smr-gamma1-cont-high}
	
	\noindent \textbf{Step 1: Continuity  \cref{eq:smr-gamma1-continuity-high}. } We first show that the following analogue of \cref{eq:smr-gamma1-continuity-high} holds: Let $b \in C_c^{\infty}(\reftri)$, $k, r \in \mathbb{N}_0$, $(s, p) \in \mathcal{A}_k \cup \{(k+1/2, 2)\}$, and $\mathfrak{E} = \{\gamma_1, \gamma_2\}$. For all $f \in W_{\mathfrak{E}, r}^{s-k- \frac{1}{p}, p}(\reftri)$, there holds
	\begin{align}
		\label{eq:proof:smr-gamma1-continuity-high-fix}
		\| \mathcal{S}_{k, r}^{[1]}(f) \|_{s, p, \reftet} \lesssim_{b, k, r, s, p} 
		\begin{cases}
			\|f\|_{2, \reftri} & \text{if } (s, p) = (k+\frac{1}{2}, 2), \\
		 	\inorm{\mathfrak{E}, r}{f}{s-k-\frac{1}{p}, p, \reftri} & \text{otherwise}.
		 \end{cases}
	\end{align}
	We proceed by induction on $r$. The case $r=0$ follows from \cref{eq:em-gamma1-continuity-high,eq:em-gamma1-continuity-low}. Now let $r \in \mathbb{N}_0$ be given, and assume that \cref{eq:proof:smr-gamma1-continuity-high-fix} holds for all $k \in \mathbb{N}_0$ and $(s, p) \in \mathcal{A}_k \cup \{(k+1/2, 2)\}$. 
	
	Let $k\in \mathbb{N}_0$, $(s, p) \in \mathcal{A}_k \cup \{(k+1/2, 2)\}$, and $f \in W_{\mathfrak{E}, r+1}^{s-k - \frac{1}{p}, p}(\reftri)$ be given. Then, applying \cref{eq:proof:mmr-id,eq:proof:smrq-id} gives
	\begin{align*}
		\mathcal{S}_{k, r+1}^{[1]}(f) &= (k+1) \mathcal{S}_{k+1, r+1, r}^{[1]}[\omega_1 b](\omega_1^{-1} f) + \mathcal{S}_{k, r+1, r}^{[1]}(f) \\
		&= x^r \left( (k+1) \mathcal{M}_{k+1, r+1}^{[1]}[\omega_1 b](\omega_1^{-(r+1)} f) + \mathcal{M}_{k, r+1}^{[1]}(\omega_1^{-r} f) \right) \\
		&= x^r \big[  (k+1)(k+2) \mathcal{M}_{k+2, r}^{[1]}[\omega_1 \omega_2 b](\omega_1^{-(r+1)} \omega_2^{-1} f) \\
		&\quad + \mathcal{M}_{k+1, r}^{[1]}[\omega_1 b](\omega_1^{-(r+1)} f) 
		+ \mathcal{M}_{k+1, r}^{[1]}[\omega_2 b](\omega_1^{-r} \omega_2^{-1} f) + \mathcal{M}_{k+1, r}^{[1]}(\omega_1^{-r} f) \big],
	\end{align*}
	where $S_{k, r, q}^{[1]}$ is defined in \cref{eq:skrq-def}, and so
	\begin{align*}
			\mathcal{S}_{k, r+1}^{[1]}(f) &= (k+1)(k+2) \mathcal{S}_{k+2, r}^{[1]}[\omega_1 \omega_2 b]((\omega_1 \omega_2)^{-1} f) \\
			&\qquad + (k+1) \left(  \mathcal{S}_{k+1, r}^{[1]}[\omega_1 b](\omega_1^{-1} f)  
			+ \mathcal{S}_{k+1, r}^{[1]}[\omega_2 b](\omega_2^{-1} f) \right) + \mathcal{S}_{k, r}^{[1]}(f).
	\end{align*}
	Consequently, we obtain
	\begin{align}
		\label{eq:proof:skr1-triangle-fix}
		\begin{aligned}
			\| \mathcal{S}_{k, r+1}^{[1]}(f) \|_{s, p, \reftet} &\lesssim_k  \| \mathcal{S}_{k, r}^{[1]}(f) \|_{s, p, \reftet} + \sum_{i=1}^{2}  \| \mathcal{S}^{[1]}_{k+1, r}[\omega_i b](\omega_i^{-1} f) \|_{s, p, \reftet} \\
			&\qquad + \| \mathcal{S}^{[1]}_{k+2, r}[\omega_1 \omega_2 b]((\omega_1 \omega_2)^{-1}f) \|_{s, p, \reftet} 
		\end{aligned}
	\end{align}
	
	\noindent \textbf{Part (a). } We first consider the terms $\| \mathcal{S}^{[1]}_{k+1, r}[\omega_i b](\omega_i^{-1} f) \|_{s, p, \reftet}$, $1 \leq i \leq 2$. For $k+1/p \leq s \leq k + 1 + 1/p$, \cref{lem:omega1-inv-mapping} shows that $\omega_i^{-1} f \in L^p(\reftri; \omega_i^{(k-s+1)p + 1} \d{\bdd{x}})$ and \cref{eq:omega1-invs-wsp1-bound,eq:omega1-inv-wsps-bound} gives
	\begin{align*}
		\| \omega_i^{k-s+1+\frac{1}{p}} \omega_i^{-1} f \|_{p, \reftri} = \| \omega_i^{k-s+\frac{1}{p}} f \|_{p, \reftri} \lesssim_{k, s, p} \inorm{\mathfrak{E}}{f}{s-k-\frac{1}{p}, \reftri} = \inorm{\mathfrak{E}, r+1}{f}{s-k-\frac{1}{p}, \reftri}
	\end{align*}
	for $1 \leq i \leq 2$. Applying \cref{eq:smrq-gamma1-continuity-low} then gives
	\begin{align}
		\label{eq:proof:sk1r-cont-fix}
		\| \mathcal{S}_{k+1, r}^{[1]}[\omega_i b](\omega_i^{-1} f) \|_{s, p, \reftet}  \lesssim_{b,k,r,s,p} \inorm{\mathfrak{E}, r+1}{f}{s-k-\frac{1}{p}, \reftri}, \qquad 1 \leq i \leq 2.
	\end{align}
	Now let $s > k + 1 + 1/p$. \Cref{cor:omega-inv-wspr-norm} shows that $\omega_i^{-1} f \in W_{\mathfrak{E}, r}^{s-k-1-\frac{1}{p}, p}(\reftri)$ and \cref{eq:omega1-inv-wspr-norm} then gives
	\begin{align*}
		\inorm{\mathfrak{E}, r}{\omega_i^{-1} f}{s-k-1-\frac{1}{p}, p, \reftri} \lesssim_{k, s, p, r} \inorm{\mathfrak{E}, r+1}{f}{s-k-\frac{1}{p}, p, \reftri}.
	\end{align*}
	Inequality \cref{eq:proof:sk1r-cont-fix} then follows from \cref{eq:proof:smr-gamma1-continuity-high-fix}. 
	
	\noindent \textbf{Part (b). } We now turn to the term $\| \mathcal{S}^{[1]}_{k+2, r}[\omega_1 \omega_2 b]((\omega_1 \omega_2)^{-1}f) \|_{s, p, \reftet}$. Assume first that $k + 1/p \leq s \leq k + 1 + 1/p$. \Cref{lem:omega1-inv-mapping} shows that $(\omega_1 \omega_2)^{-1} f \in L^p(\reftri; \omega_1^p \omega_2^{(k-s+1)p + 1} \d{\bdd{x}})$, and \cref{eq:omega1-invs-wsp1-bound,eq:omega1-inv-wsps-bound} give
	\begin{align*}
		\| \omega_1 \omega_2^{k-s+1+\frac{1}{p}} \omega_1^{-1} \omega_2^{-1} f \|_{p, \reftri}   \lesssim_{k,s,p} \inorm{\mathfrak{E}}{f}{s-k-\frac{1}{p}, \reftri} = \inorm{\mathfrak{E}, r+1}{f}{s-k-\frac{1}{p}, \reftri}.
	\end{align*}
	Applying \cref{eq:smrq-gamma1-continuity-low} then gives
	\begin{align}
		\label{eq:proof:sk2r-cont-fix}
		\| \mathcal{S}_{k+2, r}^{[1]}[\omega_1 \omega_2 b]((\omega_1 \omega_2)^{-1} f) \|_{s, p, \reftet}  \lesssim_{b,k,r,s,p} \inorm{\mathfrak{E}, r+1}{f}{s-k-\frac{1}{p}, \reftri}.
	\end{align}
	
	Now assume that $k + 1 + 1/p < s \leq k + 2 + 1/p$. Thanks to \cref{cor:omega-inv-wspr-norm}, $\omega_2^{-1} f \in W^{s-k-1-\frac{1}{p}, p}_{\gamma_1, r+1}(\reftri)$, and so \cref{lem:omega1-inv-mapping} gives $(\omega_1 \omega_2)^{-1} f \in  L^p(\reftri; \omega_1^{(k-s+2)p + 1} \d{\bdd{x}})$. Inequalities \cref{eq:omega1-inv-wspr-norm-2,eq:omega1-inv-wsps-bound} then give
	\begin{align*}
		\| \omega_1^{k-s+2+\frac{1}{p}} \omega_1^{-1} \omega_2^{-1} f \|_{p, \reftri}  \lesssim_{k,s,p} \inorm{\mathfrak{E}, r}{\omega_2^{-1} f}{s-k-1-\frac{1}{p}, \reftri} \lesssim_{k,s,p} \inorm{\mathfrak{E}, r+1}{f}{s-k-\frac{1}{p}, \reftri}.
	\end{align*}
	Applying \cref{eq:smrq-gamma1-continuity-low} then gives \cref{eq:proof:sk2r-cont-fix}. 
	
	Now assume that $s > k + 2 + 1/p$. Two applications of \cref{cor:omega-inv-wspr-norm} show that $(\omega_1 \omega_2)^{-1} f \in W^{s-k-2-\frac{1}{p}, p}_{\mathfrak{E}, r}(\reftri)$ and \cref{eq:omega1-inv-wspr-norm-2,eq:wsp-e-intersection-id-2} give
	\begin{align*}
		\inorm{\mathfrak{E}, r}{(\omega_1 \omega_2)^{-1}f}{s-k-2-\frac{1}{p}, \reftri} &\lesssim_{k, s, p, r} \inorm{\gamma_1, r}{\omega_1^{-1} f}{s-k-1-\frac{1}{p}, \reftri} + \inorm{\gamma_2, r+1}{\omega_1^{-1} f}{s-k-1-\frac{1}{p}, \reftri} \\
		&\lesssim_{k, s, p, r} \inorm{\gamma_1, r+1}{f}{s-k-\frac{1}{p}, \reftri} + \inorm{\gamma_2, r+1}{f}{s-k-\frac{1}{p}, \reftri} \\
		&\lesssim_{k, s, p} \inorm{\mathfrak{E}, r+1}{f}{s-k-\frac{1}{p}, \reftri}.
	\end{align*}	
	Applying \cref{eq:proof:smr-gamma1-continuity-high-fix} then gives \cref{eq:proof:sk2r-cont-fix}. Inequality \cref{eq:proof:smr-gamma1-continuity-high-fix} for $r+1$ now follows from the triangle inequality, \cref{eq:proof:sk1r-cont-fix}, and \cref{eq:proof:sk2r-cont-fix}. The smoothness of the mapping $\mathfrak{I}_1$ defined in \cref{eq:ek1-gamma1-def} then gives \cref{eq:smr-gamma1-continuity-high}. 
	
	\noindent \textbf{Step 2: Trace properties \cref{eq:smr-gamma1-interp,eq:smr-gamma2-zero}. } Direct computation shows that \cref{eq:smr-gamma1-interp,eq:smr-gamma2-zero} hold.
	
	\noindent \textbf{Step 3: Polynomial preservation. } 	Suppose that $f \in \mathcal{P}_N(\Gamma_1)$, $N \in \mathbb{N}_0$, satisfies $D_{\Gamma}^{l} f|_{\gamma_{12}} = D_{\Gamma}^{l} f|_{\gamma_{13}} = 0$ for $0 \leq l \leq r-1$. Then, $f \circ \mathfrak{I}_1 = (\omega_1 \omega_2)^r g$ for some $g \in \mathcal{P}_{N-2r}(\reftri)$, and so $\mathcal{S}_{k, r}^{[1]}(f) = (x_1 x_2)^r \mathcal{E}_{k}^{[1]}(g) \in \mathcal{P}_{N+k}(\reftet)$ thanks to \cref{lem:em-gamma1-cont-high}. \hfill \proofbox

	\subsection{Proof of \cref{lem:rmr-gamma1-cont-high}}
	\label{sec:proof-lem:rmr-gamma1-cont-high}
	
	\textbf{Step 1: Continuity \cref{eq:rmr-gamma1-continuity-high}. } We first show that the following analogue of \cref{eq:rmr-gamma1-continuity-high} holds: For $b \in C_c^{\infty}(\reftri)$, $k, r \in \mathbb{N}_0$, $(s, p) \in \mathcal{A}_k$, and $\mathfrak{E} = \{\gamma_1, \gamma_2, \gamma_3\}$, there holds 
	\begin{align}
		\label{eq:proof:rmr-gamma1-continuity-high}
		\| \mathcal{R}_{k, r}^{[1]}(f) \|_{s, p, \reftet} \lesssim_{b, k, r, s, p} \inorm{\mathfrak{E}, r}{f}{s-k-\frac{1}{p}, p, \reftri} \qquad \forall f \in W^{s-k-\frac{1}{p}, p}_{\mathfrak{E}, r}(\reftri).
	\end{align}
	\textbf{Part (a): Variants of $\mathcal{S}_{k, r}^{[1]}$. } 
	We begin with a brief aside. Let $\mathfrak{E}_{ij} = \{\gamma_i, \gamma_j\}$ for $1 \leq i < j \leq 3$. Formally define the following analogue of $\mathcal{S}_{k, r}^{[1]}$ \cref{eq:skr-def}:
	\begin{alignat*}{2}
		\mathcal{S}_{k, r}^{[1], (13)}(f)(\bdd{x}, z) &:= (x_1 (1-x_1-x_2-z))^r \mathcal{E}_k^{[1]}((\omega_1 \omega_3)^{-r} f)(\bdd{x}, z) \qquad & & \\
		&= \mathcal{S}_{k, r}^{[1]}[b \circ \mathfrak{F}_2](f \circ \mathfrak{F}_2) \circ \mathfrak{G}_2(\bdd{x}, z), \qquad & &(\bdd{x}, z) \in \reftet,
	\end{alignat*}
	where $\mathfrak{F}_2$ and $\mathfrak{G}_2$ are defined in \cref{eq:proof:frakf2-def}. Note that for any $s \geq 0$ and $r \in \mathbb{N}_0$, there holds $f \in W^{s, p}_{\mathfrak{E}_{13}, r}(\reftri)$ if and only if $f \circ \mathfrak{F}_2 \in W^{s, p}_{\mathfrak{E}_{12}, r}(\reftri)$. Thanks to \cref{lem:smr-gamma1-cont-high}, for $b \in C^{\infty}_c(\reftri)$, $k, r \in \mathbb{N}_0$, $(s, p) \in \mathcal{A}_k$, there holds
	\begin{align}
		\label{eq:proof:smr13-gamma1-continuity-high}
		\| \mathcal{S}_{k, r}^{[1], (13)}(f) \|_{s, p, \reftet} \lesssim_{b, k, r, s, p} \inorm{\mathfrak{E}_{13}, r}{f}{s-k-\frac{1}{p}, p, \reftri} \qquad \forall f \in W^{s-k-\frac{1}{p}, p}_{\mathfrak{E}_{13}, r}(\reftri),
	\end{align}
	where we used that $\|f\|_{t, p, \reftri} \approx_{t, p} \|f \circ \mathfrak{F}_2\|_{t, p, \reftri}$ and $\inorm{\mathfrak{E}_{13}, r}{f}{t, p, \reftri} \approx_{t, p} \inorm{\mathfrak{E}_{12}, r}{f \circ \mathfrak{F}_2}{t, p, \reftri}$. Analogous arguments show that the operator
	\begin{alignat*}{2}
		\mathcal{S}_{k, r}^{[1], (23)}(f)(\bdd{x}, z) &:= (x_2 (1-x_1-x_2-z))^r \mathcal{E}_k^{[1]}((\omega_2 \omega_3)^{-r} f)(\bdd{x}, z) \qquad & & \\
		&= \mathcal{S}_{k, r}^{[1]}[b \circ \mathfrak{F}_2^{-1}](f \circ \mathfrak{F}_2^{-1}) \circ \mathfrak{G}_2^{-1}(\bdd{x}, z) \qquad & &(\bdd{x}, z) \in \reftet
	\end{alignat*}
	satisfies the following for $b \in C^{\infty}_c(\reftri)$, $k, r \in \mathbb{N}_0$, $(s, p) \in \mathcal{A}_k$:
	\begin{align}
		\label{eq:proof:smr23-gamma1-continuity-high}
		\| \mathcal{S}_{k, r}^{[1], (23)}(f) \|_{s, p, \reftet} \lesssim_{b, k, r, s, p} \inorm{\mathfrak{E}_{23}, r}{f}{s-k-\frac{1}{p}, p, \reftri} \qquad \forall f \in W^{s-k-\frac{1}{p}, p}_{\mathfrak{E}_{23}, r}(\reftri).
	\end{align}
	 \textbf{Part (b): Key identity for $\mathcal{R}_{k, r}^{[1]}$. } Thanks to \cref{lem:partial-frac-decomp}, there holds
	 \begin{align*}
	 	\mathcal{R}_{k, r}^{[1]}(f) &= (x_1 x_2 (1-x_1-x_2-z))^r  \sum_{\substack{\alpha \in \mathbb{N}_0^3 \\ \alpha_j \leq k \\ |\alpha| \geq 2 }} \mathcal{E}_k^{[1]} \left( \frac{ c_{\alpha,1} f }{\omega_1^{\alpha_1} \omega_2^{\alpha_2}} + \frac{ c_{\alpha,2} f }{\omega_1^{\alpha_1} \omega_3^{\alpha_3}} + \frac{ c_{\alpha,3} f }{\omega_2^{\alpha_2} \omega_3^{\alpha_3}} \right) \\
	 	&= \sum_{1 \leq i < j \leq 3} \lambda_{m(i, j)}^{r} \sum_{l=1}^{r} (\lambda_i \lambda_j)^{r-l} \sum_{n=0}^{l} \left( d_{ln}^{(ij)} S_{k, l}^{[1], (ij)}(\omega_i^{n} f ) + d_{ln}^{(ji)} S_{k, l}^{[1], (ij)}(\omega_j^{n} f ) \right),
	 \end{align*}
 	where $\lambda_1 := x_1$, $\lambda_2 := x_2$, $\lambda_3 := 1 - x_1 - x_2 - z$, $m(i, j)$ is the lone element of $\{1,2,3\} \setminus \{i, j\}$, $d_{ln}^{(ij)}$ and $d_{ln}^{(ji)}$ are suitable constants, and $S_{k, r}^{[1], (12)} := \mathcal{S}_{k, r}^{[1]}$.
 	
 	Let $b \in C_c^{\infty}(\reftri)$, $k, r \in \mathbb{N}_0$, $(s, p) \in \mathcal{A}_k$, and $f \in W^{s-k-\frac{1}{p}, p}_{\mathfrak{E}, r}(\reftri) $ be given. For any $n \in \mathbb{N}_0$ and real $t \geq 0$, the mapping $g \mapsto \omega_i^{n} g$ is continuous from $W^{t, p}_{\mathfrak{E}, r}(\reftri)$ to $W^{t, p}_{\mathfrak{E}, r}(\reftri)$. Similarly, for any $\alpha \in \mathbb{N}_0^3$, the mapping $g \mapsto \lambda_1^{\alpha_1} \lambda_2^{\alpha_2} \lambda_3^{\alpha_3} g$ is continuous from $W^{s, p}(\reftet)$ to $W^{s, p}(\reftet)$. Consequently, \cref{eq:proof:rmr-gamma1-continuity-high} follows from the triangle inequality, \cref{eq:wsp-e-intersection-id}, \cref{eq:smr-gamma1-continuity-high}, \cref{eq:proof:smr13-gamma1-continuity-high}, and \cref{eq:proof:smr23-gamma1-continuity-high}. The smoothness of the mapping $\mathfrak{I}_1$ \cref{eq:ek1-gamma1-def} then gives \cref{eq:rmr-gamma1-continuity-high}.
	
	\noindent \textbf{Step 2: Trace properties \cref{eq:rmr-gamma1-interp,eq:rmr-gamma2-zero}. } Direct computation shows that \cref{eq:rmr-gamma1-interp,eq:rmr-gamma2-zero} hold.
	
	\noindent \textbf{Step 3: Polynomial preservation. } 	Suppose that $f \in \mathcal{P}_N(\Gamma_1)$, $N \in \mathbb{N}_0$, satisfies $D_{\Gamma}^{l} f|_{\partial \reftri} = 0$ for $0 \leq l \leq r-1$. Then, $f \circ \mathfrak{I}_1 = (\omega_1 \omega_2 \omega_3)^r g$ for some $g \in \mathcal{P}_{N-3r}(\reftri)$, and so $\mathcal{R}_{k, r}^{[1]}(f) = (x_1 x_2 (1-x_1-x_2-z))^r \mathcal{E}_{k}^{[1]}(g) \in \mathcal{P}_{N+k}(\reftet)$ thanks to \cref{lem:em-gamma1-cont-high}. \hfill \proofbox

	\appendix
	
	\section{Properties of spaces with vanishing traces}
	
	In this section, we show that smooth functions with vanishing traces are dense in the space $W^{s, p}_{\mathfrak{E}}(\reftri)$ \cref{eq:wsp-e-def} and that functions in $W^{s, p}_{\mathfrak{E}, r}(\reftri)$ \cref{eq:wsp-er-def} satisfy a Hardy inequality.
	
	\subsection{A Density result}
	
	We begin with a density result for the spaces $W^{s,p}_{\mathfrak{E}}(\reftri)$ defined in \cref{sec:two-faces}.
	\begin{lemma}
		\label{lem:wsp1-smooth-density}
		Let $\mathfrak{E} \subseteq \{ \gamma_1, \gamma_2, \gamma_3 \}$ and define
		\begin{align*}
			C^{\infty}_{\mathfrak{E}}(\reftri) &:= \left\{ \phi \in C^{\infty}(\bar{\reftri}) : \bigcup_{\gamma \in \mathfrak{E}} \gamma \cap \supp \phi = \emptyset \right\}.
		\end{align*}
		For $1 < p < \infty$ and $0 \leq s < \infty$, the space $C^{\infty}_{\mathfrak{E}}(\reftri)$ is dense in $W^{s,p}_{\mathfrak{E}}(\reftri)$.
	\end{lemma}
	\begin{proof}		
		Let $\mathfrak{E} \subseteq \{ \gamma_1, \gamma_2, \gamma_3 \}$ and $1 < p < \infty$ be given.
		
		\noindent \textbf{Step 1: $0 \leq s < 1/p$. } The space $C^{\infty}_c(\reftri) \subseteq C^{\infty}_{\mathfrak{E}}(\reftri)$ is dense in $W^{s, p}(\reftri) = W_{\mathfrak{E}}^{s, p}(\reftri)$ (see e.g. \cite[Theorem 1.4.5.2]{Grisvard85}).
		
		\noindent \textbf{Step 2: $s \geq 1/p$ and $\mathfrak{E} = \{\gamma_1\}$. } Let $s = m + \sigma$ with $m \in \mathbb{N}_0$ and $\sigma \in [0, 1)$, and let $f \in W^{s, p}_{\mathfrak{E}}(\reftri)$. For $n \in \mathbb{N}$, we construct a partition of unity on $\reftri$ as follows. Let $\{ \bdd{a}_i \}_{i=1}^{3}$ denote the vertices of $T$ labeled counterclockwise as in \cref{fig:reference triangle} and define the following sets:
		\begin{alignat*}{2}
			\mathcal{U}_0 &:= \left\{ \bdd{x} \in \reftri : \dist(\bdd{x}, \partial \reftri) > \frac{1}{2n} \right\}, \qquad & & \\
			\mathcal{U}_{(i-1)(n-1) + j} = \mathcal{U}_{j}^{(i)} &:= B\left(\bdd{a}_{i+2} + \frac{j}{n} \unitvec{t}_{i}, \frac{3}{4n}\right) \cap \bar{T}, \qquad & &1 \leq j \leq n-1, \ 1 \leq i \leq 3, \\
			\mathcal{U}_{3n - 3 + k} &:= B\left(\bdd{a}_k, \frac{3}{4n}\right) \cap \bar{T}, \qquad & &1 \leq k \leq 3,
		\end{alignat*}
		where we use the notation $B(\bdd{x}, r)$ to denote the ball of radius $r$ centered at $\bdd{x}$. By construction, $\reftri \subset \bigcup_{i=0}^{3n} \mathcal{U}_{i}$, and so there exists a partition of unity $\{ \phi_i \in C^{\infty}_c(\mathcal{U}_i) : 0 \leq i \leq 3n\}$ satisfying 
		\begin{align*}
			\sum_{i=0}^{3n} \phi_i = 1 \quad \text{and} \quad \| D^k \phi_i \|_{ \infty, \mathcal{U}_i} \lesssim_{k} n^k, \qquad 0 \leq i \leq 3n, \ \forall k \in \mathbb{N}_0.
		\end{align*} 
		We denote $f_i := \phi_i f$ for $0 \leq i \leq 3n-3$ and set $\mathcal{V}_i := \mathcal{U}_i \cap T$ for $0 \leq i \leq 3n$.
		
		Let $\{ \delta_i \}_{i=0}^{3n-3}$ be arbitrary positive constants. The construction proceeds in several parts.
		
		\noindent \textbf{Part (a).} The function $f_0 := \phi_0 f$ satisfies
		\begin{alignat*}{2}
			D^l f_0|_{\partial \mathcal{U}_0} &= 0, \qquad & &0 \leq l < s - \frac{1}{p}, \\
			\| \dist(\cdot, \partial \mathcal{U}_0)^{-\sigma} D^m f_0 \|_{p, \mathcal{U}_0} \lesssim_{m, p} n^{\frac{1}{p}+m} \|f\|_{m, p, \mathcal{U}_0} &< \infty \qquad& &\text{if $\sigma p=1$}.
		\end{alignat*} 
		By \cite[Theorem 1.4.5.2]{Grisvard85}, there exists a $\psi_0 \in C^{\infty}_c(\mathcal{U}_0)$ satisfying
		\begin{align*}
			\delta_0 \geq \| f_0 - \psi_0 \|_{s, p, \mathcal{U}_0} + \begin{cases}
				\| \dist(\cdot, \partial \mathcal{U}_0)^{-\sigma} D^m (f_0 - \psi_0) \|_{p, \mathcal{U}_0} & \text{if } \sigma p = 1, \\
				0 & \text{otherwise}.
			\end{cases}
		\end{align*}
		Since $f \in W^{s, p}(T)$, we apply the same argument to $f_i := \phi_i f$ on $\mathcal{V}_i$, $1 \leq i \leq n-1$, to show that there exists a $\psi_i \in C^{\infty}_c(\mathcal{V}_i)$ satisfying
		\begin{align*}
			\delta_i \geq \| f_i - \psi_i \|_{s, p, \mathcal{V}_i} + \begin{cases}
				\| \dist(\cdot, \partial \mathcal{V}_i)^{-\sigma} D^m(f_i - \psi_i) \|_{p, \mathcal{V}_i} & \text{if } \sigma p = 1, \\
				0 & \text{otherwise}.
			\end{cases}
		\end{align*}
		
		\noindent \textbf{Part (b).} For $n \leq i \le 3n-2$, $f \in W^{s, p}(\mathcal{V}_i)$, and so there exists a $\rho_i \in C^{\infty}(\bar{\mathcal{V}}_i)$ satisfying $\|f - \rho_i\|_{s, p, \mathcal{V}_i} \leq \delta_i n^{m+2}$ thanks to \cite[Theorem 1.4.5.2]{Grisvard85}. Then, the function $\psi_i := \phi_i \rho_i$ satisfies
		\begin{align*}
			\| f_i - \psi_i \|_{s, p, \mathcal{V}_i} &\lesssim_{s,p} \|\phi_i\|_{m, \infty, \mathcal{V}_i} \|f - \rho_i\|_{m, p, \mathcal{V}_i} +  \sum_{l=0}^{m}  \| D^l \phi_i D^{m-l}(f - \rho_i) \|_{\sigma, p, \mathcal{V}_i}  
			\lesssim_{s, p} \delta_i,
		\end{align*} 
		where we used \cite[Theorem 6.3]{Leoni23} to conclude that
		\begin{align*}
			\| D^l \phi_i D^{m-l}(f - \rho_i) \|_{\sigma, p, \mathcal{V}_i} &\lesssim_{s, p} \|D^l \phi_i \|_{\infty, \mathcal{V}_i}^{1-\sigma} \| D^{l+1} \phi_i \|_{\infty, \mathcal{V}_i}^{\sigma} \|D^{m-l}(f - \rho_i)\|_{p, \mathcal{U}_i} \\
			&\qquad \qquad + \|D^l \phi_i\|_{\infty, \mathcal{V}_i} \|D^{m-l}(f - \rho_i)\|_{\sigma, p, \mathcal{V}_i} \\
			&\leq \delta_i.
		\end{align*}
		Moreover, when $\sigma p=1$, we have
		\begin{align*}
			\| \omega_1^{-\sigma} D^m(f_i - \psi_i) \|_{p, \mathcal{V}_i} \lesssim_p n^{\sigma} \|D^m (f_i - \psi_i) \|_{p, \mathcal{V}_i} \lesssim_{s, p} \delta_i.
		\end{align*}
		
		\noindent \textbf{Part (c).} For $3n-1 \leq i \leq 3n$, we will show that
		\begin{align}
			\label{eq:proof:corner-contrib-vanish}
			\lim_{n\to\infty} \| f_i \|_{s, p, \mathcal{V}_i} = 0 \quad \text{and if $\sigma p=1$,} \quad \lim_{n\to\infty} \| \omega_1^{-\sigma} D^m f_i \|_{p, \mathcal{V}_i} = 0.
		\end{align}
		Thanks to \cite[Theorem 6.3]{Leoni23}, there holds
		\begin{align*}
			\| f_i \|_{s, p, \mathcal{V}_i} &\lesssim_{s, p} \sum_{j=0}^{m} \sum_{l=0}^{j} \|D^l \phi_i\|_{\infty, \mathcal{V}_i} \|D^{j-l} f \|_{p, \mathcal{V}_i} +  \sum_{l=0}^{m} \|D^l \phi_i\|_{\infty, \mathcal{V}_i} \|D^{m-l}f \|_{\sigma, p, \mathcal{V}_i} \\
			&\qquad + \sum_{l=0}^{m}   \|D^l \phi_i \|_{\infty, \mathcal{V}_i}^{1-\sigma} \| D^{l+1} \phi_i \|_{\infty, \mathcal{V}_i}^{\sigma} \|D^{m-l}f \|_{p, \mathcal{U}_i}   \\
			&\lesssim_{s, p} \sum_{j=0}^{m} \sum_{l=0}^{j} n^{l} \|D^{j-l} f \|_{p, \mathcal{V}_i} + \sum_{l=0}^{m}  \left(  n^{l+\sigma} \|D^{m-l}f \|_{p, \mathcal{U}_i} + n^l \|D^{m-l}f \|_{\sigma, p, \mathcal{V}_i} \right).
		\end{align*}
		Similar computations show that for $\sigma p =1 $, there holds
		\begin{align*}
			\| \omega_1^{-\sigma} D^m f_i\|_{p, \mathcal{V}_i} \lesssim_{s, p}  \sum_{j=0}^{m} n^j \|\omega_1^{-\sigma} D^{m-j} f\|_{p, \mathcal{V}_i}.
		\end{align*}
		Since $\mathcal{V}_i \cap \gamma_1 \neq \emptyset$,  Poincar\'{e}'s inequality gives
		\begin{align*}
			\|D^r f\|_{p, \mathcal{V}_i} \lesssim_{r, p} n^{-(s-r)} |D^m f|_{\sigma, p, \mathcal{V}_i} \qquad 0 \leq r \leq m,
		\end{align*}	
		and so $\| f_i \|_{s, p, \mathcal{V}_i} \lesssim_{s, p}  |D^m f|_{\sigma, p, \mathcal{V}_i}$. Moreover, if $\sigma p=1$, then $D^{m-j} f \in W_{\mathfrak{E}}^{1,p}(\mathcal{V}_i)$ for $1 \leq j \leq m$, and so \cite[Theorem 5.2]{Brewster14}, \cite[Theorem 3.2]{Egert15}, and a standard scaling argument give
		\begin{align*}
			\|\omega_1^{-\sigma} D^{m-j} f\|_{p, \mathcal{V}_i} \lesssim n^{\sigma-1} \| \omega_1^{-1} D^{m-j} f\|_{p, \mathcal{V}_i} &\lesssim_{p} n^{\sigma-1} \| D^{m-j+1} f\|_{p, \mathcal{V}_i} \\
			&\lesssim_{s, p} n^{-j} |D^m f|_{\sigma, p, \mathcal{V}_i},
		\end{align*}
		and so $\| \omega_1^{-\sigma} D^m f_i\|_{p, \mathcal{V}_i}  \lesssim_{s, p} |D^m f|_{\sigma, p, \mathcal{V}_i}$. Equality \cref{eq:proof:corner-contrib-vanish} now follows from that fact that $ |D^m f|_{\sigma, p, \mathcal{V}_i} \to 0$ as $n \to \infty$ since $|\mathcal{V}_i| \to 0$ as $n \to \infty$.
		
		\noindent \textbf{Part (d).} Let $\epsilon > 0$ be given. First, choose $n$ large enough so that 
		\begin{align*}
			\frac{\epsilon}{2} \geq \sum_{i=3n-1}^{3n} \| f_i \|_{s, p, \mathcal{V}_i} + \begin{cases}
				\sum_{i=3n-1}^{3n} 	\| \omega_1^{-\sigma} D^k f_i\|_{p, \mathcal{V}_i}  & \text{if } \sigma p=1, \\
				0 & \text{otherwise}.
			\end{cases}
		\end{align*}
		Then, for $\{\delta_i\}_{i=0}^{3n-2}$ chosen sufficiently small, we construct $\psi_i$ as above so that
		\begin{align*}
			\frac{\epsilon}{2} \geq \sum_{i=0}^{3n-2} \| f_i - \psi_i \|_{s, p, \mathcal{V}_i} + \begin{cases}
				\sum_{i=3n-1}^{3n} 	\| \omega_1^{-\sigma} D^m (f_i - \psi_i)\|_{p, \mathcal{V}_i}  & \text{if } \sigma p=1, \\
				0 & \text{otherwise}.
			\end{cases}
		\end{align*}
		Let $\tilde{\psi}_i$ denote the extension of $\psi_i$ by zero to $\reftri \setminus \mathcal{U}_i$, $0 \leq i \leq 3n-2$ and set $\tilde{\psi}_{j} \equiv 0$ for $3n-1 \leq j \leq 3n$. Then, $\tilde{\psi}_i \in C^{\infty}_{\mathfrak{E}}(\reftri)$, and the function $\psi = \sum_{i=0}^{3n} \tilde{\psi}_i$ then satisfies $\psi \in C^{\infty}_{\mathfrak{E}}(\reftri)$ and
		\begin{align*}
			\inorm{\mathfrak{E}}{f - \psi}{s, p, T} &\leq \sum_{i=0}^{3n} \| f_i - \psi_i \|_{s, p, \mathcal{V}_i} + \begin{cases}
				\sum_{i=0}^{3n} 	\| \omega_1^{-\sigma} D^m (f_i - \psi_i)\|_{p, \mathcal{V}_i}  & \text{if } \sigma p=1, \\
				0 & \text{otherwise},
			\end{cases} \\
			&\lesssim_{s, p} \epsilon,
		\end{align*}
		which shows that $C^{\infty}_{\mathfrak{E}}(\reftri)$ is dense in $W^{s,p}_{\mathfrak{E}}(\reftri)$. 
		
		\noindent \textbf{Step 3: $\mathfrak{E} = \{\gamma_2\}$ or $\mathfrak{E} = \{\gamma_3\}$. } If $\mathfrak{E} = \{\gamma_2\}$, the density of $C^{\infty}_{\gamma_2}(\reftri)$ in $W_{\gamma_2}^{s, p}(\reftri)$ follows from the fact that $f \in W_{\gamma_2}^{s, p}(\reftri)$ if and only if $f \circ \mathfrak{F}_2^{-1} \in W_{\gamma_1}^{s, p}(\reftri)$, where $\mathfrak{F}_2^{-1}(x_1, x_2) = (1-x_1-x_2, x_1)$ is the inverse of $\mathfrak{F}_2$ defined in \cref{eq:proof:frakf2-def}. The case $\mathfrak{E} = \{\gamma_3\}$ follows from similar arguments using the mapping $\mathfrak{F}_2$.

		\noindent \textbf{Step 4: $|\mathfrak{E}| = 2$. } Now let $\mathfrak{E} = \{\gamma_1,\gamma_2\}$. The density of $C^{\infty}_{\mathfrak{E}}(\reftri)$ in $W^{s,p}_{\mathfrak{E}}(\reftri)$ may be shown using a similar construction to the case $\mathfrak{E} = \{\gamma_1\}$. In particular, we apply the construction of Step 1 Part (a) for $0 \leq i \leq 2n-2$ and $i = 3n$, Part (b) for $2n-1 \leq i \leq 3n-3$, Part (c) for $3n-2 \leq i \leq 3n-1$, and proceed analogously as in Part (d). The remaining cases for $|\mathfrak{E}| = 2$ are proved along similar lines.
		
		\noindent \textbf{Step 5: $\mathfrak{E} = \{\gamma_1,\gamma_2,\gamma_3\}$. } This case is a restatement of \cite[Lemma 1.4.5.2]{Grisvard85}.		
	\end{proof}

	\subsection{Hardy inequalities}
	
	First, we construct a bounded averaging operator.
	
	\begin{lemma}
		\label{lem:hardy1}
		There exists a linear operator $\mathcal{H}_1$ satisfying the following properties: 
		\begin{enumerate}
			\item[(i)] $\mathcal{H}_1$ maps $C(\bar{\reftri})$ boundedly into $C(\bar{\reftri})$, and there holds	
			\begin{align}
				\label{eq:hardy-average-op-def}
				\mathcal{H}_1(f)(\bdd{x}) = \frac{1}{x_1} \int_{0}^{x_1} f(u, x_2) \d{u} = \int_{0}^{1} f(u x_1 , x_2) \d{u} \qquad \forall \bdd{x} \in \reftri.
			\end{align}
			
			\item[(ii)] $\mathcal{H}_1$ maps $W^{s, p}_{\mathfrak{E}, r}(\reftri)$ boundedly into $W^{s, p}_{\mathfrak{E}, r}(\reftri)$ and for all $p \in (1,\infty)$, $s \in [0, \infty)$, $r \in \mathbb{N}_0$, and $\mathfrak{E} \in \{ \emptyset, \{\gamma_1\}, \{\gamma_1, \gamma_2\} \}$. In particular,
			\begin{alignat}{2}
				\label{eq:hardy1-wsps-cont}
				\inorm{\mathfrak{E}, r}{ \mathcal{H}_1(f)}{s, p, \reftri} &\lesssim_{s, p, r} \inorm{\mathfrak{E}, r}{f}{s, p, \reftri} \qquad & &\forall f \in W_{\mathfrak{E}, r}^{s, p}(\reftri).
			\end{alignat}

		\end{enumerate}
	\end{lemma}
	\begin{proof}
		\textbf{Step 1:  Continuity on $C(\bar{T})$. } Let $f \in C(\bar{\reftri})$ and define $\mathcal{H}_1(f)$ by \cref{eq:hardy-average-op-def}. Elementary arguments show that $\mathcal{H}_1(f) \in C(\bar{\reftri})$ with $\|\mathcal{H}_1(\phi_i)\|_{\infty, \reftri} \leq \|\phi_i\|_{\infty, \reftri}$.		
		
		\noindent \textbf{Step 2: Extension to $W^{s, p}_{\mathfrak{E}, r}(\reftri)$ when $\mathfrak{E} = \emptyset$. } Let $f \in C^{\infty}(\bar{\reftri})$ and $1 < p < \infty$. For $\alpha \in \mathbb{N}_0^2$, there holds
		\begin{align}
			\label{eq:proof:hardy-derivative-id}
			D^{\alpha} \mathcal{H}_1(f)(\bdd{x}) &= \int_{0}^{1} u^{\alpha_1} (D^{\alpha} f)(u x_1, x_2) \d{u} = \frac{1}{x_1^{\alpha_1+1}} \int_{0}^{x_1} u^{\alpha_1} (D^{\alpha} f)(u, x_2) \d{u},
		\end{align}
		and so $\mathcal{H}_1(f) \in C^{\infty}(\bar{\reftri})$. Moreover, Hardy's inequality \cite[Theorem 327]{Hardy52} gives
		\begin{align*}
			\| 	D^{\alpha} \mathcal{H}_1(f) \|_{p, \reftri}^p
			& \leq \int_{0}^{1} \int_{0}^{1-x_2} \left| \frac{1}{x_1} \int_{0}^{x_1} |D^{\alpha} f(u, x_2)| \d{u} \right|^p \d{x_1} \d{x_2} \\
			&\leq \left( \frac{p}{p-1} \right)^p  \int_{0}^{1} \int_{0}^{1-x_2}  |D^{\alpha} f(x_1, x_2)|^p \d{x_1} \d{x_2}.
		\end{align*}
		Consequently, we obtain
		\begin{align*}
			\| \mathcal{H}_1(f) \|_{s, p, \reftri} &\lesssim_{s, p} \| f \|_{s, p, \reftri} \qquad \forall f \in C^{\infty}(\bar{\reftri}), \ s \in \mathbb{N}_0.
		\end{align*}
		Since $C^{\infty}(\bar{\reftri})$ is dense in $W^{s, p}(\reftri)$ \cite[Theorem 1.4.5.2]{Grisvard85}, $\mathcal{H}_1$ can be continuously extended to a linear operator from $W^{s, p}(\reftri)$ into $W^{s, p}(\reftri)$ for $s \in \mathbb{N}_0$. The case for non-integer $s \in (0,\infty)$ follows from interpolation.		
		
		\noindent \textbf{Step 3: Inequality \cref{eq:hardy1-wsps-cont} when $\mathfrak{E} = \{\gamma_1\}$. } The case $r = 0$ follows from Step 2, so let $r \in \mathbb{N}$. Assume first the $s \leq r$ and let $f \in C^{\infty}_{\gamma_1}(\reftri)$. \Cref{eq:proof:hardy-derivative-id} shows that $\mathcal{H}_1(f) \in C^{\infty}_{\gamma_1}(\reftri)$. Moreover, for $s = m + 1/p$, $m \in \mathbb{N}_0$, we apply Hardy's inequality \cite[Theorem 327]{Hardy52} to obtain
		\begin{align}
			\| \omega_1^{-\frac{1}{p}} D^{\alpha} \mathcal{H}_1(f) \|_{p, \reftri}^p 
			&\leq \int_{0}^{1} \int_{0}^{1-x_2} \left( \frac{1}{x_1} \int_{0}^{x_1} \frac{1}{u^{\frac{1}{p}}} |D^{\alpha} f(u, x_2)| \d{u} \right)^p \d{x_1} \d{x_2} \notag \\
			\label{eq:proof:h1-weighted-cont}
			&\leq \left( \frac{p}{p-1} \right)^p \| \omega_1^{-\frac{1}{p}} D^{\alpha} f \|_{p, \reftri}^p 
		\end{align}
		for all $\alpha \in \mathbb{N}_0$ with $|\alpha| = m$. Thus, $\inorm{\gamma_1}{\mathcal{H}_1(f)}{s, p, \reftri} \lesssim_{s, p} \inorm{\gamma_1}{f}{s, p, \reftri}$ for all $f \in C^{\infty}_{\gamma_1}(\reftri)$. By density (\cref{lem:wsp1-smooth-density}),  $\mathcal{H}_1$ maps $W_{\gamma_1, r}^{s, p}(\reftri)$ boundedly into $W_{\gamma_1, r}^{s, p}(\reftri)$ for all $p \in (1,\infty)$ and $s \in [0, r]$.
		
		Now let $s > r$ and $f \in W_{\gamma_1, r}^{s, p}(\reftri)$. Step 2 and the arguments above show that $\mathcal{H}_1(f) \in W^{s, p}(\reftri) \cap W^{r, p}_{\gamma_1}(\reftri)$, and so $\mathcal{H}_1(f) \in W^{s, p}_{\gamma_1, r}(\reftri)$ if $s-1/p \notin \mathbb{Z}$ with $\inorm{\gamma_1, r}{\mathcal{H}_1(f)}{s, p, \reftri} \lesssim_{s, p}  \inorm{\gamma_1, r}{f}{s, p, \reftri}$. Now let $s = m + 1/p$ for some $m \in \mathbb{N}$. Then, $f \in C(\bar{\reftri})$ and thanks to Step 1 and \cref{eq:proof:h1-weighted-cont}, we have
		\begin{align*}
				\left\| \omega_1^{-\frac{1}{p}} \frac{\partial^{m-r-1} D^{r-1} \mathcal{H}_1(f)}{\partial x_2^{m-r-1}} \right\|_{p, \reftri} \lesssim_{p} \left\| \omega_1^{-\frac{1}{p}} \frac{\partial^{m-r-1} D^{r-1} f}{\partial x_2^{m-r-1}} \right\|_{p, \reftri}.
		\end{align*}
		Consequently, $\mathcal{H}_1(f) \in W^{s, p}_{\gamma_1, r}(\reftri)$ and  \cref{eq:hardy1-wsps-cont} holds when $\mathfrak{E} = \{\gamma_1\}$.
		
		\noindent \textbf{Step 4: Inequality \cref{eq:hardy1-wsps-cont} when $\mathfrak{E} = \{\gamma_1, \gamma_2\}$. } Again let $r \in \mathbb{N}$. Assume first that $s \leq r$. As above, $\mathcal{H}_1(f) \in C^{\infty}_{\mathfrak{E}}(\reftri)$ for any $f \in C^{\infty}_{\mathfrak{E}}(\reftri)$ by \cref{eq:proof:hardy-derivative-id}, and for $s = m + 1/p$, $m \in \mathbb{N}_0$, Hardy's inequality \cite[Theorem 327]{Hardy52} gives
		\begin{align}
			\| \omega_2^{-\frac{1}{p}} D^{\alpha} \mathcal{H}_1(f) \|_{p, \reftri}^p 
			&\leq \int_{0}^{1} \frac{1}{x_2} \int_{0}^{1-x_2} \left( \frac{1}{x_1} \int_{0}^{x_1} |D^{\alpha} f(u, x_2)| \d{u} \right)^p \d{x_1} \d{x_2} \notag \\
			\label{eq:proof:h1-weighted-2-cont}
			&\leq \left( \frac{p}{p-1} \right)^p \| \omega_2^{-\frac{1}{p}} D^{\alpha} f \|_{p, \reftri}^p 
		\end{align}
		for all $\alpha \in \mathbb{N}_0$ with $|\alpha| = m$. Inequality \cref{eq:hardy1-wsps-cont} now follows from Step 2 and \cref{eq:wsp-e-intersection-id}. By density, $\mathcal{H}_1$ maps $W_{\mathfrak{E}, r}^{s, p}(\reftri)$ boundedly into $W_{\mathfrak{E}, r}^{s, p}(\reftet)$ for all $p \in (1,\infty)$ and $s \in [0, r]$.
		
		Now let $s > r$ and $f \in W^{s, p}_{\mathfrak{E}, r}(\reftri)$. Arguing analogously as in Step 3, we have $\mathcal{H}_1(f) \in W^{s, p}_{\mathfrak{E}, r}(\reftri)$ if $s - \frac{1}{p} \notin \mathbb{Z}$ with $\inorm{\mathfrak{E}, r}{\mathcal{H}_1(f)}{s, p, \reftri} \lesssim_{s, p}  \inorm{\mathfrak{E}, r}{f}{s, p, \reftri}$. Moreover, \cref{eq:proof:h1-weighted-2-cont} gives
		\begin{align*}
			\left\| \omega_2^{-\frac{1}{p}} \frac{\partial^{m-r-1} D^{r-1} \mathcal{H}_1(f)}{\partial x_1^{m-r-1}} \right\|_{p, \reftri} \lesssim_{p} \left\| \omega_2^{-\frac{1}{p}} \frac{\partial^{m-r-1} D^{r-1} f}{\partial x_1^{m-r-1}} \right\|_{p, \reftri}.
		\end{align*}
		Consequently, $\mathcal{H}_1(f) \in W^{s, p}_{\mathfrak{E}, r}(\reftri)$ and  \cref{eq:hardy1-wsps-cont} holds when $\mathfrak{E} = \{\gamma_1, \gamma_2\}$.		
	\end{proof}

	Finally, we state and prove various versions of Hardy's inequality.
	\begin{theorem}
		\label{lem:omega1-inv-mapping}
		Let $1 < p < \infty$ and $\emptyset \neq  \mathfrak{E} \subseteq \{\gamma_1,\gamma_2\}$. For $0 \leq s < \infty$ and $i \in \{1,2\}$ such that $\gamma_i \in \mathfrak{E}$, the mapping $f \mapsto \omega_i^{-1} f$ is bounded (i) $W^{s+1,p}(\reftri) \cap W^{1,p}_{\mathfrak{E}}(\reftri)$ to $W^{s,p}(\reftri)$, and (ii) $W^{s+1,p}_{\mathfrak{E}}(\reftri)$ to $W^{s,p}_{\mathfrak{E}}(\reftri)$, and there holds
		\begin{alignat}{3}
			\label{eq:omega1-inv-wsp-bound}
			\| \omega_i^{-1} f \|_{s, p, \reftri} &\lesssim_{s, p} \left\| \partial_{i}  f  \right\|_{s, p, \reftri} \qquad & &\forall f \in W^{s+1,p}(\reftri) \cap W_{\mathfrak{E}}^{1, p}(\reftri), \qquad & &\ \\
			\label{eq:omega1-inv-wsps-bound}
			\inorm{\mathfrak{E}}{\omega_i^{-1} f}{s, p, \reftri} &\lesssim_{s, p} \inorm{\mathfrak{E}}{ \partial_i f   }{s, p, \reftri} \qquad & &\forall f \in W_{\mathfrak{E}}^{s+1, p}(\reftri). \qquad & &
		\end{alignat}
		Additionally, for $0 \leq s < 1$ and $i \in \{1,2\}$, the mapping $f \mapsto \omega_i^{-s} f$ is bounded $W_{\gamma_i}^{s, p}(\reftri)$ to $L^p(\reftri)$, and there holds
		\begin{align}
			\label{eq:omega1-invs-wsp1-bound}
			\| \omega_i^{-s} f\|_{p, \reftri} &\lesssim_{s, p} \inorm{\gamma_i}{f}{s, p, \reftri} \qquad \forall f \in W_{\gamma_i}^{s, p}(\reftri).
		\end{align}
	\end{theorem}
	\begin{proof}
		Let $1 < p < \infty$ be given.
		
		\noindent \textbf{Step 1: Inequalities \cref{eq:omega1-inv-wsp-bound,eq:omega1-inv-wsps-bound} when $\mathfrak{E} \in \{  \{\gamma_1\}, \{\gamma_1,\gamma_2\} \}$.  } Thanks to the fundamental theorem of calculus, there holds
		\begin{align}
			\label{eq:proof:wsp1-equal-hardy-deriv}
			f(\bdd{x}) = \int_{0}^{x_1} (\partial_{1} f)(u, x_2) \d{u} = x_1 \mathcal{H}_1(\partial_1 f) \qquad  \forall \bdd{x} \in \reftri, \ \forall f \in C^{\infty}_{\gamma_1}(\reftri).
		\end{align}
		By density (\cref{lem:wsp1-smooth-density}), \cref{eq:proof:wsp1-equal-hardy-deriv} holds for a.e. $\bdd{x} \in \reftri$ for all $f \in W_{\mathfrak{E}}^{1, p}(\reftri)$. \Cref{lem:hardy1} and \cref{eq:proof:wsp1-equal-hardy-deriv} then show that the mapping $f \mapsto \omega_i^{-1} f$ is bounded (i) from $W^{s+1,p}(\reftri) \cap W^{1,p}_{\mathfrak{E}}(\reftri)$ to $W^{s,p}(\reftri)$, and (ii) from $W^{s+1,p}_{\mathfrak{E}}(\reftri)$ to $W^{s,p}_{\mathfrak{E}}(\reftri)$ provided that $\gamma_i \in \mathfrak{E}$. Inequalities \cref{eq:omega1-inv-wsp-bound,eq:omega1-inv-wsps-bound} now follow from \cref{eq:hardy1-wsps-cont,eq:proof:wsp1-equal-hardy-deriv}. 
		
		\noindent \textbf{Step 2:  Inequalities \cref{eq:omega1-inv-wsp-bound,eq:omega1-inv-wsps-bound} when $\mathfrak{E} = \{2\}$. } Note that $f \in W^{s+1, p}(\reftri) \cap W^{1, p}_{\gamma_2}(\reftri)$ if and only if $g := f \circ \mathfrak{F}_1 \in W^{s+1, p}(\reftri) \cap W^{1, p}_{\gamma_1}(\reftri)$ and $f \in W^{s+1}_{\gamma_2}(\reftri)$ if and only if $g \in W^{s+1}_{\gamma_1}(\reftri)$, where $\mathfrak{F}_1$ is defined in \cref{eq:proof:frakf1-def}.  Inequalities \cref{eq:omega1-inv-wsp-bound,eq:omega1-inv-wsps-bound} then follow from Step 1.
		
		\noindent \textbf{Step 3: Inequality \cref{eq:omega1-invs-wsp1-bound} with $i=1$. } Now let $0 \leq s < 1$. For $sp = 1$, \cref{eq:omega1-invs-wsp1-bound} follows immediately from the definition of the norm. In the case $sp < 1$, the proof of Theorem 1.4.4.4 in \cite{Grisvard85} gives
		\begin{align*}
			\| \omega_i^{-s} f\|_{p, \reftri} \leq \| \dist(\cdot, \partial \reftri)^{-s} f\|_{p, \reftri} \lesssim_{s, p} \|f\|_{s, p, \reftri} = \inorm{\gamma_1}{f}{s, p, \reftri} \qquad \forall f \in W^{s, p}_{\gamma_1}(\reftri).
		\end{align*}
		Finally, let $sp > 1$ and let $f \in W^{s, p}_{\gamma_1}(\reftri)$ be given. We denote by $\tilde{f} \in W^{s, p}(\reftri)$ any extension  of $f$ to $\mathbb{R}^2$ satisfying $\|\tilde{f}\|_{s, p, \mathbb{R}^2} \lesssim_{s, p} \|f\|_{s, p, \reftri}$ (see e.g. \cite{Devore93} or \cite[Theorem 8.4]{Leoni23}). Thanks to Theorem 6.79, inequality (6.58), and Remark 6.80 of \cite{Leoni23}, there holds
		\begin{align*}
			\int_{0}^{\infty} \int_{\mathbb{R}} \frac{|\tilde{f}(x_1, x_2) - \tilde{f}(0, x_2)|^p}{x_1^{sp}} \d{x_2} \d{x_1} \lesssim_{s, p} | \tilde{f} |_{s, p, \mathbb{R}_+ \times \mathbb{R}}^p. 
		\end{align*}
		Since $f|_{\gamma_1} = \tilde{f}|_{\gamma_1} = 0$, we obtain
		\begin{align*}
			\| \omega_1^{-s} f\|_{p, \reftri}  = \| \omega_1^{-s} (\tilde{f} - \tilde{f}(0, \cdot)) \|_{p, \reftri} \leq \|  \omega_1^{-s} (\tilde{f} - \tilde{f}(0, \cdot)) \|_{p, \mathbb{R}_+ \times \mathbb{R}} \lesssim_{s, p} |\tilde{f}|_{s, p, \mathbb{R}_+ \times \mathbb{R}},
		\end{align*}
		and so $\| \omega_1^{-s} f\|_{p, \reftri} \lesssim_{s, p} \|f\|_{s, p, \reftri}$, which completes the proof of \cref{eq:omega1-invs-wsp1-bound} for $i=1$. 
		
		\noindent \textbf{Step 4: Inequality \cref{eq:omega1-invs-wsp1-bound} with $i=2$. } This can be reduced to the case $i=1$ using similar arguments as in Step 2 using the mapping $\mathfrak{F}_1$ defined in \cref{eq:proof:frakf1-def}.
\end{proof}
	
	\begin{corollary}
		\label{cor:omega-inv-wspr-norm}
		Let $1 < p < \infty$, $0 \leq s < \infty$, $1 \leq i \leq 2$, and $r_1, r_2 \in \mathbb{N}_0$ with $r_i \geq 1$. Then, for all $f \in W^{s+1, p}_{\gamma_1, r_1}(\reftri) \cap W^{s+1, p}_{\gamma_2, r_2}(\reftri)$, there holds $\omega_i^{-1} f \in W^{s, p}_{\gamma_i, r_i - 1}(\reftri) \cap W^{s, p}_{\gamma_j, r_j}(\reftri)$ and
		\begin{subequations}
			\label{eq:omega1-inv-wspr-norm}
			\begin{alignat}{2}
				\label{eq:omega1-inv-wspr-norm-1}
				\inorm{\gamma_i, r_i-1}{\omega_i^{-1} f}{s, p, \reftri} &\lesssim_{s, p, r_1, r_2}  \inorm{\gamma_i, r_i-1}{\partial_i f}{s, p, \reftri} & &\leq \inorm{\gamma_i, r_i}{f}{s+1, p, \reftri}, \\
				\label{eq:omega1-inv-wspr-norm-2}
				\inorm{\gamma_j, r_j}{\omega_i^{-1} f}{s, p} &\lesssim_{s, p, r_1, r_2}  \inorm{\gamma_j, r_j}{\partial_i f}{s, p, \reftri} & &\leq \inorm{\gamma_j, r_j}{f}{s+1, p, \reftri},
			\end{alignat}
		\end{subequations}
		where $1 \leq j \leq 2$, $j \neq i$.
	\end{corollary}
	\begin{proof}
		Let $f \in W^{s+1, p}_{\gamma_1, r_1+1}(\reftri) \cap W^{s+1, p}_{\gamma_2, r_2}(\reftri)$ be given. By definition, there holds $\partial_1 f \in  W^{s, p}_{\gamma_1, r_1}(\reftri) \cap W^{s, p}_{\gamma_2, r_2}(\reftri)$ since $\unitvec{t}_{\gamma_2} = [0, 1]^T$. Thanks to identity \cref{eq:proof:wsp1-equal-hardy-deriv}, which was shown in the proof of \cref{lem:omega1-inv-mapping} to hold for all $f \in W^{1, p}_{\gamma_1}(\reftri)$, the result for $i = 1$ follows from \cref{lem:hardy1}. The case $i=2$ can be reduced to the case $i=1$ using similar arguments as in the proof of \cref{lem:omega1-inv-mapping} using the mapping $\mathfrak{F}_1$ \cref{eq:proof:frakf1-def}.	
	\end{proof}

	\section{Equivalent Boundary Norm}
	
	We begin with a result that states necessary and sufficient conditions for a function defined on two faces $\Gamma_i \cup \Gamma_j \subset \partial \reftet$ to belong to $W^{s, p}(\Gamma_i \cup \Gamma_j)$.
	\begin{lemma}
		\label{lem:wsp-two-face}
		Let $0 \leq s < 1$, $1 < p < \infty$, and $1 \leq i < j \leq 4$.  Then, $f \in L^p(\Gamma_i \cup \Gamma_j)$ satisfies $f \in W^{s, p}(\Gamma_i \cup \Gamma_j)$ if and only if
		\begin{enumerate}
			\item[(i)] $f_i \in W^{s, p}(\Gamma_i)$ and $f_i \in W^{s, p}(\Gamma_i)$;
			
			\item[(ii)] if $s > 1/p$, then $f_i|_{\gamma_{ij}} = f_j|_{\gamma_{ij}}$; and
			
			\item[(iii)] if $s = 1/p$, then $I_{ij}^p(f_i, f_j) < \infty$,
		\end{enumerate}
		where $\mathcal{I}_{ij}^p(\cdot,\cdot)$ is defined in \cref{eq:edge-integral-def}. Additionally, 
		\begin{align*}
			\|f\|_{s, p, \Gamma_i \cup \Gamma_j}^p \approx_{s, p} \vertiii{f}_{s,p,\Gamma_i \cup \Gamma_j}^p := \|f_i\|_{s, p, \Gamma_i}^p + \|f_j\|_{s, p, \Gamma_j}^p + \begin{cases}
				\mathcal{I}_{ij}^p(f_i, f_j) & \text{if } sp = 1, \\
				0 & \text{otherwise}.
			\end{cases}
		\end{align*}
	\end{lemma}	
	\begin{proof}
		Let $0 \leq s < 1$, $1 < p < \infty$, and $1 \leq i < j \leq 4$ be given.
		
		\noindent \textbf{Step 1: $f \in W^{s,p}(\Gamma_i \cup \Gamma_j)\implies \text{(i-iii)}$. } Assume first that $f \in W^{s, p}(\Gamma_i \cup \Gamma_j)$. Condition (i) follows from the definition of the norms, and in particular, $ \|f_i\|_{s, p, \Gamma_i}^p + \|f_j\|_{s, p, \Gamma_j}^p \leq \|f\|_{s, p, \Gamma_i \cup \Gamma_j}^p$. If $s > 1/p$, then the trace theorem shows that $f$ has a well-defined trace on $\gamma_{ij}$, and so (ii) holds. 
		
		We now show that condition (iii) is satisfied. There holds
		\begin{align*}
			|\bdd{x} - \bdd{y}|^2 \approx ([\bdd{F}_{ij}^{-1}(\bdd{x}) - \bdd{F}_{ji}^{-1}(\bdd{y})] \cdot \unitvec{e}_1)^2 + ([\bdd{F}_{ij}^{-1}(\bdd{x}) + \bdd{F}_{ji}^{-1}(\bdd{y})] \cdot \unitvec{e}_2)^2
		\end{align*}
		for all  $(\bdd{x}, \bdd{y}) \in  \Gamma_i \times \Gamma_j$, where $\bdd{F}_{ij} : \reftri \to \Gamma_i$ and $\bdd{F}_{ji} : \reftri \to \Gamma_j$ are defined in \cref{eq:faces-to-reftri-edge-mappings}, and so
		\begin{align}
			\label{eq:proof:ref-equivalence}
			\int_{\reftri} \int_{\reftri} \frac{ |f_i \circ \bdd{F}_{ij}(\bdd{x}) - f_j \circ \bdd{F}_{ji} (\bdd{y})|^p }{ |\bdd{x}_{r} - \bdd{y}|^{sp+2}} \d{\bdd{y}} \d{\bdd{x}} 
			&\approx_{s,p} 			
			\int_{\Gamma_i} \int_{\Gamma_j} \frac{ |f_i(\bdd{x}) - f_j(\bdd{y})|^p }{|\bdd{x}-\bdd{y}|^{sp+2}} \d{\bdd{y}} \d{\bdd{x}} \\
			&< \|f\|_{s, p, \Gamma_i \cup \Gamma_j}^p, \notag
		\end{align}
		where $\bdd{x}_{r} := (x_1, -x_2)$.
		Now let $s = 1/p$. Applying the triangle inequality in conjunction with the above inequality, we obtain
		\begin{align*}
			\int_{\reftri} \int_{\reftri} \frac{ |f_i \circ \bdd{F}_{ij}(\bdd{x}) - f_j \circ \bdd{F}_{ji} (\bdd{x})|^p }{ |\bdd{x}_{r} - \bdd{y}|^{3}} \d{\bdd{y}} \d{\bdd{x}} 
			&\lesssim_p \|f\|_{s, p, \Gamma_i \cup \Gamma_j}^p.
		\end{align*}
		Assume, for the moment, that the following holds.
		\begin{align}
			\label{eq:proof:double-int-lower-bound}
			A(\bdd{x}) := \int_{\reftri} \frac{x_2}{ |\bdd{x}_{r} - \bdd{y}|^{3}} \d{\bdd{y}} \gtrsim 1 \qquad \forall \bdd{x} \in T.
		\end{align}
		Then, we obtain the following bound for $I_{ij}^p(f_i, f_j)$ defined in \cref{eq:edge-integral-def}:
		\begin{align*}
			I_{ij}^p(f_i, f_j) \lesssim \int_{\reftri} \int_{\reftri} \frac{ |f_i \circ \bdd{F}_{ij}(\bdd{x}) - f_j \circ \bdd{F}_{ji} (\bdd{x})|^p }{ |\bdd{x}_{r} - \bdd{y}|^{3}} \d{\bdd{y}} \d{\bdd{x}} \lesssim_p \|f\|_{s, p, \Gamma}^p.
		\end{align*}
		Thus, $f$ satisfies (iii) and we have shown that for all $f \in W^{s, p}(\Gamma_i \cup \Gamma_j)$ there holds
		\begin{align}
			\label{eq:proof:boundary-norm-gtr}
			\|f\|_{s, p, \Gamma_i \cup \Gamma_j} \gtrsim_{s, p}  \vertiii{f}_{s, p, \Gamma_i \cup \Gamma_j}.
		\end{align}
		
		We now turn to the proof of \cref{eq:proof:double-int-lower-bound}. First note that $A$ is well-defined on the half plane $\mathbb{R}_+^2 = \{\bdd{x} \in \mathbb{R}^2 : x_2 > 0\}$ and is a continuous function with $A(\bdd{x}) > 0$ for $\bdd{x} \in \mathbb{R}_+^2$. Moreover, a tedious calculation (in Step 3) shows that
		\begin{align*}
			A(\bdd{x}) = \frac{x_1 (x_2^2 + (1-x_1)^2)^{\frac{1}{2}} + (1-x_1 + x_2) (x_1^2 + x_2^2)^{\frac{1}{2}} - x_2((1+x_2)^2 + x_1^2)^{\frac{1}{2}}}{x_1(1-x_1+x_2)},
		\end{align*}
		and so $\lim_{x_2 \to 0^{+}} A(\bdd{x}) = 2$ for all $\bdd{x} \in \mathbb{R}_{+}^2$. Inequality \cref{eq:proof:double-int-lower-bound} now follows.
		
		\noindent \textbf{Step 2: $\text{(i-iii)} \implies f \in W^{s,p}(\Gamma_i \cup \Gamma_j)$. } Now assume that $f \in L^p(\Gamma_i \cup \Gamma_j)$ satisfies (i-iii).  Thanks to the triangle inequality, there holds
		\begin{multline}
			\label{eq:proof:triangle-inequality-mixed-faces}
			\int_{\reftri} \int_{\reftri} \frac{ |f_i \circ \bdd{F}_{ij}(\bdd{x}) - f_j \circ \bdd{F}_{ji} (\bdd{y})|^p }{ |\bdd{x}_{r} - \bdd{y}|^{sp+2}} \d{\bdd{y}} \d{\bdd{x}} \\
			\lesssim_{s, p} \int_{\reftri} \int_{\reftri} \frac{  |f_i \circ \bdd{F}_{ij}(\bdd{x}) - f_j \circ \bdd{F}_{ji} (\bdd{x})|^p }{ |\bdd{x}_{r} - \bdd{y}|^{sp+2}} \d{\bdd{y}} \d{\bdd{x}} \\
			\int_{\reftri} \int_{\reftri} \frac{ |f_j \circ \bdd{F}_{ji}(\bdd{x}) - f_j \circ \bdd{F}_{ji} (\bdd{y})|^p }{ |\bdd{x}_{r} - \bdd{y}|^{sp+2}} \d{\bdd{y}} \d{\bdd{x}}.
		\end{multline}
		We appeal to \cref{eq:proof:ref-equivalence} to bound the second term:
		\begin{align*}
			\int_{\reftri} \int_{\reftri} \frac{ |f_j \circ \bdd{F}_{ji}(\bdd{x}) - f_j \circ \bdd{F}_{ji} (\bdd{y})|^p }{ |\bdd{x}_{r} - \bdd{y}|^{sp+2}} \d{\bdd{y}} \d{\bdd{x}} \lesssim_{s, p} \|f_j\|_{s, p, \Gamma_j}^{p}.
		\end{align*}
		For the first term, there holds
		\begin{align*}
			\int_{\reftri} \frac{\d{\bdd{y}}}{|\bdd{x}_{r} - \bdd{y}|^{sp+2}} \leq \int_{\mathbb{R}^2 \setminus B(\bdd{x}_{r}, x_2)} \frac{\d{\bdd{y}}}{|\bdd{x}_{r} - \bdd{y}|^{sp+2}} = 2\pi \int_{x_2}^{\infty} \frac{\d{\rho}}{\rho^{sp+1}} = \frac{2\pi}{sp} x_2^{-sp},
		\end{align*}
		for all $\bdd{x} \in \reftri$, and so
		\begin{align*}
			\int_{\reftri} \int_{\reftri} \frac{ |f_i \circ \bdd{F}_{ij}(\bdd{x}) - f_j \circ \bdd{F}_{ji} (\bdd{x})|^p }{ |\bdd{x}_{r} - \bdd{y}|^{sp+2}} \d{\bdd{y}} \d{\bdd{x}} \lesssim_{s, p} \int_{\reftri} \frac{ |f_i \circ \bdd{F}_{ij}(\bdd{x}) - f_j \circ \bdd{F}_{ji} (\bdd{x})|^p }{x_2^{sp}}  \d{\bdd{x}}.  
		\end{align*}
		Thanks to conditions (ii)-(iii), the function $g = f_i \circ \bdd{F}_{ij} - f_j \circ \bdd{F}_{ji}$ belongs to $W_{\gamma_2}^{s, p}(\reftri)$ and applying \cref{eq:omega1-inv-wsps-bound} and the triangle inequality gives
		\begin{align*}
			\int_{\reftri} \frac{ |f_i \circ \bdd{F}_{ij}(\bdd{x}) - f_j \circ \bdd{F}_{ji} (\bdd{x})|^p }{x_2^{sp}}  \d{\bdd{x}} \lesssim_{s, p} \vertiii{f}_{s, p, \Gamma_i \cup \Gamma_j}^p
		\end{align*}
		and so
		\begin{align*}
			\int_{\reftri} \int_{\reftri} \frac{ |f_i \circ \bdd{F}_{ij}(\bdd{x}) - f_j \circ \bdd{F}_{ji} (\bdd{y})|^p }{ |\bdd{x}_{r} - \bdd{y}|^{sp+2}} \d{\bdd{y}} \d{\bdd{x}} \lesssim_{s, p} \vertiii{f}_{s, p, \Gamma_i \cup \Gamma_j}^p.
		\end{align*}
		Then, the reverse inequality of \cref{eq:proof:boundary-norm-gtr} immediately follows from \cref{eq:proof:ref-equivalence}, and so $f \in W^{s, p}(\Gamma_i \cup \Gamma_j)$ and the result follows.

		\noindent \textbf{Step 3: Computing $A(\bdd{x})$. } First note that		
		\begin{align*}
			\int_{0}^{1-y_2} \frac{\d{y_1}}{[(y_1 - x_1)^2 + z^2]^{\frac{3}{2}}} = \frac{1}{z^2} \left( \frac{1-y_2-x_1}{((1-y_2-x_1)^2 + z^2)^{\frac{1}{2}}} + \frac{x_1}{(x_1^2 + z^2)^{\frac{1}{2}}} \right),
		\end{align*}
		and so
		\begin{align*}
			&\int_{T} \frac{\d{\bdd{y}}}{[(x_1 - y_1)^2 + (x_2 + y_2)^2]^{\frac{3}{2}}} \\
			&\qquad = \int_{0}^{1} \frac{1}{z^2} \left[ \frac{1-y_2-x_1}{((1-y_2-x_1)^2 + z^2)^{\frac{1}{2}}} + \frac{x_1}{(x_1^2 + z^2)^{\frac{1}{2}}} \right]_{z = x_2 + y_2} \d{y_2} \\
			&\qquad \leftstackrel{z = x_2 + y_2}{=} \int_{x_2}^{1+x_2} \frac{1}{z^2} \left( \frac{1+x_2-x_1-z}{((1 + x_2-x_1-z)^2 + z^2)^{\frac{1}{2}}} + \frac{x_1}{(x_1^2 + z^2)^{\frac{1}{2}}} \right) \d{z}.
		\end{align*}
		Now note that for any $a \in \mathbb{R}$, there holds
		\begin{align*}
			\int_{x_2}^{1+x_2} \frac{a-z}{z^2((a-z)^2 + z^2)^{\frac{1}{2}}} \d{z} = \frac{(x_2^2 + (a-x_2)^2)^{\frac{1}{2}}}{a x_2} - \frac{ ((1+x_2)^2 + (a-1-x_2)^2)^{\frac{1}{2}} }{a (1+x_2)},
		\end{align*}
		and
		\begin{align*}
			\int_{x_2}^{1+x_2} \frac{x_1}{z^2 (x_1^2 + z^2)^{\frac{1}{2}}} \d{z} = 	\frac{(x_1^2 + x_2^2)^{\frac{1}{2}}}{x_1 x_2} - \frac{(x_1^2 + (1 + x_2)^2)^{\frac{1}{2}}}{x_1 (1 + x_2)}.
		\end{align*}
		Consequently, there holds
		\begin{align*}
			A(\bdd{x}) &= \frac{(x_2^2 + (1-x_1)^2)^{\frac{1}{2}}}{(1-x_1+x_2)}  + \frac{(x_1^2 + x_2^2)^{\frac{1}{2}}}{x_1 } \\
			&\qquad \qquad -  x_2((1+x_2)^2 + x_1^2)^{\frac{1}{2}} \left( \frac{1}{(1+x_2-x_1) (1+x_2)} + \frac{1}{x_1 (1 + x_2)} \right) \\
			&= \frac{x_1 (x_2^2 + (1-x_1)^2)^{\frac{1}{2}} + (1-x_1 + x_2) (x_1^2 + x_2^2)^{\frac{1}{2}} - x_2((1+x_2)^2 + x_1^2)^{\frac{1}{2}}}{x_1(1-x_1+x_2)}.
		\end{align*}
	\end{proof}
	
	On noting that $f \in W^{s, p}(\partial \reftet)$ if and only if $f \in L^p(\partial K)$ and $f|_{\Gamma_i \cup \Gamma_j} \in W^{s, p}(\Gamma_i \cup \Gamma_j)$ for all $1 \leq i < j \leq 4$, the following result is an immediate consequence of \cref{lem:wsp-two-face}.
	\begin{corollary}
		\label{lem:equivalent-boundary-norm}
		Let $0 \leq s < 1$ and $1 < p < \infty$. Then, $f \in L^p(\partial \reftet)$ satisfies $f \in W^{s, p}(\partial \reftet)$ if and only if
		\begin{enumerate}
			\item[(i)] $f_i \in W^{s, p}(\Gamma_i)$ for $1 \leq i \leq 4$;
			
			\item[(ii)] if $s > 1/p$, then $f_i|_{\gamma_{ij}} = f_j|_{\gamma_{ij}}$ for all $1 \leq i < j \leq 4$; and
			
			\item[(iii)] if $s = 1/p$, then $I_{ij}^p(f_i, f_j) < \infty$ for all $1 \leq i < j \leq 4$.
		\end{enumerate}
		Additionally, \cref{eq:equivalent-boundary-norm} holds.
	\end{corollary}
	
	\section{Partial fractions decomposition}
	
	\begin{lemma}
		\label{lem:partial-frac-decomp}
		For all $\beta \in \mathbb{N}_0^3$ with $|\beta| \geq 2$, there holds
		\begin{align}
			\label{eq:partial-frac-decomp}
			\frac{1}{\omega_1^{\beta_1} \omega_2^{\beta_2} \omega_3^{\beta_3}} = \sum_{\substack{\alpha \in \mathbb{N}_0^3 \\ \alpha_j \leq \beta_j \\ |\alpha| \geq 2 }} \left( \frac{ c_{\alpha,1} }{\omega_1^{\alpha_1} \omega_2^{\alpha_2}} + \frac{ c_{\alpha,2} }{\omega_1^{\alpha_1} \omega_3^{\alpha_3}} + \frac{ c_{\alpha,3} }{\omega_2^{\alpha_2} \omega_3^{\alpha_3}} \right) \qquad \text{in } \reftri,
		\end{align}
		where $\{ c_{\alpha, j} \}$ are suitable positive constants.
	\end{lemma}
	\begin{proof}
		We proceed by induction on $|\beta|$. The case $|\beta| = 2$ is trivially true. Assume that \cref{eq:partial-frac-decomp} holds for all $\beta \in \mathbb{N}_0^3$ with $|\beta| = r \geq 2$. Let $\beta \in \mathbb{N}_0^3$ with $|\beta| = r+1$. If $\beta_j = 0$ for some $j \in \{1,2,3\}$, then \cref{eq:partial-frac-decomp} is trivially true, so assume that $\beta_j > 0$ for $1 \leq j \leq 3$. Then,
		\begin{align*}
			\frac{1}{\omega_1^{\beta_1} \omega_2^{\beta_2} \omega_3^{\beta_3}} &= \frac{1}{\omega_1 \omega_2 \omega_3} \cdot \frac{1}{\omega_1^{\beta_1-1} \omega_2^{\beta_2-1} \omega_3^{\beta_3-1}} \\
			&= \left( \frac{1}{\omega_1 \omega_2} + \frac{1}{\omega_1 \omega_3} + \frac{1}{\omega_2 \omega_3} \right) \frac{1}{\omega_1^{\beta_1-1} \omega_2^{\beta_2-1} \omega_3^{\beta_3-1}} \\
			&= \frac{1}{\omega_1^{\beta_1} \omega_2^{\beta_2} \omega_3^{\beta_3-1}} + \frac{1}{\omega_1^{\beta_1} \omega_2^{\beta_2-1} \omega_3^{\beta_3}} + \frac{1}{\omega_1^{\beta_1-1} \omega_2^{\beta_2} \omega_3^{\beta_3}}.
		\end{align*}
		By assumption, each of the three terms above is of the form \cref{eq:partial-frac-decomp}, which completes the proof.
	\end{proof}

\end{document}